\documentclass[letterpaper,12pt,oneside,reqno]{amsart}
\usepackage[T2A]{fontenc}%
\usepackage[utf8]{inputenc}%
\usepackage[english]{babel}%
\usepackage{amsmath,amssymb,amsthm,amsfonts}%
\usepackage{hyperref}%
\usepackage{graphicx}
\usepackage{enumerate}
\usepackage[mathscr]{euscript}
\usepackage{color,tikz}
\usepackage[DIV14]{typearea}
\usepackage[width=.9\textwidth]{caption}
\allowdisplaybreaks%
\numberwithin{equation}{section}%

\usepackage{tikz}
\usetikzlibrary{decorations.markings, arrows}

\newcommand{\Z}{\mathbb{Z}}
\renewcommand{\C}{\mathbb{C}}
\newcommand{\R}{\mathbb{R}}
\DeclareMathOperator{\E}{\mathbb{E}}
\renewcommand{\i}{\mathbf{i}}
\DeclareMathOperator{\Prob}{\mathsf{Prob}}

\newcommand{\al}{\alpha}

\newcommand{\la}{\lambda}
\newcommand{\La}{\Lambda}

\newcommand{\eps}{\varepsilon}

\newcommand{\Ub}{\mathscr{U}}
\newcommand{\Hb}{\mathscr{H}}
\DeclareMathOperator{\diag}{\mathrm{diag}}
\DeclareMathOperator{\Tr}{\mathrm{Trace}}
\DeclareMathOperator{\D}{\mathscr{D}}
\DeclareMathOperator{\T}{\mathscr{T}}
\DeclareMathOperator*{\Res}{\mathrm{Res}}
\DeclareMathOperator*{\argmin}{\mathrm{argmin}}
\newcommand{\de}{{\mathsf{e}}}
\newcommand{\gap}{\mathsf{gap}}

\newcommand{\lozvd}{\begin{tikzpicture}[scale=.25]
	\draw [thick] (0,0) -- (.5,0.866) -- (1,0) -- (.5,-0.866) -- (0,0) -- (.5,0.866);
	\draw [thick, densely dotted] (0,0) -- (1,0);
\end{tikzpicture}}
\newcommand{\lozrd}{\begin{tikzpicture}[scale=.25]
	\draw [thick] (0,0) -- (.5,0.866) -- (1.5,0.866) -- (1,0) -- (0,0) -- (.5,0.866);
	\draw [thick, densely dotted] (1,0) -- (.5,0.866);
\end{tikzpicture}}
\newcommand{\lozld}{\begin{tikzpicture}[scale=.25]
	\draw [thick] (0,0) -- (-.5,0.866) -- (.5,0.866) -- (1,0) -- (0,0) -- (-.5,0.866);
	\draw [thick, densely dotted] (0,0) -- (.5,0.866);
\end{tikzpicture}}
\newcommand{\lozv}{\begin{tikzpicture}[scale=.25]
	\draw [thick] (0,0) -- (.5,0.866) -- (1,0) -- (.5,-0.866) -- (0,0) -- (.5,0.866);
\end{tikzpicture}}

\newtheorem{proposition}{Proposition}[section]
\newtheorem{lemma}[proposition]{Lemma}

\newtheorem{theorem}[proposition]{Theorem}
\theoremstyle{definition}

\newtheorem{remark}[proposition]{Remark}

\begin{document}
\title[From representation theory to Macdonald processes]{Integrable probability:\\
From representation theory to Macdonald processes}

\author[A. Borodin]{Alexei Borodin}
\address{A. Borodin,
Department of Mathematics, 
Massachusetts Institute of Technology,
77 Massachusetts ave.,
Cambridge, MA 02139, USA\newline
Institute for Information Transmission Problems, Bolshoy Karetny per. 19, Moscow, 127994, Russia}
\email{borodin@math.mit.edu}

\author[L. Petrov]{Leonid Petrov}
\address{L. Petrov,
Department of Mathematics, Northeastern University, 360 Huntington ave., Boston, MA 02115, USA\newline
Institute for Information Transmission Problems, Bolshoy Karetny per. 19, Moscow, 127994, Russia}
\email{lenia.petrov@gmail.com}

\begin{abstract}
	These are lecture notes for a mini-course given at the 
	Cornell Probability Summer School
	in July 2013.
	Topics include lozenge tilings of polygons and 
	their representation theoretic interpretation, the $(q,t)$-deformation of those 
	leading to the Macdonald processes, nearest neighbor dynamics 
	on Macdonald processes, their limit to semi-discrete Brownian polymers, and 
	large time asymptotic analysis of polymer's partition function. 
\end{abstract}

\maketitle

\setcounter{tocdepth}{2}
\tableofcontents
\setcounter{tocdepth}{2}

\section{Introduction} 
\label{sec:introduction}

One way to describe the content of these lecture notes is to say that they give 
a proof of the following statement (up to certain technical details that can 
be looked up in suitable articles).           

\begin{theorem}\label{Theorem1.1}
	Let $B_1,B_2,\dots$ be independent standard Brownian 
	motions. Define
	$$
	Z_N^t=
	\int\limits_{0<s_1<\ldots<s_{N-1}<t}
	e^{B_1(s_1)+
	\big(B_2(s_2)-B_2(s_1)\big)
	+\ldots
	+\big(B_N(\tau)-B_N(s_{N-1})\big)}
	ds_1 \ldots ds_{N-1}.
	$$
	Then, for any $\varkappa>0$,
	$$
	\lim_{N\to\infty} \text{\bf{P}} \left\{\frac{\log Z_N^{\varkappa 
	N}-c_1(\varkappa) N}{c_2(\varkappa) N^{1/3} }\le u \right\}=F_2(u)
	$$
	with certain explicit $\varkappa$-dependent constants $c_1,c_2>0$, where 
	$F_2(\,\cdot\,)$ is the distribution function of the GUE Tracy-Widom 
	distribution.
\end{theorem}

The quantity $Z_N^t$ was introduced by O'Connell-Yor \cite{OConnellYor2001}, 
and it can be viewed as the partition function of a semi-discrete 
Brownian polymer (also sometimes referred to as the ``O'Connell-Yor polymer'').
The limit relation above shows that this polymer model belongs to 
the celebrated Kardar-Parisi-Zhang (KPZ) universality class, see 
Corwin \cite{CorwinKPZ} for details on the KPZ class and \S1.6 
of \cite{BorodinCorwin2011Macdonald} for 
more explanations, consequences, and references concerning the polymer 
interpretation. 

The exact value of $c_1(\varkappa)$ was conjectured by O'Connell-Yor \cite{OConnellYor2001} and 
proven by Moriarty-O'Connell \cite{MoriartyOConnell}, and the above limit theorem was proven 
by Borodin-Corwin \cite{BorodinCorwin2011Macdonald} for a restricted range of $\varkappa$ and by 
Borodin-Corwin-Ferrari \cite{BorodinCorwinFerrari2012} for all $\varkappa>0$. A nice physics-oriented 
explanation of $c_2(\varkappa)$ was given by Spohn \cite{Spohn2012}.

Although the most direct proof of this theorem would likely be quite a bit 
shorter than these notes, brevity was not our goal. Despite the probabilistic 
appearance of the statement, any of the known approaches to the proof would 
involve a substantial algebraic component, and the appearance of algebra at 
first seems at least slightly surprising. The goal of these lectures notes is 
to suggest the most logically straightforward path (in authors' opinion) that 
leads to the 
desired result, minimizing as much as possible the number of {\it ad hoc\/} 
steps one takes along the way. (For an interested reader we remark that a 
shorter proof of Theorem \ref{Theorem1.1} can be obtained via combining 
Corollary 4.2 of \cite{Oconnell2009_Toda}, Theorem 2 of 
\cite{BorodinCorwinRemenik}, and asymptotic analysis of 
\cite{BorodinCorwin2011Macdonald}.)

As we travel along our path (that naturally starts on the algebraic side --- in 
representation theory of unitary groups), we encounter other probabilistic 
models that are amenable to similar tools of analysis.
The approach that we develop has a number of other applications as well. 
It was so far used for (we refer the reader to the indicated references for further explanations)
\begin{itemize}
	\item asymptotics of the KPZ equation with a certain class of initial conditions \cite{BorodinCorwinFerrari2012};
	\item asymptotics of Log-Gamma fully discrete random directed polymers \cite{BorodinCorwinRemenik};
	\item asymptotics of $q$-TASEP and ASEP \cite{BorodinCorwinSasamoto2012}, \cite{FerrariVeto2013};
	\item analysis of new integrable (1+1)d interacting particle systems 
--- discrete time $q$-TASEPs of \cite{BorodinCorwin2013discrete}, $q$-PushASEP 
\cite{BorodinPetrov2013NN}, \cite{CorwinPetrov2013}, and $q$-Hahn TASEP 
\cite{Corwin2014qmunu};
	\item 
	establishing a law of large numbers for infinite random matrices over a finite field
	\cite{BufetovPetrov2014} (conjectured by Vershik and Kerov, see \cite{GorinKerovVershikFq2012});
	\item Gaussian Free Field asymptotics of the general beta Jacobi corners process \cite{BorodinGorin2013beta};
	\item developing spectral theory for the $q$-Boson particle system \cite{BorodinCorwinPetrovSasamoto2013} and other integrable particle systems
	\cite{BCPSprep};
	\item 
	asymptotics of probabilistic models 
	originating from representation theory of the 
	infinite-dimensional unitary group
	$U(\infty)$
	\cite{BorodinBufetov2012},
	\cite{BorodinBufetov2013},
	\cite{GorinBufetov2013free},
	\cite{BBO2013}.
\end{itemize}
The emerging 
domain of studying such models that enjoy the benefits of a rich algebraic 
structure behind, is sometimes called \emph{Integrable Probability}, and we 
refer the reader to the introduction of Borodin-Gorin \cite{BorodinGorinSPB12} for a 
brief discussion of the domain and of its name (the integrable nature of the 
semi-discrete polymer of Theorem \ref{Theorem1.1} was first established by 
O'Connell \cite{Oconnell2009_Toda}).
To a certain extent, the present 
text may be considered as a continuation of \cite{BorodinGorinSPB12}, but it can be 
also read independently. 

In contrast with \cite{BorodinGorinSPB12}, in our exposition below we do not shy away from the representation theoretic background and intuition that were essential in developing the subject. We also focus on proving a single theorem, rather than describing the variety of other related problems listed above, in order to discuss in depth the analytic difficulties arising in converting an algebraic formalism into analytic statements. These difficulties are related to the phenomenon of {\it intermittency} and popular yet highly non-rigorous and sometimes dangerous {\it replica trick} favoured by physicists, and one of our goals is to show how raising the amount of ``algebraization'' of the problem can be used to overcome them.

\medskip

The notes are organized as follows. 

In Section \ref{sec:lozenge_tilings_and_representation_theory}, we explain how 
lozenge tilings of a class of polygons on the 
triangular lattice can be interpreted via representation theory of the unitary 
groups, and how this leads to contour integral formulas for averages of various 
observables. 

In Section \ref{sec:asymptotics}, we show, in a specific example, how the steepest descent analysis 
of the obtained contour integrals yields meaningful probabilistic information 
about lozenge tilings. 

Section \ref{sec:markov_dynamics} describes an approach to constructing local Markov dynamics on 
lozenge tilings, and how (1+1)-dimensional interacting particle systems (like 
usual and long range Totally Asymmetric Simple Exclusion Processes (TASEPs)) 
arise as marginals of such dynamics. The approach we describe is relatively 
recent; it was developed in Borodin-Petrov \cite{BorodinPetrov2013NN} 
(an extension of the method will appear in \cite{BufetovPetrov2014}).

Section \ref{sec:_q_t_generalization} deals with a two-parameter (Macdonald, $(q,t)$-) generalization of 
the previous material. 

In Section \ref{sec:asymptotics_of_q_deformed_growth_models} we show how simple-minded asymptotics of the contour integrals in 
the $q$-deformation of lozenge tilings leads to semi-discrete Brownian 
polymers. The contour integrals describe the $q$-moments of the $q$-TASEP, an 
integrable deformation of the usual TASEP. 

Section \ref{sec:moments_for_q_whittaker_processes} explains difficulties which arise if one straightforwardly tries 
to describe the distribution of the polymer partition function using its 
moments. The latter 
come out naturally as limits of the $q$-TASEP's $q$-moments. 

In final Section \ref{sec:laplace_transforms} we demonstrate how those 
difficulties can be overcome 
through considering the Laplace transform of the polymer partition function
and its $q$-analog for the $q$-TASEP particle locations.

\subsection*{Acknowledgments}
These are notes for lectures delivered at the 2013 Cornell Probability
Summer School, and we would like to thank the organizers for the 
invitation and warm hospitality. 
We are also very grateful to Ivan Corwin and Vadim Gorin for numerous valuable 
comments, and we thank the anonymous referee for several helpful remarks. 
AB~was partially supported by the NSF
grant DMS-1056390.
LP~was partially supported by 
the RFBR-CNRS grants 10-01-93114 and 11-01-93105.


\section{Lozenge tilings and representation theory} 
\label{sec:lozenge_tilings_and_representation_theory}

We begin with a discussion of a 
well-known 
probabilistic 
model of 
randomly 
tiling a 
hexagon drawn on the triangular lattice, and explain its relation to 
representation theory of unitary groups. This relation
produces rather natural tools for analysis of uniformly random lozenge tilings of the 
hexagon.

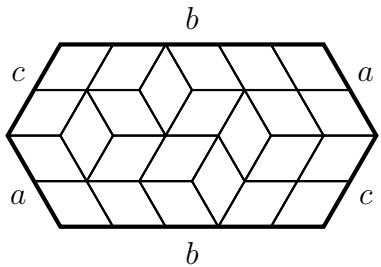
\begin{figure}[htbp]
	\begin{center}
\begin{tikzpicture}
			[scale=.7, ultra thick]
			\def\rt{0.866025}
			\node at (-.5-.3,\rt-.3) {$a$};
			\node at (-.5-.3,3*\rt+.3) {$c$};
			\node at (6-.2,3*\rt+.3) {$a$};
			\node at (6-.2,\rt-.3) {$c$};
			\node at (2.5,-.5) {$b$};
			\node at (2.5,4*\rt+.5) {$b$};
			\node at (2.5,4*\rt+.8) {{}};
			\draw 
			(0,0) --++ (5,0) --++ (1,2*\rt) 
			--++ (-1,2*\rt) --++ (-5,0)
			--++ (-1,-2*\rt) --++ (1,-2*\rt) -- cycle;
			\def\sml{1}
			\def\vlbx{-1/2}
			\def\vlby{1}
			\foreach \vl in {
			(\vlbx+3,\vlby*\rt),
			(\vlbx+3.5,\vlby*\rt+\rt),
			(\vlbx+.5,\vlby*\rt+\rt),
			(\vlbx+2,\vlby*\rt+2*\rt)
			}
			{
			\begin{scope}[shift=\vl,scale=\sml]
				\draw [thick] (0,0) -- (.5,\rt) -- (1,0) -- (.5,-\rt) -- cycle;
			\end{scope}}
			\def\llbx{0}
			\def\llby{0}
			\foreach \ll in {
			(\llbx,\llby),(\llbx+1,\llby),(\llbx+2,\llby),
			(\llbx-1/2,\llby+\rt),
			(\llbx+1,\llby+2*\rt),
			(\llbx+4,\llby+2*\rt),(\llbx+5,\llby+2*\rt),
			(\llbx+2.5,\llby+3*\rt),(\llbx+3.5,\llby+3*\rt),(\llbx+4.5,\llby+3*\rt)
			}
			{
			\begin{scope}[shift=\ll,scale=\sml]
				\draw [thick] (0,0) -- (-.5,\rt) -- (.5,\rt) -- (1,0) -- cycle;
			\end{scope}}
			\def\rlbx{0}
			\def\rlby{0}
			\foreach \rl in {
			(\rlbx+3,\rlby),(\rlbx+4,\rlby),
			(\rlbx+.5,\rlby+\rt),(\rlbx+1.5,\rlby+\rt),(\rlbx+3.5,\rlby+\rt),(\rlbx+4.5,\rlby+\rt),
			(\rlbx-1,\rlby+2*\rt),(\rlbx+2,\rlby+2*\rt),
			(\rlbx-.5,\rlby+3*\rt),(\rlbx+.5,\rlby+3*\rt)
			}
			{
			\begin{scope}[shift=\rl,scale=\sml]
				\draw [thick] 
				(0,0) -- (.5,\rt) -- (1.5,\rt) -- (1,0) -- cycle;
			\end{scope}}
		\end{tikzpicture}
		\end{center}
\caption{An example of a lozenge tiling of the hexagon with sides
$a,b,c,a,b,c$, where
$a=2$, $b=5$, and $c=2$.}
\label{fig:tiling_abc}
\end{figure}

\subsection{Lozenge tilings of a hexagon} 
\label{sub:lozenge_tilings_of_a_hexagon}

Consider the problem of tiling a hexagon 
with sides of length $a,b,c,a,b,c$ 
drawn on the triangular lattice 
by \emph{lozenges} that are defined
as pairs of triangles glued together (see Fig.~\ref{fig:tiling_abc}a).
Here $a$, $b$, and $c$ are any positive integers, and we assume that
the side of an elementary triangle has length $1$.
There are three different types of lozenges: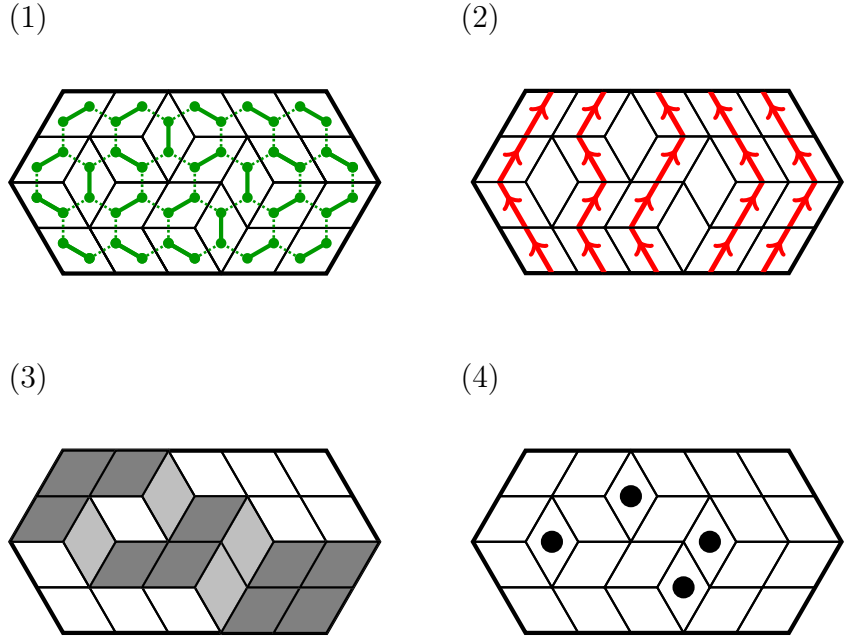
\begin{figure}[htbp]
	\begin{center}
		\begin{tabular}{ll}
		(1)&\hspace{20pt}(2)\\
		\begin{tikzpicture}
			[scale=.7, ultra thick]
			\def\rt{0.866025}
			\node at (2.5,4*\rt+.8) {{}};
			\draw 
			(0,0) --++ (5,0) --++ (1,2*\rt) 
			--++ (-1,2*\rt) --++ (-5,0)
			--++ (-1,-2*\rt) --++ (1,-2*\rt) -- cycle;
			\def\sml{1}
			\def\vlbx{-1/2}
			\def\vlby{1}
			\foreach \vl in {
			(\vlbx+3,\vlby*\rt),
			(\vlbx+3.5,\vlby*\rt+\rt),
			(\vlbx+.5,\vlby*\rt+\rt),
			(\vlbx+2,\vlby*\rt+2*\rt)
			}
			{
			\begin{scope}[shift=\vl,scale=\sml]
				\draw [thick] (0,0) -- (.5,\rt) -- (1,0) -- (.5,-\rt) -- cycle;
				\draw [color=green!60!black] (.5,-\rt/3) -- (.5,\rt/3);
				\draw [color=green!60!black, fill] (.5,-\rt/3) circle (.06);
				\draw [color=green!60!black, fill] (.5,\rt/3) circle (.06);
			\end{scope}}
			\def\llbx{0}
			\def\llby{0}
			\foreach \ll in {
			(\llbx,\llby),(\llbx+1,\llby),(\llbx+2,\llby),
			(\llbx-1/2,\llby+\rt),
			(\llbx+1,\llby+2*\rt),
			(\llbx+4,\llby+2*\rt),(\llbx+5,\llby+2*\rt)
			}
			{
			\begin{scope}[shift=\ll,scale=\sml]
				\draw [thick] (0,0) -- (-.5,\rt) -- (.5,\rt) -- (1,0) -- cycle;
				\draw [color=green!60!black] (.5,\rt/3) --++ (-.5,\rt/3);
				\draw [color=green!60!black, fill] (0,2*\rt/3) circle (.06);
				\draw [color=green!60!black, fill] (.5,\rt/3) circle (.06);
			\end{scope}}
			\foreach \ll in {
			(\llbx,\llby),(\llbx+1,\llby),(\llbx+2,\llby),
			(\llbx-1/2,\llby+\rt),
			(\llbx+1,\llby+2*\rt),
			(\llbx+4,\llby+2*\rt),(\llbx+5,\llby+2*\rt),
			(\llbx+2.5,\llby+3*\rt),(\llbx+3.5,\llby+3*\rt),(\llbx+4.5,\llby+3*\rt)
			}
			{
			\begin{scope}[shift=\ll,scale=\sml]
				\draw [thick] (0,0) -- (-.5,\rt) -- (.5,\rt) -- (1,0) -- cycle;
				\draw [color=green!60!black] (.5,\rt/3) --++ (-.5,\rt/3);
				\draw [color=green!60!black, fill] (0,2*\rt/3) circle (.06);
				\draw [color=green!60!black, fill] (.5,\rt/3) circle (.06);
			\end{scope}}
			\def\rlbx{0}
			\def\rlby{0}
			\foreach \rl in {
			(\rlbx+.5,\rlby+\rt),(\rlbx+1.5,\rlby+\rt),(\rlbx+3.5,\rlby+\rt),(\rlbx+4.5,\rlby+\rt),
			(\rlbx-1,\rlby+2*\rt),(\rlbx+2,\rlby+2*\rt)
			}
			{
			\begin{scope}[shift=\rl,scale=\sml]
				\draw [thick] 
				(0,0) -- (.5,\rt) -- (1.5,\rt) -- (1,0) -- cycle;
				\draw [color=green!60!black] (.5,\rt/3) --++ (.5,\rt/3);
				\draw [color=green!60!black, fill] (1,2*\rt/3) circle (.06);
				\draw [color=green!60!black, fill] (.5,\rt/3) circle (.06);
			\end{scope}}
			\foreach \rl in {
			(\rlbx-.5,\rlby+3*\rt),(\rlbx+.5,\rlby+3*\rt)
			}
			{
			\begin{scope}[shift=\rl,scale=\sml]
				\draw [thick] 
				(0,0) -- (.5,\rt) -- (1.5,\rt) -- (1,0) -- cycle;
				\draw [color=green!60!black] (.5,\rt/3) --++ (.5,\rt/3);
				\draw [color=green!60!black, fill] (1,2*\rt/3) circle (.06);
				\draw [color=green!60!black, fill] (.5,\rt/3) circle (.06);
			\end{scope}}
			\foreach \rl in {
			(\rlbx+3,\rlby),(\rlbx+4,\rlby)
			}
			{
			\begin{scope}[shift=\rl,scale=\sml]
				\draw [thick] 
				(0,0) -- (.5,\rt) -- (1.5,\rt) -- (1,0) -- cycle;
				\draw [color=green!60!black] (.5,\rt/3) --++ (.5,\rt/3);
				\draw [color=green!60!black, fill] (1,2*\rt/3) circle (.06);
				\draw [color=green!60!black, fill] (.5,\rt/3) circle (.06);
			\end{scope}}
			\foreach \ver in
			{
			(0,2*\rt/3),(1,2*\rt/3),(2,2*\rt/3),(4,2*\rt/3),(5,2*\rt/3),
			(-.5,2*\rt/3+\rt),(2-.5,2*\rt/3+\rt),(3-.5,2*\rt/3+\rt),(5-.5,2*\rt/3+\rt),(6-.5,2*\rt/3+\rt),
			(0,2*\rt/3+2*\rt),(1,2*\rt/3+2*\rt),(3,2*\rt/3+2*\rt),(4,2*\rt/3+2*\rt),(5,2*\rt/3+2*\rt)
			}
			{
			\begin{scope}[shift=\ver]
				\draw [color=green!60!black, densely dotted, line width=1]
				(0,0) -- (0,2*\rt/3);
			\end{scope}
			}
			\foreach \ri in
			{
			(.5,\rt/3),(1.5,\rt/3),(2.5,\rt/3),
			(0,\rt/3+\rt),(3,\rt/3+\rt),
			(0.5,\rt/3+2*\rt),(1.5,\rt/3+2*\rt),(3.5,\rt/3+2*\rt),(4.5,\rt/3+2*\rt),
			(2,\rt/3+3*\rt),(3,\rt/3+3*\rt),(4,\rt/3+3*\rt)
			}
			{
			\begin{scope}[shift=\ri]
				\draw [color=green!60!black, densely dotted, line width=1]
				(0,0) -- (.5,\rt/3);
			\end{scope}
			}
			\foreach \li in
			{
			(3.5,\rt/3),(4.5,\rt/3),
			(1,\rt/3+\rt),(2,\rt/3+\rt),(3,\rt/3+\rt),(4,\rt/3+\rt),(5,\rt/3+\rt),
			(.5,\rt/3+2*\rt),(2.5,\rt/3+2*\rt),(3.5,\rt/3+2*\rt),
			(1,\rt/3+3*\rt),(2,\rt/3+3*\rt)
			}
			{
			\begin{scope}[shift=\li]
				\draw [color=green!60!black, densely dotted, line width=1]
				(0,0) -- (-.5,\rt/3);
			\end{scope}
			}
		\end{tikzpicture}
		&\hspace{20pt}
		\begin{tikzpicture}
			[scale=.7, ultra thick]
			\def\rt{0.866025}
			\node at (2.5,4*\rt+.8) {{}};
			\draw 
			(0,0) --++ (5,0) --++ (1,2*\rt) 
			--++ (-1,2*\rt) --++ (-5,0)
			--++ (-1,-2*\rt) --++ (1,-2*\rt) -- cycle;
			\def\sml{1}
			\def\vlbx{-1/2}
			\def\vlby{1}
			\foreach \vl in {
			(\vlbx+3,\vlby*\rt),
			(\vlbx+3.5,\vlby*\rt+\rt),
			(\vlbx+.5,\vlby*\rt+\rt),
			(\vlbx+2,\vlby*\rt+2*\rt)
			}
			{
			\begin{scope}[shift=\vl,scale=\sml]
				\draw [thick] (0,0) -- (.5,\rt) -- (1,0) -- (.5,-\rt) -- cycle;
			\end{scope}}
			\def\llbx{0}
			\def\llby{0}
			\foreach \ll in {
			(\llbx,\llby),(\llbx+1,\llby),(\llbx+2,\llby),
			(\llbx-1/2,\llby+\rt),
			(\llbx+1,\llby+2*\rt),
			(\llbx+4,\llby+2*\rt),(\llbx+5,\llby+2*\rt),
			(\llbx+2.5,\llby+3*\rt),(\llbx+3.5,\llby+3*\rt),(\llbx+4.5,\llby+3*\rt)
			}
			{
			\begin{scope}[shift=\ll,scale=\sml]
				\draw[decoration={markings,
				    mark=at position .7 with {\arrow{>}}}, 
				    postaction={decorate},color = red, line width = 1.9] (0.5,0) --++ (-.5,\rt);
				\draw [thick] (0,0) -- (-.5,\rt) -- (.5,\rt) -- (1,0) -- cycle;
			\end{scope}}
			\def\rlbx{0}
			\def\rlby{0}
			\foreach \rl in {
			(\rlbx+3,\rlby),(\rlbx+4,\rlby),
			(\rlbx+.5,\rlby+\rt),(\rlbx+1.5,\rlby+\rt),(\rlbx+3.5,\rlby+\rt),(\rlbx+4.5,\rlby+\rt),
			(\rlbx-1,\rlby+2*\rt),(\rlbx+2,\rlby+2*\rt),
			(\rlbx-.5,\rlby+3*\rt),(\rlbx+.5,\rlby+3*\rt)
			}
			{
			\begin{scope}[shift=\rl,scale=\sml]
				\draw[decoration={markings,
				    mark=at position .7 with {\arrow{>}}}, 
				    postaction={decorate},color = red, line width = 1.9] (0.5,0) --++ (.5,\rt);
				\draw [thick] 
				(0,0) -- (.5,\rt) -- (1.5,\rt) -- (1,0) -- cycle;
			\end{scope}}
		\end{tikzpicture}
		\\
		\end{tabular}

		\vspace{18pt}

		\begin{tabular}{ll}
		(3)&\hspace{20pt}(4)\rule{0pt}{19pt}\\
		\begin{tikzpicture}
			[scale=.7, ultra thick]
			\def\rt{0.866025}
			\node at (2.5,4*\rt+.8) {{}};
			\draw 
			(0,0) --++ (5,0) --++ (1,2*\rt) 
			--++ (-1,2*\rt) --++ (-5,0)
			--++ (-1,-2*\rt) --++ (1,-2*\rt) -- cycle;
			\def\sml{1}
			\def\vlbx{-1/2}
			\def\vlby{1}
			\foreach \vl in {
			(\vlbx+3,\vlby*\rt),
			(\vlbx+3.5,\vlby*\rt+\rt),
			(\vlbx+.5,\vlby*\rt+\rt),
			(\vlbx+2,\vlby*\rt+2*\rt)
			}
			{
			\begin{scope}[shift=\vl,scale=\sml]
				\draw [thick, color=gray!50!white,fill] (0,0) -- (.5,\rt) -- (1,0) -- (.5,-\rt) -- cycle;
				\draw [thick] (0,0) -- (.5,\rt) -- (1,0) -- (.5,-\rt) -- cycle;
			\end{scope}}
			\def\llbx{0}
			\def\llby{0}
			\foreach \ll in {
			(\llbx,\llby),(\llbx+1,\llby),(\llbx+2,\llby),
			(\llbx-1/2,\llby+\rt),
			(\llbx+1,\llby+2*\rt),
			(\llbx+4,\llby+2*\rt),(\llbx+5,\llby+2*\rt),
			(\llbx+2.5,\llby+3*\rt),(\llbx+3.5,\llby+3*\rt),(\llbx+4.5,\llby+3*\rt)
			}
			{
			\begin{scope}[shift=\ll,scale=\sml]
				\draw [thick] (0,0) -- (-.5,\rt) -- (.5,\rt) -- (1,0) -- cycle;
			\end{scope}}
			\def\rlbx{0}
			\def\rlby{0}
			\foreach \rl in {
			(\rlbx+3,\rlby),(\rlbx+4,\rlby),
			(\rlbx+.5,\rlby+\rt),(\rlbx+1.5,\rlby+\rt),(\rlbx+3.5,\rlby+\rt),(\rlbx+4.5,\rlby+\rt),
			(\rlbx-1,\rlby+2*\rt),(\rlbx+2,\rlby+2*\rt),
			(\rlbx-.5,\rlby+3*\rt),(\rlbx+.5,\rlby+3*\rt)
			}
			{
			\begin{scope}[shift=\rl,scale=\sml]
				\draw [thick, fill, color=gray] 
				(0,0) -- (.5,\rt) -- (1.5,\rt) -- (1,0) -- cycle;
				\draw [thick] 
				(0,0) -- (.5,\rt) -- (1.5,\rt) -- (1,0) -- cycle;
			\end{scope}}
		\end{tikzpicture}
		&\hspace{20pt}
		\begin{tikzpicture}
			[scale=.7, ultra thick]
			\def\rt{0.866025}
			\node at (2.5,4*\rt+.8) {{}};
			\draw 
			(0,0) --++ (5,0) --++ (1,2*\rt) 
			--++ (-1,2*\rt) --++ (-5,0)
			--++ (-1,-2*\rt) --++ (1,-2*\rt) -- cycle;
			\def\sml{1}
			\def\vlbx{-1/2}
			\def\vlby{1}
			\foreach \vl in {
			(\vlbx+3,\vlby*\rt),
			(\vlbx+3.5,\vlby*\rt+\rt),
			(\vlbx+.5,\vlby*\rt+\rt),
			(\vlbx+2,\vlby*\rt+2*\rt)
			}
			{
			\begin{scope}[shift=\vl,scale=\sml]
				\draw [thick] (0,0) -- (.5,\rt) -- (1,0) -- (.5,-\rt) -- cycle;
				\draw [fill] (.5,0) circle (.17);
			\end{scope}}
			\def\llbx{0}
			\def\llby{0}
			\foreach \ll in {
			(\llbx,\llby),(\llbx+1,\llby),(\llbx+2,\llby),
			(\llbx-1/2,\llby+\rt),
			(\llbx+1,\llby+2*\rt),
			(\llbx+4,\llby+2*\rt),(\llbx+5,\llby+2*\rt),
			(\llbx+2.5,\llby+3*\rt),(\llbx+3.5,\llby+3*\rt),(\llbx+4.5,\llby+3*\rt)
			}
			{
			\begin{scope}[shift=\ll,scale=\sml]
				\draw [thick] (0,0) -- (-.5,\rt) -- (.5,\rt) -- (1,0) -- cycle;
			\end{scope}}
			\def\rlbx{0}
			\def\rlby{0}
			\foreach \rl in {
			(\rlbx+3,\rlby),(\rlbx+4,\rlby),
			(\rlbx+.5,\rlby+\rt),(\rlbx+1.5,\rlby+\rt),(\rlbx+3.5,\rlby+\rt),(\rlbx+4.5,\rlby+\rt),
			(\rlbx-1,\rlby+2*\rt),(\rlbx+2,\rlby+2*\rt),
			(\rlbx-.5,\rlby+3*\rt),(\rlbx+.5,\rlby+3*\rt)
			}
			{
			\begin{scope}[shift=\rl,scale=\sml]
				\draw [thick]
				(0,0) -- (.5,\rt) -- (1.5,\rt) -- (1,0) -- cycle;
			\end{scope}}
		\end{tikzpicture}
		\end{tabular}
	\end{center}
	\caption{Various interpretations of a lozenge tiling.}
	\label{fig:tiling_abc1}
\end{figure}
$\lozvd$, $\lozld$, and $\lozrd$.
Such tilings 
(that are in a bijective correspondence 
with 
\emph{boxed plane
partitions})
can be interpreted in a variety of ways (see Fig.~\ref{fig:tiling_abc1}):
\begin{enumerate}[(1)]
	\item As \emph{dimers} (or \emph{perfect matchings}) on the dual hexagonal lattice.
	\item As sets of nonintersecting Bernoulli paths following 
	lozenges of two types 
	(\begin{tikzpicture}[scale=.25]
		\def\rt{0.866025}
		\draw[color = red, line width = 1.9] (0.5,0) --++ (-.5,\rt);
		\draw [thick] (0,0) -- (-.5,\rt) -- (.5,\rt) -- (1,0) -- cycle;
	\end{tikzpicture}
	and 
	\begin{tikzpicture}[scale=.25]
		\def\rt{0.866025}
		\draw[color = red, line width = 1.9] (0.5,0) --++ (.5,\rt);
		\draw [thick] 
		(0,0) -- (.5,\rt) -- (1.5,\rt) -- (1,0) -- cycle;
	\end{tikzpicture}
	with prescribed beginnings and ends.
	\item As stepped surfaces made of $1\times 1\times 1$ cubes.
	\item As interlacing configurations of 
	lattice points --- centers of lozenges of one of the types,
	say, $\lozv$, as on Fig. \ref{fig:tiling_abc1}.
	Such configurations must have a prescribed number of points in each \emph{horizontal}
	section that may depend on the section.
\end{enumerate}
Our first goal is to match this combinatorial object with a basic representation
theoretic one.


\subsection{Representations of unitary groups} 
\label{sub:representations_of_unitary_groups}

Denote by $\Ub(N)$ the (compact Lie)
group of all the unitary matrices\footnote{$U^{*}=U^{-1}$, where $U^*$ is the conjugate transpose.}
of size $N$.
A (finite-dimensional) \emph{representation} of $\Ub(N)$ is a continuous map
\begin{align*}
	T\colon \Ub(N)\to GL(m,\C)
\end{align*}
(for some $m=1,2,\ldots$)
which respects the 
group structure: 
$T(UV)=T(U)T(V)$, $U,V\in \Ub(N)$.
A representation is called \emph{irreducible}
if it has no \emph{invariant} subspaces $E\subset \C^{m}$
($E\ne 0$ or $\C^{m}$), i.e.,
such that $T(\Ub(N))E\subset E$.

The classification of irreducible representations
of $\Ub(N)$
(equivalently, of $GL(N,\C)$
by analytic continuation --- ``unitary trick'' of H. Weyl)
is one of high points of the classical representation
theory.
It is due to Hermann Weyl in mid-1920's.
In order to 
understand how it works, let us restrict $T$ 
to the abelian subgroup of diagonal 
unitary matrices
\begin{align*}
	\Hb_N:=\left\{\diag(e^{\i\varphi_1},\ldots,e^{\i\varphi_N})\colon \varphi_1,\ldots,\varphi_N\in\R\right\}.
\end{align*}
Any commuting family of (diagonalizable\footnote{Any
finite-dimensional representation of a finite or compact
group, in particular, $U(N)$, is unitary in a suitable basis
(e.g., see \cite{Zhelobenko-book-eng}), hence all our 
matrices are diagonalizable.}) matrices can be simultaneously 
diagonalized. In particular, this is true for $T(\Hb_N)$. 
Hence, for $1\le j\le m$,
\begin{align*}
	\C^{m}=\bigoplus_{j=1}^{m}\C v_j,\qquad
	T\left(\diag(e^{\i\varphi_1},\ldots,e^{\i\varphi_N})\right)v_j=
	t_j(e^{\i\varphi_1},\ldots,e^{\i\varphi_N})\cdot v_j,
\end{align*}
where
each $t_j$ is a continuous homomorphism
$\Hb_N\to\C$. Any such homomorphism
has the form
\begin{align*}
	t(z_1,\ldots,z_N)=z_1^{k_1}\ldots z_N^{k_N},\qquad k_1,\ldots,k_N\in\Z.
\end{align*}
Each $N$-tuple $(k_1,\ldots,k_N)\in\Z^{N}$
for $t=t_j$, $1\le j\le m$, is called a \emph{weight} 
of the representation $T$. 
There is a total of $m$ weights (which is the 
dimension of the representation).

\begin{theorem}[H. Weyl, see, e.g., \cite{Weyl1946}, \cite{Zhelobenko-book-eng}]
	Irreducible representations of $\Ub(N)$
	are in one-to-one correspondence with ordered $N$-tuples 
	$\la=(\la_1\ge \ldots\ge\la_N)\in\Z^{N}$.

	The correspondence is established by requiring that 
	$\la$ is the unique highest 
	(in lexicographic order) 
	weight of the 
	corresponding representation. Then the generating function 
	of all weights of this representation $T_\la$ can be written as
	\begin{align}
		&\nonumber
		\Tr\Big(T_\la\big(\diag(z_1,\ldots,z_N)\big)
		\Big)\\&\hspace{40pt}=\sum_{\text{$(k_1,\ldots,k_N)$ weight of $T_\la$}}
		z_1^{k_1}\ldots z_N^{k_N}
		=\frac{\det\big[z_i^{N+\la_j-j}\big]_{i,j=1}^{N}}
		{\det\big[z_i^{N-j}\big]_{i,j=1}^{N}}.
		\label{Schur}
	\end{align}
\end{theorem}
Note that the denominator in \eqref{Schur} is the Vandermonde determinant 
which evaluates to 
\begin{align*}
	\det\big[z_i^{N-j}\big]_{i,j=1}^{N}=\prod_{1\le i<j\le N}(z_i-z_j).
\end{align*}
The numerator in \eqref{Schur} is necessarily divisible 
by the denominator because of its skew-symmetry with respect to $z_i\leftrightarrow z_j$,
and thus the ratio is a finite linear combination 
of the monomials of the form $z_1^{k_1}\ldots z_N^{k_N}$,
$k_1,\ldots,k_N\in\Z$
(i.e., an element of $\C[z_1^{\pm1},\ldots,z_N^{\pm1}]^{{S}(N)}$).

The polynomials $\Tr(T_\la)$
are called \emph{Schur polynomials},
after Issai Schur, who used them in 
the representation theory of the symmetric group in his thesis around 1900.
However,
one of the earliest appearances of them
dates back to 
Cauchy
\cite{Cauchy1815} and
Jacobi
\cite{Jacobi1841}, 
over 100 years before Weyl's work.
The Schur polynomials are denoted by $s_\la=\Tr(T_\la)$.
Schur polynomials are, generally speaking, 
symmetric 
homogeneous Laurent polynomials in $N$ variables.

While the ratio of determinants 
formula
\eqref{Schur} 
is beautiful and concise 
(it is a special case of \emph{Weyl's character formula}
which works for any compact semi-simple Lie group),
is does not describe the set of weights 
explicitly. To do that, we need
an elementary
\begin{lemma}\label{lemma:Schur_branching}
	For any $\la=(\la_1\ge \ldots\ge\la_N)\in\Z^{N}$, 
	\begin{align}\label{Schur_branching}
		s_\la(z_1,\ldots,z_N)=
		\sum_{\mu\prec\la}s_\mu(z_1,\ldots,z_{N-1})z_N^{|\la|-|\mu|},
	\end{align}
	where the sum is taken over $\mu=(\mu_1,\ldots,\mu_{N-1})\in\Z^{N-1}$,
	the notation $\mu\prec\la$
	means the \emph{interlacing}
	\begin{align*}
		\la_N\le \mu_{N-1}\le \la_{N-1}\le \ldots\le \la_2\le \mu_1\le \la_1,
	\end{align*}
	and $|\la|=\sum_{j=1}^{N}\la_j$, $|\mu|=\sum_{j=1}^{N-1}\mu_j$.
\end{lemma}
\begin{proof}
	Clear the denominators in \eqref{Schur_branching} and compare 
	coefficients by each of the monomials $z_1^{k_1}\ldots z_N^{k_N}$.
\end{proof}
Applying this lemma $N$ times, we see that 
the weights are in one-to-one correspondence with interlacing
triangular arrays of integers
\begin{align}
	\begin{array}{c}
	\begin{tikzpicture}[scale=1]
		\def\h{0.8}
		\def\x{2}
		\node at (-5,5) {$\la_N$};
		\node at (-3,5) {$\la_{N-1}$};
		\node at (0-\x/2,5) {$\ldots\ldots\ldots\ldots$};
		\node at (3-\x,5) {$\la_{2}$};
		\node at (5-\x,5) {$\la_{1}$};
		\node at (-4,5-\h) {$\mu_{N-1}$};
		\node at (-2,5-\h) {$\mu_{N-2}$};
		\node at (0-\x/2+.1,5-\h) {$\ldots$};
		\node at (2-\x,5-\h) {$\mu_{2}$};
		\node at (4-\x,5-\h) {$\mu_{1}$};
		\node at (-3,5-2*\h) {$\nu_{N-2}$};
		\node at (3-\x,5-2*\h) {$\nu_{1}$};
		\foreach \LePoint in {(4.5-\x,5-\h/2),(2.5-\x,5-\h/2),(-3.5,5-\h/2),
		(-2.5,5-3*\h/2),(3.5-\x,5-3*\h/2),(2.5-\x,5-5*\h/2),(1.7-\x,5-7*\h/2)} {
    		\node [rotate=45] at \LePoint {$\le$};
    	};
    	\foreach \GePoint in {(3.5-\x,5-\h/2),(-4.5,5-\h/2),(-2.5,5-\h/2),
    	(-3.5,5-3*\h/2),(-2.5,5-5*\h/2),(.5,5-3*\h/2),(-1.7,5-7*\h/2)} {
    		\node [rotate=135] at \GePoint {$\ge$};
    	};
    	\node at (0-\x/2,5-3*\h) {$\ldots\ldots\ldots\ldots$};
    	\node at (0-\x/2+.08,5-4*\h) {$\omega_1$};
	\end{tikzpicture}
	\end{array}
	\label{GT_scheme}
\end{align}
Such arrays are called \emph{Gelfand--Tsetlin schemes/patterns},
and they will play a prominent role in what follows.

Observe that if we shift all leftmost entries
of a Gelfand--Tsetlin scheme of depth (or height) $N$ by $0$,
the second to left ones by $1$, etc., then in the end
we obtain a similar array where some of the inequalities 
become strict:
\begin{align}
	\begin{array}{c}
	\begin{tikzpicture}[scale=1]
		\def\h{0.85}
		\def\x{.7}
		\def\y{.15}
		\node at (-5,5) {$\la_N$};
		\node at (-3+\y,5) {$\la_{N-1}+1$};
		\node at (0-\x/2,5) {$\ldots\ldots\ldots$};
		\node at (3-\x+0*\y,5) {$\la_{2}+N-2$};
		\node at (5-\x+9*\y,5) {$\la_{1}+N-1$};
		\node at (-4,5-\h) {$\mu_{N-1}$};
		\node at (-2+\y,5-\h) {$\mu_{N-2}+1$};
		\node at (0-\x/2,5-\h) {$\ldots$};
		\node at (2-\x+0*\y,5-\h) {$\mu_{2}+N-3$};
		\node at (4-\x+5*\y,5-\h) {$\mu_{1}+N-2$};
		\node at (-3,5-2*\h) {$\nu_{N-2}$};
		\node at (3-\x+2*\y,5-2*\h) {$\nu_{1}+N-3$};
		\foreach \LePoint in {(4.5-\x+5*\y,5-\h/2),(2.5-\x,5-\h/2),(-3.5,5-\h/2),
		(-2.5,5-3*\h/2),(3.5+3*\y-\x,5-3*\h/2),(2.5-\x,5-5*\h/2),(1.7-\x,5-7*\h/2)} {
    		\node [rotate=45] at \LePoint {$<$};
    	};
    	\foreach \GePoint in {(3.5+4*\y-\x,5-\h/2),(-4.5,5-\h/2),(-2.5,5-\h/2),
    	(-3.5,5-3*\h/2),(1.5+2*\y,5-3*\h/2),(-2.5,5-5*\h/2),(-1.7,5-7*\h/2)} {
    		\node [rotate=135] at \GePoint {$\ge$};
    	};
    	\node at (0-\x/2,5-3*\h) {$\ldots\ldots\ldots\ldots\ldots\ldots$};
    	\node at (0-\x/2+.08,5-4*\h) {$\omega_1$};
	\end{tikzpicture}
	\end{array}
	\label{shifted_GT_scheme}
\end{align}

\begin{proposition}\label{prop:U_a+c}
	Lozenge tilings of the hexagon 
	with sides $a,b,c,a,b,c$ (Fig.~\ref{fig:tiling_abc})
	are in one-to-one correspondence 
	with weights of the irreducible representation
	of the unitary group $\Ub(a+b)$
	with highest weight 
	\begin{align}
		\la=\big(\underbrace{b,b,\ldots,b}_{\text{$a$ times}},\underbrace{0,0,\ldots,0}_{\text{$c$ times}}\big).
		\label{la_for_hexagon}
	\end{align}
\end{proposition}
\begin{proof}
	By picture, see Fig.~\ref{fig:tilings_weights}.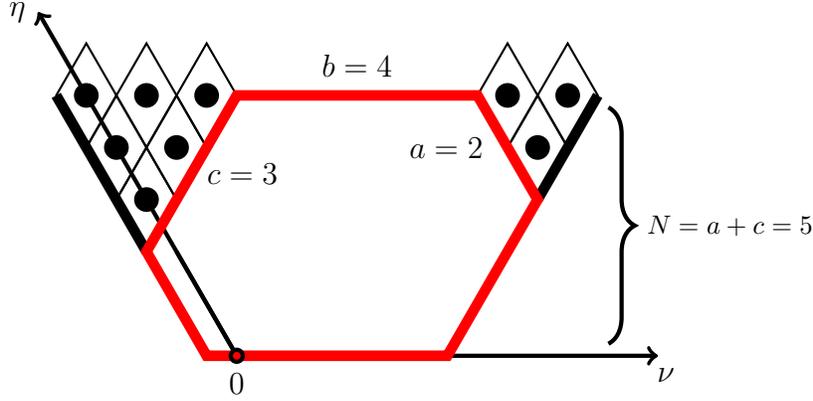
\begin{figure}[htbp]
	\begin{center}
		\begin{tikzpicture}
			[scale=.8, ultra thick,
			block/.style ={rectangle, thick, draw=black, text width=13em,
			align=center, rounded corners, minimum height=2em}]
			\def\rt{0.866025}
			\draw[->] (-1,2*\rt)--(0,0) -- (7,0) node[below, xshift=3] {$\nu$};
			\draw[->] (0,0) -- (-3.3,6.6*\rt) node[left] {$\eta$};
			\def\sml{1}
			\def\vlbx{-\sml/2}
			\def\vlby{0}
			\foreach \vl in {
			(\vlbx-1.5,\vlby*\rt+3*\rt),
			(\vlbx-2,\vlby*\rt+4*\rt),(\vlbx-1,\vlby*\rt+4*\rt),
			(\vlbx-2.5,\vlby*\rt+5*\rt),(\vlbx-1.5,\vlby*\rt+5*\rt),(\vlbx-.5,\vlby*\rt+5*\rt),
			(\vlbx+4.5,\vlby*\rt+5*\rt),(\vlbx+5.5,\vlby*\rt+5*\rt),
			(\vlbx+5,\vlby*\rt+4*\rt)}
			{
			\begin{scope}[shift=\vl,scale=\sml]
				\draw [thick] (0,0) -- (.5,\rt) -- (1,0) -- (.5,-\rt) -- cycle;
				\draw [fill] (.5,0) circle (.17);
			\end{scope}}
			\draw [line width=4] (-1.5,2*\rt) --++(-1.5,3*\rt);
			\draw [line width=4] (5,3*\rt) --++(1,2*\rt);
			\draw [color=red, line width = 4] (-.5,0)
			--++ (-1,2*\rt) --++(1.5,3*\rt) --++(4,0) --++(1,-2*\rt) 
			--++(-1.5,-3*\rt)--++(-4,0)--cycle;
			\draw (0,0) circle (.1) node [below, yshift=-2] {0};
			\node [above] at (2,5*\rt+.1) {$b=4$};
			\node [right] at (-.7,3.5*\rt) {$c=3$};
			\node [left] at (4.3,4*\rt) {$a=2$};
			\usetikzlibrary{decorations.pathreplacing}
			\draw [decorate,decoration={brace,amplitude=10pt,mirror,raise=4pt},yshift=0pt]
			(6,0.2) -- (6,5*\rt-.2) node [black,midway,xshift=50] {\footnotesize
			$N=a+c=5$};
		\end{tikzpicture}
	\end{center}
	\caption{On the correspondence between lozenge
	tilings of a hexagon and weights.}
	\label{fig:tilings_weights}
\end{figure}
	If we coordinatize by taking the centers of the 
	vertical lozenges in the coordinate
	system on the picture, then we read off the shifted Gelfand--Tsetlin
	schemes \eqref{shifted_GT_scheme}.
\end{proof}
The total number of weights of $T_\la$ (or,
equivalently, the dimension of the representation) 
was denoted by $m=m(\la)$ above, and it is given by
\begin{proposition}\label{prop:s_la_1111}
	For any $\la=(\la_1\ge \ldots\ge\la_N)\in\Z^{N}$,
	\begin{align}\label{dim_T_lambda}
		\dim T_\la=
		s_\la\big(\underbrace{1,1,\ldots,1}_N\big)=
		\prod_{1\le i<j\le N}\frac{(\la_i-i)-(\la_j-j)}{j-i}.
	\end{align}
\end{proposition}
This is a special case of \emph{Weyl's dimension formula} 
(which again works for any compact semi-simple Lie group).
\begin{proof}
	Follows from 
	\eqref{Schur} either directly (by the L'H\^opital's rule),
	or through the substitution
	$(z_1,\ldots,z_N)=(1,q,q^{2},\ldots,q^{N-1})$
	and the limit
	$q\to 1$.
\end{proof}


\subsection{Distribution of lozenges on a horizontal slice} 
\label{sub:distribution_of_lozenges_on_a_horizontal_slice}

Consider the uniform probability measure 
$\Prob_{a,b,c}$
on the space of all lozenge tilings of the hexagon
with sides $a,b,c,a,b,c$ (see Fig.~\ref{fig:tiling_abc}).
The normalizing factor in the measure
$\Prob_{a,b,c}$ (the so-called \emph{partition function})
is given in \eqref{dim_T_lambda} with $\la$ as in \eqref{la_for_hexagon}.

\begin{remark}
	For tilings 
	of the hexagon with sides of length $a,b,c,a,b,c$
	the 
	partition function was first computed 
	in a nicer product form
	\begin{align*}
		s_{
		(b,b,\ldots,b,0,0,\ldots,0)}
		=\prod_{i=1}^{a}
		\prod_{j=1}^{b}
		\prod_{k=1}^{c}\frac{i+j+k-1}{i+j+k-2}
	\end{align*}
	by MacMahon \cite{MacMahonBook}.
\end{remark}

Let us focus on 
what happens when we consider one horizontal
slice of our uniformly random lozenge tiling.
As above, we coordinatize it by 
locations of centers of the vertical lozenges.
It is easiest to assume that the slice
is close enough to one of the two horizontal boundaries,
say, the lowest one (see Fig.~\ref{fig:slice_h}).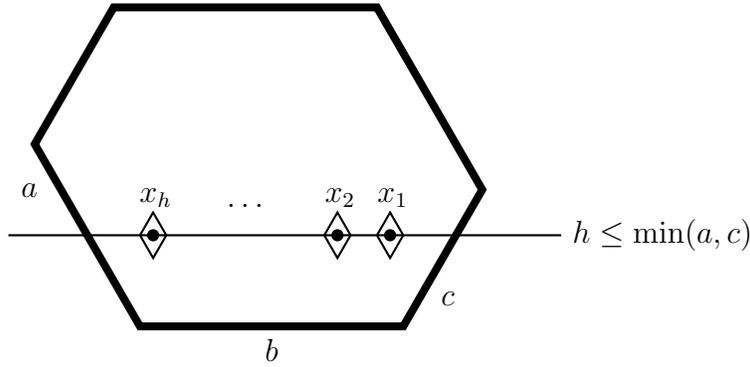
\begin{figure}[htbp]
	\begin{center}
		\begin{tikzpicture}
			[scale=.7, ultra thick]
			\def\rt{0.866025}
			\draw [line width=3]
			(0,0) --++ (5,0) --++ (1.5,3*\rt)
			--++(-2,4*\rt)--++(-5,0)--++(-1.5,-3*\rt)
			--++(2,-4*\rt)--cycle;
			\node at (-2.1,3*\rt) {$a$};
			\node at (2.5,-.45) {$b$};
			\node at (5.85,.6*\rt) {$c$};

			\draw[thick] (-2.5,2*\rt) --++(10.5,0)
			node [right] {$h\le \min(a,c)$};

			\def\sml{.5}
			\def\vlbx{-1/2}
			\def\vlby{1}
			\foreach \vl in {
			(\vlbx+4,\vlby*\rt+\rt),
			(\vlbx+.5,\vlby*\rt+\rt),
			(\vlbx+5,\vlby*\rt+\rt)
			}
			{
			\begin{scope}[shift=\vl,scale=\sml]
				\draw [thick] (0,0) -- (.5,\rt) -- (1,0) -- (.5,-\rt) -- cycle;
				\draw [fill] (.5,0) circle (.15);
			\end{scope}}
			\node [above] at (.3,2*\rt+.3) {$x_h$};
			\node [above] at (4.8,2*\rt+.3) {$x_1$};
			\node [above] at (3.8,2*\rt+.3) {$x_2$};
			\node [above] at (2,2*\rt+.3) {$\ldots$};
		\end{tikzpicture}
	\end{center}
	\caption{Horizontal slice of a lozenge tiling.}
	\label{fig:slice_h}
\end{figure}
\begin{proposition}\label{prop:horizontal_slice_distribution}
	For any $h$, $0\le h\le \min(a,c)$, the 
	distribution of lozenges on the horizontal slice
	at height $h$
	has the form
	\begin{align}\label{Prob_abc}
		\Prob_{a,b,c,h}
		\{x_1,\ldots,x_h\}=
		\mathrm{const}(a,b,c,h)\cdot 
		\prod_{1\le i<j\le h}(x_i-x_j)^{2}\prod_{i=1}^{h}w_{a,b,c,h}(x_i),
	\end{align}
	where 
	\begin{align}\label{w_abch}
		w_{a,b,c,h}(x)=\frac{(b+c-1-x)!\,(a-h+x)!}{x!\,(b+h-1-x)!}.
	\end{align}
\end{proposition}
\begin{remark}
	Probability measures of the form \eqref{Prob_abc}
	with arbitrary positive weight function 
	$w(\cdot)$ are known as \emph{orthogonal polynomial ensembles}
	as they are closely related to the orthogonal polynomials
	with weight $w$. 
	The measure \eqref{Prob_abc} itself is often
	referred to as the Hahn orthogonal polynomial ensemble,
	as this particular weight $w$ \eqref{w_abch} 
	corresponds to the classical Hahn orthogonal polynomials.
	See, e.g., \cite{Konig2005} and references therein
	for details.
\end{remark}
\begin{proof}
	We can cut the enumeration problem into two
	that look like those on Fig.~\ref{fig:fill_fill},
	and then multiply the results.
	Each of the two problems (compute the number of 
	tilings of the corresponding region with fixed top row) is solved by the 
	dimension formula \eqref{dim_T_lambda}.
	\begin{figure}[htbp]
	\begin{center}
		\begin{tikzpicture}
			[scale=.65, ultra thick]
			\def\rt{0.866025}
			\draw [line width=2]
			(-1,2*\rt)
			--(0,0)--(5,0)--++(1,2*\rt)
			--++(-1,0)
			--++(-.25,-\rt/2)
			--++(-.25,\rt/2)
			--++(-.5,0)
			--++(-.25,-\rt/2)
			--++(-.25,\rt/2)
			--++(-3,0)
			--++(-.25,-\rt/2)
			--++(-.25,\rt/2)
			--++(-1,0)--cycle;

			\node at (2.3,\rt) {\large {fill}};
			\def\sml{.5}
			\def\vlbx{-1/2}
			\def\vlby{1}
			\foreach \vl in {
			(\vlbx+4,\vlby*\rt+\rt),
			(\vlbx+.5,\vlby*\rt+\rt),
			(\vlbx+5,\vlby*\rt+\rt)
			}
			{
			\begin{scope}[shift=\vl,scale=\sml]
				\draw [thick] (0,0) -- (.5,\rt) -- (1,0) -- (.5,-\rt) -- cycle;
				\draw [fill] (.5,0) circle (.12);
			\end{scope}}
			\node [above] at (.3,2*\rt+.3) {$x_h$};
			\node [above] at (4.8,2*\rt+.3) {$x_1$};
			\node [above] at (3.8,2*\rt+.3) {$x_2$};
			\node [above] at (2,2*\rt+.3) {$\ldots$};
			\begin{scope}[shift={(9.5,3*\rt)}]
				\foreach \vl in {
				(\vlbx+1.5,\vlby*\rt+\rt),
				(\vlbx+.5,\vlby*\rt+\rt),
				(\vlbx+5,\vlby*\rt+\rt),
				(\vlbx+6.5,\vlby*\rt+\rt),(\vlbx+7,\vlby*\rt+\rt),(\vlbx+7.5,\vlby*\rt+\rt),(\vlbx+8,\vlby*\rt+\rt),
				(\vlbx+6.75,\vlby*\rt+\rt/2),(\vlbx+7.25,\vlby*\rt+\rt/2),(\vlbx+7.75,\vlby*\rt+\rt/2),
				(\vlbx+7,\vlby*\rt),(\vlbx+7.5,\vlby*\rt),
				(\vlbx+7.25,\vlby*\rt-\rt/2),
				(\vlbx-1,\vlby*\rt+\rt),(\vlbx-1.5,\vlby*\rt+\rt),(\vlbx-1.25,\vlby*\rt+\rt/2)
				}
				{
				\begin{scope}[shift=\vl,scale=\sml]
					\draw [thick] (0,0) -- (.5,\rt) -- (1,0) -- (.5,-\rt) -- cycle;
					\draw [fill] (.5,0) circle (.12);
				\end{scope}}
				\node [above] at (.3,2*\rt+.3) {$x_1$};
				\node [above] at (4.8,2*\rt+.3) {$x_h$};
				\node [above] at (1.3,2*\rt+.3) {$x_2$};
				\node [above] at (3,2*\rt+.3) {$\ldots$};
				\draw [line width=2]
				(0,2*\rt)
				--++(.25,-\rt/2)
				--++(.25,\rt/2)
				--++(.5,0)
				--++(.25,-\rt/2)
				--++(.25,\rt/2)
				--++(3,0)
				--++(.25,-\rt/2)
				--++(.25,\rt/2)
				--++(1,0)
				--++(1,-2*\rt)
				--++(-1.5,-3*\rt)
				--++(-5,0)
				--++(-2,4*\rt)
				--++(.5,\rt)--cycle;
				\node at (3,-.5*\rt) {\large {fill}};
			\end{scope}
		\end{tikzpicture}
	\end{center}
	\caption{Computing the distribution of 
	lozenges on a horizontal slice (cf.~Fig.~\ref{fig:slice_h})
	amounts to two enumeration
	problems.}
	\label{fig:fill_fill}
\end{figure}
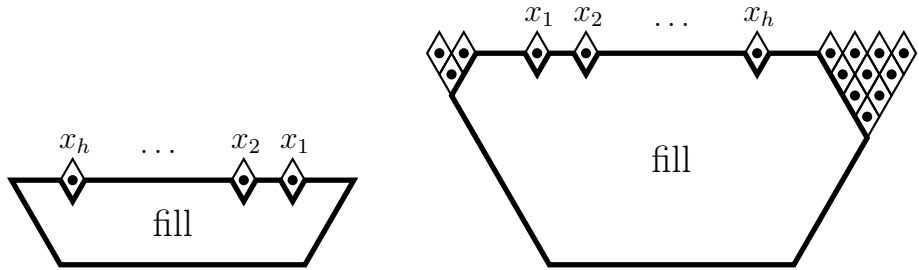
\end{proof}

The computation of Proposition \ref{prop:horizontal_slice_distribution}
already allows to see asymptotic transitions for a fixed $h$ (for example, 
$h=1$).
We can rewrite \eqref{w_abch} as
\begin{align}\label{w_abch_rewrite}
	&
	w_{a,b,c,h}(x)=\frac{(b+c-1)!\,(a-h)!}{(b+h-1)!}\times\\&\hspace{50pt}\times
	\left[\frac{1}{x!}\cdot
	\frac{(a-h+1)\ldots(a-h+x)\cdot(b+h-1)\ldots(b+h-x)}
	{(c+b-1)\ldots(c+b-x)}
	\right].
	\nonumber
\end{align}
One can consider the following limit regimes:
\begin{enumerate}[(1)]
	\item 
	If $a,b,c\to\infty$
	so that $ab/c\to t$,
	the first term just contributes to a constant,
	while the second one
	converges to $t^x/x!$.
	\item 
	In a similar way, if we keep $a$ finite and send 
	$b,c\to\infty$
	in such a way that $b/(b+c)\to\xi$, $0<\xi<1$, then we 
	see that the relevant part of \eqref{w_abch_rewrite}
	converges to 
	\begin{align*}
		\frac{(a-h+1)(a-h+2)\ldots(a-h+x)}{x!}\xi^x.
	\end{align*}
	\item
	A slightly more complicated limit transition 
	would be to take $a,b,c\to\infty$ so that the triple ratio $a:b:c$
	has a finite limit. Then Stirling's formula shows that 
	after proper shifting and scaling of $(x_1,\ldots,x_h)$, 
	which would not affect the factor 
	$\prod_{1\le i<j\le h}(x_i-x_j)^{2}$,
	the nontrivial part of $w_{a,b,c,h}(x)$
	converges to a Gaussian weight $e^{-x^2/2}$, $x\in\R$.
\end{enumerate}

There is also a representation-theoretic way 
to view these results. Restricting to a fixed 
horizontal slice means that we 
only care about the restriction of our representation
of $\Ub(a+c)$ (recall Proposition \ref{prop:U_a+c})
to the subgroup $\Ub(h)$
of matrices which are nontrivial (i.e., different from $\mathrm{Id}$) 
only in the top-left 
$h\times h$ corner. In terms of weights, we only care about
powers of $z_1,\ldots,z_h$ and substitute
$z_{h+1}=\ldots=z_N=1$.
This is equivalent to saying that the probability 
\eqref{Prob_abc}
of 
$(x_1,\ldots,x_h)=(\mu_1+h-1,\mu_2+h-2,\ldots,\mu_h)$
is the normalized coefficient of $s_\mu(z_1,\ldots,z_h)$ in the 
identity
\begin{align}\label{Prob_as_coeffs}
	\frac{s_\la(z_1,\ldots,z_{h},1,\ldots,1)}{s_\la(1,\ldots,1)}
	=\sum_{\mu_1\ge \ldots\ge \mu_h}
	\Prob_{a,b,c,h}\{\mu\}\cdot \frac{s_\mu(z_1,\ldots,z_h)}{s_\mu(1,\ldots,1)},
\end{align}
where $\la$ is as in \eqref{la_for_hexagon}, and we 
are dividing by the normalizing constants to have the ``$\Prob$''
coefficients add up to 1.
This corresponds to looking at relative dimensions 
of \emph{isotypical} subspaces 
(i.e., 
those that transform according to fixed irreducible representation)
rather than the actual ones.

The first two of the above three limit transitions turn 
\eqref{Prob_as_coeffs} into
\begin{align}
	&&\frac{ab}{c}\to t:\hspace{40pt}&
	\prod_{i=1}^{h}e^{t(z_i-1)}&=&
	\sum_{\mu_1\ge \ldots\ge \mu_h}\Prob_{t,h}\{\mu\}\cdot \frac{s_\mu(z_1,\ldots,z_h)}{s_\mu(1,\ldots,1)};
	\label{Prob_t}\\
	&&\frac{b}{b+c}\to\xi:\hspace{40pt}&
	\prod_{i=1}^{h}
	\frac{(1-\xi)^{a}}{(1-\xi z_i)^{a}}&=&
	\sum_{\mu_1\ge \ldots\ge \mu_h}\Prob_{a,\xi,h}\{\mu\}\cdot \frac{s_\mu(z_1,\ldots,z_h)}{s_\mu(1,\ldots,1)}.
	\label{Prob_axi}
\end{align}
In fact,
these two
limits $ab/c\to t$ and $b/(b+c)\to \xi$ correspond
to certain infinite-dimensional representations of the
infinite-dimensional unitary group $\Ub(\infty)=\varinjlim\Ub(N)$.

The third (Gaussian) limit is the eigenvalue projection 
of the matrix Fourier transform identity
\begin{align*}
	\int\limits_{\mathrm{Herm}(N)}e^{\i \Tr(AB)}\mathrm{M}(dB)=e^{-\Tr(A^{2})/2},
\end{align*}
where $\mathrm{Herm}(N)$ is the space of $N\times N$ Hermitian 
($H^*=H$)
matrices, $A\in\mathrm{Herm}(N)$, and $\mathrm{M}(dB)$ is the 
probability measure on $\mathrm{Herm}(N)$
with the density $e^{-\Tr(B^2)/2}dB$ also known as 
the \emph{Gaussian Unitary Ensemble} (or \emph{GUE}).
This limit is a special case of 
the so-called quasi-classical limit in representation
theory that degenerates ``large'' representations to 
probability measures on (co-adjoint orbits of)
the associated Lie algebra, 
e.g., see
\cite{TyanShanski1973},
\cite{Heckmann1982},
\cite{GuilleminSternberg1982}. 
For a broad survey of quantization ideas in 
representation theory see e.g. 
\cite{Kirillov_orbit}
and references therein.


\subsection{Scalar operators and observables} 
\label{sub:averaging_observables}

We are interested in more complex limit transitions than 
those
in \S \ref{sub:distribution_of_lozenges_on_a_horizontal_slice},
and for accessing them the following 
representation theoretic thinking is useful. 
Our probability weights \eqref{Prob_abc} 
arise as relative dimensions of the
isotypical
subspaces in the representation space for $\Ub(N)$. 
Moreover, these subspaces are blocks of identical irreducibles 
with respect to the action of the smaller group 
$\Ub(h)$.

\subsubsection{Locally scalar operators} 
\label{ssub:scalar_operators}

The problem of decomposing 
a representation on irreducible components is often
referred to as the \emph{problem of} 
(\emph{noncommutative}) \emph{harmonic analysis}.
It can be viewed as a noncommutative Fourier
transform
--- an analogue of 
the classical Fourier transform when $\R$ acts by shifts on $L^2(\R)$.
The ``best'' way to solve such a problem would be to find operators in the representation
space which project to a given 
isotypical component. For the classical Fourier transform, these operators
have the form
\begin{align*}
	f\mapsto \int_{-\infty}^{+\infty}e^{-\i xp}f(x)dx.
\end{align*}

For the action of the symmetric group, such operators
are known under the name \emph{Young symmetrizers},
they date back to the earliest days of representation
theory. However, even if one can construct such operators,
they are quite complicated. 
The ``next best'' thing is to find 
operators which are scalar in each irreducible representation 
(the projection operators 
take value $1$ in one irreducible representation, and $0$ in all other irreducible representations).
By a simple Schur's lemma, such operators are
exactly those that commute with the action of the group.


\subsubsection{Dilation operators} 
\label{ssub:dilations}

Observe that $\Ub(h)$ has a nontrivial center --- scalar matrices 
of the form $e^{\i\varphi}\cdot\mathbf{1}$, $\varphi\in\R$. Their action on elements
$\diag(z_1,\ldots,z_h)\in\Hb_h$ amounts to multiplying each $z_j$
by $e^{\i\varphi}$, and their action
on a vector of weight $(k_1,\ldots,k_h)\in\Z^{h}$ is the multiplication by
$e^{\i\varphi |k|}=e^{\i\varphi (k_1+\ldots+k_h)}$ (see \S \ref{sub:representations_of_unitary_groups}).
Hence, 
using the homogeneity of the Schur polynomials we see that
on an irreducible representation of $\Ub(h)$
with highest weight $\mu=(\mu_1\ge \ldots\ge \mu_h)$
such an operator acts as the scalar operator
$e^{\i\varphi|\mu|}\cdot\mathbf{1}$.

Let us now apply such an operator, viewed simply as 
the dilation operator $(\D_{\varphi}f)(z_1,\ldots,z_h)=f(e^{\i\varphi}z_1,\ldots,e^{\i\varphi}z_h)$,
to the decomposition identity \eqref{Prob_t} defining $\Prob_{t,h}\{\mu\}$:
\begin{align*}
	\prod_{j=1}^{h}e^{t(e^{\i\varphi}z_j-1)}=&
	\sum_{\mu_1\ge \ldots\ge \mu_h}e^{\i\varphi|\mu|}
	\Prob_{t,h}\{\mu\}\cdot \frac{s_\mu(z_1,\ldots,z_h)}{s_\mu(1,\ldots,1)}
\end{align*}
(clearly, $\D_\varphi s_\mu=e^{\i\varphi|\mu|}s_\mu$).
Setting $z_1=\ldots=z_h=1$ above, we get
\begin{align*}
	e^{ht(e^{\i\varphi}-1)}=&
	\sum_{\mu_1\ge \ldots\ge \mu_h}e^{\i\varphi|\mu|}
	\Prob_{t,h}\{\mu\}.
\end{align*}
This immediately tells us (``for free''), that $|\mu|$
has the Poisson distribution with parameter $ht$
because the left-hand side is the characteristic function
of that distribution.


\subsubsection{Quadratic Casimir--Laplace operator} 
\label{ssub:quadratic_casimir_laplace_operator}

Going further, the first nontrivial example of an operator
which commutes with the action of
$\Ub(h)$ is the so-called \emph{quadratic
Casimir--Laplace operator} $\mathscr{C}_2$. Its action on functions on $\Hb_h$
is given by 
\begin{align*}&
	(\mathscr{C}_2f)(z_1,\ldots,z_h)\\&\hspace{40pt}={\prod\limits_{1\le i<j\le h}(z_i-z_j)^{-1}}
	\sum_{r=1}^{h}\left(z_r \frac{\partial}{\partial z_r}\right)^{2}
	\prod_{1\le i<j\le h}(z_i-z_j) \; f(z_1,\ldots,z_h).
\end{align*}
Such operators
exist for all semi-simple Lie groups and are one of the 
basic representation-theoretic objects.
Also, 
\begin{align}\label{generator_of_Dyson_CM}
	\mathscr{C}_2-\sum_{j=1}^{h-1}j^2
\end{align}
is the (projection to eigenvalues of the) generator
of the Brownian motion on $\Ub(h)$.
In other words, \eqref{generator_of_Dyson_CM}
is the generator of the circular Dyson Brownian motion
\cite{Dyson1962_III}, \cite{dyson1962brownian}.
See also \S \ref{sub:dyson_brownian_motion} below
for a related Markov dynamics.

It is immediate to see (using the ratio of determinants formula \eqref{Schur} and the fact that 
$(z \frac{\partial}{\partial z})z^k=kz^k$)
that the action of the quadratic Casimir--Laplace operators on the Schur polynomials
is diagonal, and
\begin{align*}
	\mathscr{C}_2 s_\mu=\sum_{i=1}^{h}
	(\mu_i+h-i)^{2}s_\mu.
\end{align*}

We could now proceed with the application of $\mathscr{C}_2$ to \eqref{Prob_t}.
However, let us first note that the dilation operators
$\D_\varphi$ can be written in a form rather similar to $\mathscr{C}_2$:
\begin{align*}&
	(\D_\varphi f)(z_1,\ldots,z_h)\\&\hspace{15pt}={\prod\limits_{1\le i<j\le h}(z_i-z_j)^{-1}}\,
	e^{\i\varphi\left(\sum\limits_{r=1}^{h}z_r \frac{\partial}{\partial z_r}-\frac{h(h-1)}{2}\right)}
	\prod_{1\le i<j\le h}(z_i-z_j)\; f(z_1,\ldots,z_h).
\end{align*}
Indeed, the desired eigenrelation $\D_\varphi s_\mu=e^{\i\varphi|\mu|}s_\mu$
again follows from \eqref{Schur} and the fact that 
$e^{\i\varphi z \frac{\partial}{\partial z}}z^{k}=e^{\i\varphi k}z^{k}$.


\subsubsection{A $q$-deformation} 
\label{ssub:general_recipe}

Let us now note that we have a general recipe on our hands of
constructing operators which have Schur functions as their eigenfunctions. Namely, 
for any operator of the form
\begin{align}\label{Df_general}
	(\D f)(z_1,\ldots,z_h)={\prod\limits_{1\le i<j\le h}(z_i-z_j)^{-1}}
	\Big(\sum_{r=1}^{h}\D^{(z_r)}\Big)
	\prod_{1\le i<j\le h}(z_i-z_j)\; f(z_1,\ldots,z_h)
\end{align}
with 
\begin{align*}
	\D^{(z)}z^{k}=d_kz^k,
\end{align*}
we have
\begin{align}\label{Df_smu}
	\D s_\mu=\Big(\sum_{i=1}^{h}d_{\mu_i+h-i}\Big)s_\mu.
\end{align}

For example, we can take
\begin{align*}
	\D^{(z)}=\T_{q,z},\qquad
	(\T_{q,z}f)(z)=f(qz),\qquad d_k=q^{k},
\end{align*}
where $q\in\C$ is a parameter.

Then using \eqref{Df_general} we obtain a $q$-difference operator that can be rewritten in the form
\begin{align}\label{D1_q}
	\D^{(1)}=\sum_{i=1}^{h}\prod_{j\ne i}\frac{qz_i-z_j}{z_i-z_j}\T_{q,z_i},
\end{align}
and \eqref{Df_smu} gives
\begin{align}\label{D1_smu}
	\D s_\mu=\Big(\sum_{i=1}^{h}q^{\mu_i+h-i}\Big)s_\mu.
\end{align}



\subsection{Contour integrals and the density function} 
\label{sub:density_function_of_prob_t{mu}_}

We can now apply $\D^{(1)}$ \eqref{D1_q} to the identity \eqref{Prob_t} defining the measure $\Prob_{t,h}$.
This gives
\begin{align}\label{D1_identity}&
	e^{t\sum\limits_{r=1}^{h}(z_r-1)}
	\sum_{i=1}^{h}\prod_{j\ne i}\frac{qz_i-z_j}{z_i-z_j}\,
	e^{t(qz_i-z_i)}\\&\hspace{80pt}=\nonumber
	\sum_{\mu_1\ge \ldots\ge \mu_h}
	\Big(\sum_{r=1}^{h}q^{\mu_r+h-r}\Big)
	\Prob_{t,h}\{\mu\}\cdot \frac{s_\mu(z_1,\ldots,z_h)}{s_\mu(1,\ldots,1)}.
\end{align}
As before, we would like to substitute $z_1=\ldots=z_h=1$ in the above 
identity. However, observe that the left-hand side 
is not well-suited for that. A standard trick helps --- 
the left-hand side can be rewritten as a simple contour integral:
\begin{lemma}\label{lemma:contour_integral}
	Let $f\colon\C\to\C$ be a holomorphic function. Then
	\begin{align*}
		\sum_{r=1}^{h}\prod_{j\ne r}\frac{qz_r-z_j}{z_r-z_j}\cdot
		\frac{f(qz_r)}{f(z_r)}=
		\frac{1}{2\pi\i}\oint\limits_{\{w\}} \prod_{j=1}^{h}\frac{qw-z_j}{w-z_j}\frac{1}{qw-w}
		\frac{f(qw)}{f(w)}dw,
	\end{align*}
	where the integration contour goes around $z_1,\ldots,z_h$
	is the positive direction.
\end{lemma}
Using the above lemma and setting $z_1=\ldots=z_h=1$,
we read from \eqref{D1_identity}:
\begin{align}\label{D1_integral}&
	\frac{1}{2\pi\i}\oint\limits_{|w-1|=\varepsilon}
	\left(\frac{qw-1}{w-1}\right)^{h}\frac{1}{(q-1)w}e^{t(q-1)w}dw
	\\&\hspace{90pt}\nonumber=
	\sum_{\mu_1\ge \ldots\ge \mu_h}
	\Big(\sum_{r=1}^{h}q^{\mu_r+h-r}\Big)
	\Prob_{t,h}\{\mu\}.
\end{align}
The quantity in the right-hand side of \eqref{D1_integral}
does not seem very probabilistic, but we can now use the 
arbitrariness of the parameter $q$.
For any $n\in\Z$, we can compare the coefficients of $q^{n}$
in both sides of \eqref{D1_integral}. This amounts to integrating the left-hand side again 
(with $dq/q^{n+1}$),
and thus yields:
\begin{theorem} 
\label{thm:tiling_density}
For any $t\ge0$ and $h=1,2,\ldots$,
\begin{align}\label{D1_integral_density}&
	\Prob_{t,h}\big\{n\in\{\mu_i+h-i\}_{i=1}^{h}\big\}
	\\&\hspace{80pt}\nonumber=
	\frac{1}{(2\pi\i)^{2}}
	\oint\limits_{|q|=\varepsilon} \frac{dq}{q^{n+1}}
	\oint\limits_{|w-1|=\varepsilon}
	\left(\frac{qw-1}{w-1}\right)^{h}\frac{e^{t(q-1)w}}{(q-1)w}dw.
\end{align}
\end{theorem}
The left-hand side of \eqref{D1_integral_density} is a very
meaningful
probabilistic quantity --- it is the probability of seeing a vertical lozenge
at any given location on the horizontal slice (cf. Fig.~\ref{fig:slice_h}).
This is the so-called \emph{density function} of the measure 
$\Prob_{t,h}$. Furthermore, we see that the right-hand side of \eqref{D1_integral_density}
is well-suited for asymptotics.
We perform the asymptotic analysis in the next section.

\begin{remark}
	Theorem \ref{thm:tiling_density} is a special case of a more
	general formula that represents correlation functions of the so-called
	\emph{Schur measures} as multiple contour integrals. 
	See \cite{BorodinGorinSPB12} and references therein for details.
\end{remark}



\section{Asymptotics of tiling density via double contour integrals} 
\label{sec:asymptotics}

Here we perform an asymptotic analysis of the density function 
\eqref{D1_integral_density} of the measure 
$\Prob_{t,h}$ on the $h$th horizontal slice in the regime
\begin{align}\label{limit_regime}
	t=\tau L,\qquad
	n=\nu L,\qquad
	h=\eta L,\qquad L\to\infty,
\end{align}
where $n$ is the point of observation in \eqref{D1_integral_density},
which also must be scaled to yield nontrivial asymptotics.
The limit regime \eqref{limit_regime} is quite nontrivial
and is not achievable via elementary tools (in contrast with the limit transitions 
in \S \ref{sub:distribution_of_lozenges_on_a_horizontal_slice}).
The reader may want to peek at Figures \ref{fig:asymp_picture} and 
\ref{fig:sim_asymp_picture} below to see what type of description we
are aiming at.

Changing variables $q\mapsto v=wq$, $dq=dv/w$ in the left-hand side of \eqref{D1_integral_density}
gives
\begin{align*}&
	\Prob_{t,h}\big\{n\in\{\mu_i+h-i\}_{i=1}^{h}\big\}
	\\&\hspace{50pt}=\frac{1}{(2\pi\i)^{2}}
	\oint_{|v|=\varepsilon} 
	\oint_{|w-1|=\varepsilon}
	\frac{dv\cdot w^{-1}}{v^{n+1}w^{-n-1}}
	\left(\frac{v-1}{w-1}\right)^{h}
	\frac{e^{t(v-w)}dw}{v-w}
	\\&\hspace{50pt}=
	\frac{1}{(2\pi\i)^{2}}
	\oint\limits_{\Gamma_0} 
	\frac{dv}{v}
	\oint\limits_{\Gamma_1}
	dw \frac{e^{tv}(v-1)^{h}v^{-n}}
	{e^{tw}(w-1)^{h}w^{-n}}\frac{1}{v-w}.
\end{align*}
Here by $\Gamma_0$ and $\Gamma_1$ we have denoted small
positively oriented contours around $0$ and $1$, respectively.
Further analysis uses the original idea 
of Okounkov
\cite{Okounkov2002}
and largely follows 
\cite{BorFerr2008DF}.
We observe that
the integrand
above has the form
\begin{align*}
	\frac{e^{L(F(v)-F(w))}}{v(v-w)},\qquad \qquad
	F(z):=\tau z+\eta \ln(z-1)-\nu \ln z.
\end{align*}
If we manage to deform the contours
in such a way that $\Re\big(F(v)-F(w)\big)<0$
on them
except for possibly finite number of points
(where $\Re$ denotes the real part),
then our integral would asymptotically
vanish as $L\to\infty$. The deformation depends on the 
location of the 
critical points of $F(z)$, i.e.,
of the roots of the equation
\begin{align}\label{critical_points_equation}
	F'(z)=\frac{\tau z(z-1)+\eta z-\nu(z-1)}{z(z-1)}=0.
\end{align}
The discriminant of the numerator has the form
\begin{align}\label{discr}
	\mathrm{discr}=\big(\nu-(\sqrt\tau-\sqrt\eta)^2\big)\big(
	\nu-(\sqrt\tau+\sqrt\eta)^2\big).
\end{align}
We will now consider all possible cases one by one.

\medskip
\par\noindent
\textbf{Case 1. $\sqrt\nu>\sqrt\tau+\sqrt\eta$.} 
In this case both roots of \eqref{critical_points_equation}
are real and greater than $1$. The plot of $\Re\big(F(z)\big)$ looks as on 
Fig.~\ref{fig:nu_bigger} (top).\begin{figure}[htbp]
	\begin{center}
		\begin{tabular}{cc}
			\begin{tikzpicture}
				[scale=1.2, ultra thick]
				\draw[->] (0,0) -- (6.8,0);
				\draw[->] (0,0) -- (0,2.5);
				\draw[->] (0,0) -- (-2,-2);
				\draw[line width=.5, dashed] (0,0) -- (0,-2.5);
				\draw[line width=.5] (2,2.5) -- (2,-2.5);
				\draw[line width=.5] (4.95,-1.7) -- (4.95,0);
				\draw[line width=.5] (3.55,.9) -- (3.55,0);
				\draw[fill] (0,0) circle (.08) node[below right] {0};
				\draw[fill] (2,0) circle (.08) node[below right] {1};
				\draw[thick] plot [smooth, tension=.7] coordinates {(1.8,-2.5) (1.5,-1) (.6,1) (.3,2.5)};
				\draw[thick] plot [smooth, tension=1] coordinates { (-.3,2.5) (-.9,.4) (-2,-1)};
				\draw[fill] (4.95,-1.7) circle (.08) node[below] {$z_{\min}$};
				\draw[fill] (3.55,.9) circle (.08) node[above,yshift=4] {$z_{\max}$};
				\draw[thick] plot [smooth, tension=.8] coordinates { (2.2,-2.5) (3.5,.9) (5,-1.7) (6.8,2)};
				\draw [thick,domain=0:-180] plot ({-.02+.475*cos(\x)}, {1.5+.18*sin(\x)});
				\draw [thick,domain=0:180,dashed] plot ({-.02+.475*cos(\x)}, {1.5+.18*sin(\x)});
				\draw [thick,domain=0:-180] plot ({2.04+.534*cos(\x)}, {-1+.2*sin(\x)});
				\draw [thick,domain=0:180,dashed] plot ({2.04+.534*cos(\x)}, {-1+.2*sin(\x)});
				\draw [thick,domain=0:-180] plot ({3.529+.36*cos(\x)}, {.55+.11*sin(\x)});
				\draw [thick,domain=0:180,dashed] plot ({3.529+.36*cos(\x)}, {.55+.11*sin(\x)});
				\draw [thick,domain=0:-180] plot ({5.02+.73*cos(\x)}, {-.7+.21*sin(\x)});
				\draw [thick,domain=0:180,dashed] plot ({5.02+.73*cos(\x)}, {-.7+.21*sin(\x)});
			\end{tikzpicture}\\
			\rule{0pt}{144pt}
			\begin{tikzpicture}
				[scale=9]
				\def\x{0.008}
				\draw[->,thick] (-.36,0)--(.7,0);
				\draw[->,thick] (0,-.25)--(0,.25);
				\draw[fill] (0,0) circle(\x);\node [below left] at (0,0) {$0$};
				\draw[fill] (.3,0) circle(\x);\node [below] at (.3,0) {$1$};
			    \draw[ultra thick, decoration={markings,
				    mark=at position .1 with {\arrow{>}}}, 
				    postaction={decorate}]  
				    (.14,0) ellipse (.44 and .18)
				    node[yshift=38,xshift=90] {$v$};
				\draw[ultra thick, decoration={markings,
				    mark=at position .1 with {\arrow{>}}}, 
				    postaction={decorate}]  
				    (.3,0) ellipse (.1 and .08)
				    node[yshift=10,xshift=12] {$w$};
				\draw [fill] (.4,0) circle (\x) node[below right, xshift=-1] {$z_{\max}$};
				\draw [fill] (.58,0) circle (\x) node[below right, xshift=-1] {$z_{\min}$};
			\end{tikzpicture}
		\end{tabular}
	\end{center}
	\caption{Case 1. Plot of $\Re\big(F(z)\big)$ 
	(top), and the 
	deformed contours of integration (bottom).
	In this case
	the two critical points are real and $>1$.}
	\label{fig:nu_bigger}
\end{figure}
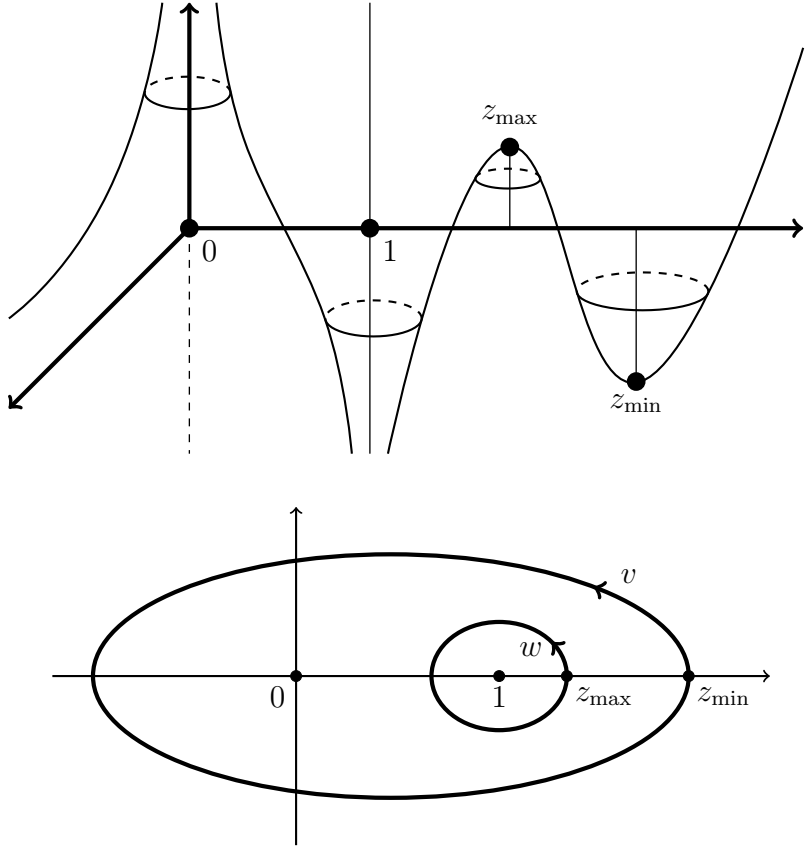
Moving the $v$ contour to the
level line $\Re\big(F(v)\big)=
\Re\big(F(z_{\min})\big)$
and 
the $w$ contour to the
level line $\Re\big(F(w)\big)=
\Re\big(F(z_{\max})\big)$,
we achieve that
$\Re\big(F(v)-F(w)\big)=F(z_{\min})-F(z_{\max})<0$,
which implies the desired vanishing. 
However, 
in the process of deformation,
the $v$ contour, which was originally a small circle around the origin,
has swallowed the $w$ contour, see Fig.~\ref{fig:nu_bigger} (bottom). 
Because of 
$(v-w)^{-1}$ in the integrand, we have to compensate 
the result of moving the contours by subtracting the residue
\begin{align*}
	-\frac{1}{2\pi\i}
	\oint_{\Gamma_1}dw
	\Res\limits_{v=w}\frac{1}{v}
	\frac{e^{L(F(v)-F(w))}}{v-w}
	=-\frac{1}{2\pi\i}
	\oint_{\Gamma_1}\frac{dw}{w}=0.
\end{align*}
Thus, we see that for $\sqrt\nu>\sqrt\tau+\sqrt\eta$,
the density of vertical lozenges asymptotically
vanishes.

\medskip
\par\noindent
\textbf{Case 2. $|\sqrt\tau-\sqrt\eta|<\sqrt\nu<\sqrt\tau+\sqrt\eta$.} 
In this case,
two critical points --- solutions of \eqref{critical_points_equation}
--- are complex conjugate. 
Consider the contour plot of 
$\Re\big(F(z)-F(z_c)\big)$, where we have shifted $F(z)$
by the value of $F$ at the upper critical point,
$F'(z_c)=0$, $\Im(z_c)>0$ ($\Im$ denotes the imaginary part).
This contour plot looks like Fig.~\ref{fig:contour_plot} (left).\begin{figure}[htbp]
	\begin{center}
		\begin{tabular}{cc}
			\includegraphics[height=0.35\textwidth]
			{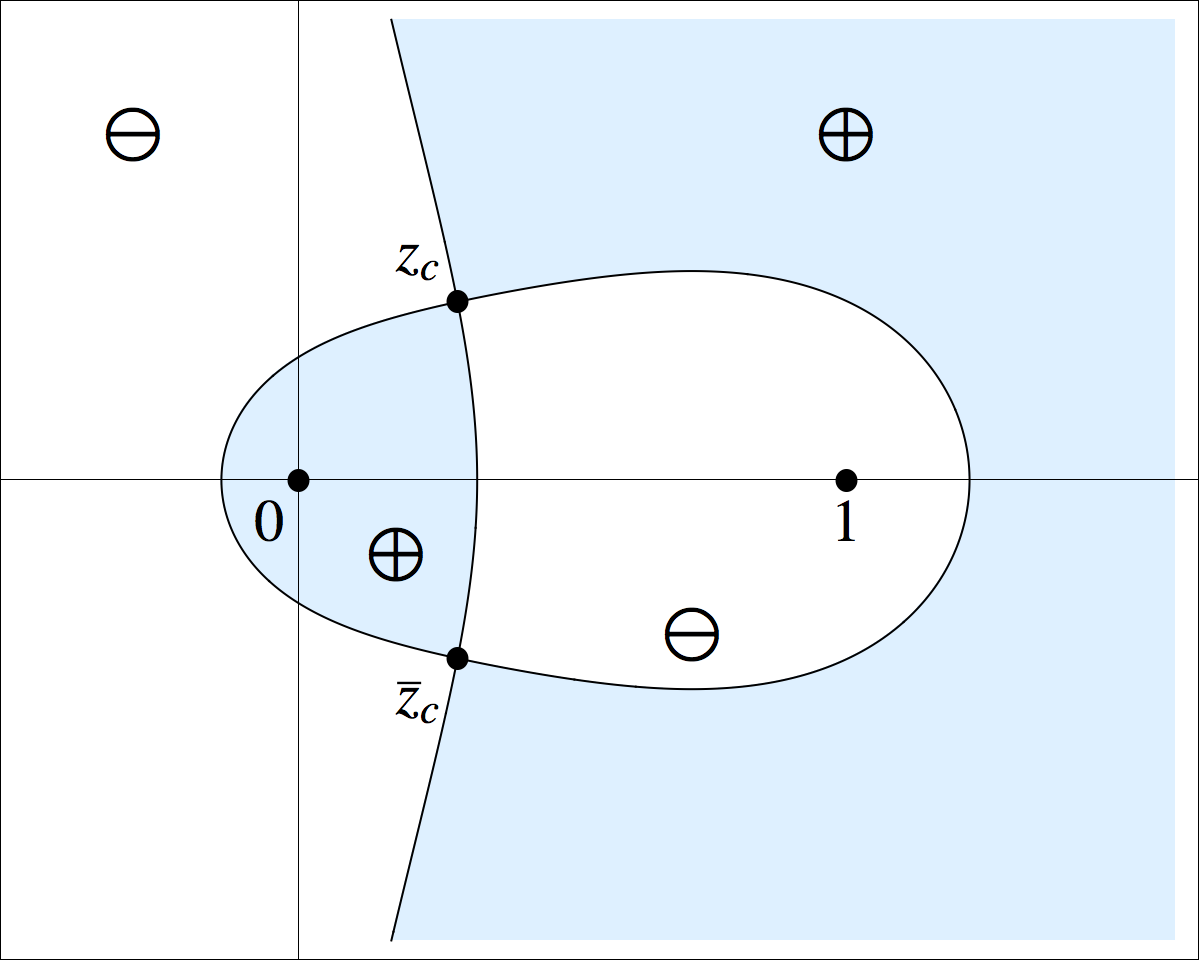}&	
			\hspace{10pt}
			\includegraphics[height=0.35\textwidth]{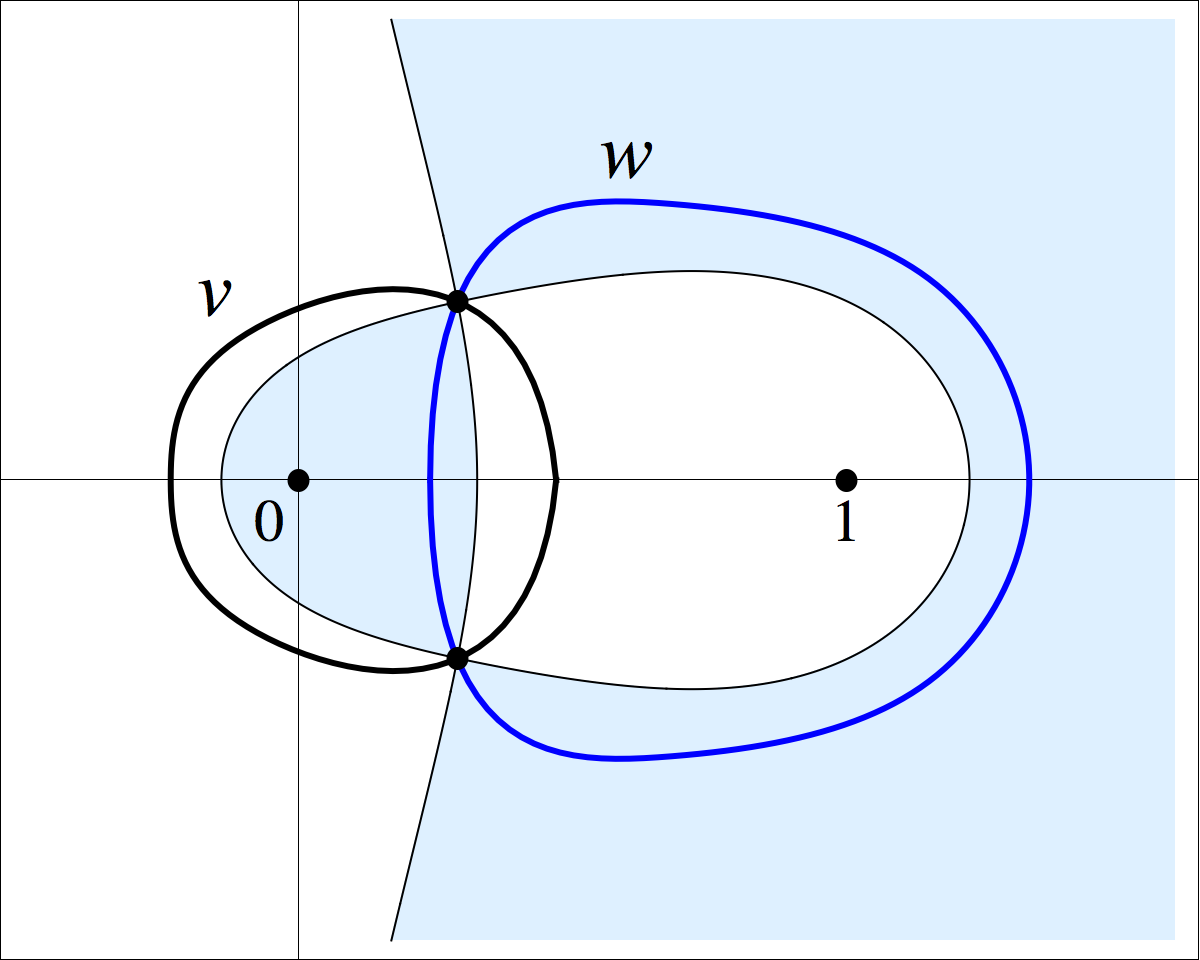}
		\end{tabular}
	\end{center}
	\caption{Case 2. Contour plot of $\Re\big(F(z)-F(z_c)\big)$ (left),
	and the deformed contours (right).}
	\label{fig:contour_plot}
\end{figure}
Deforming the $w$ contour into the 
region where $\Re\big(F(w)\big)$
is greater than 
$\Re\big(F(z_c)\big)$, and 
the $v$ contour into the region 
where $\Re\big(F(v)\big)$ is less than 
$\Re\big(F(z_c)\big)$, we again achieve
that $\Re\big(F(v)-F(w)\big)<0$ on the deformed contours.
However, in the process of deformation, we pick up the residue 
\begin{align*}
	-\frac{1}{2\pi\i}
	\int_{\bar z_c}^{z_c}
	dw
	\Res\limits_{v=w}\frac{1}{v}
	\frac{e^{L(F(v)-F(w))}}{v-w}
	=\frac{1}{2\pi\i}
	\int_{\bar z_c}^{z_c}
	\frac{dw}{w}=\frac{\arg(z_c)}{\pi},
\end{align*}
which is the limiting 
density function for vertical lozenges in this regime.

\medskip
\par\noindent
\textbf{Case 3. $0<\sqrt\nu<|\sqrt\tau-\sqrt\eta|$.} 
This final case contains two subcases
depending on whether $\sqrt\tau>\sqrt\eta$
or $\sqrt\tau<\sqrt\eta$.

In the first one, the plot of $\Re\big(F(z)\big)$
looks as on Fig.~\ref{fig:case3} (upper).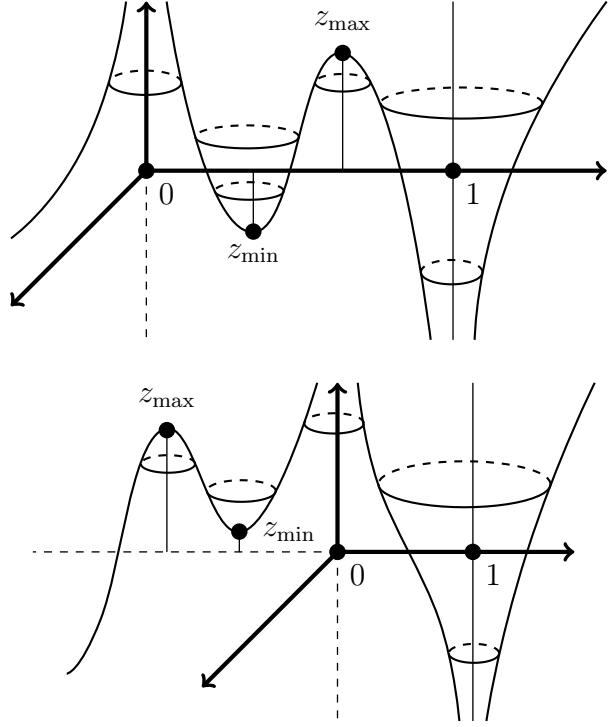
\begin{figure}[htbp]
	\begin{center}
		\begin{tabular}{c}
			\begin{tikzpicture}
				[scale=.9, ultra thick]
				\draw[->] (0,0) -- (6.8,0);
				\draw[->] (0,0) -- (0,2.5);
				\draw[->] (0,0) -- (-2,-2);
				\draw[line width=.5, dashed] (0,0) -- (0,-2.5);
				\draw[line width=.5] (4.53,2.5) -- (4.53,-2.5);
				\draw[line width=.5](2.9,1.74) -- (2.9,0);
				\draw[line width=.5] (1.58,-.9) -- (1.58,0);
				\draw[fill] (0,0) circle (.09) node[below right] {0};
				\draw[fill] (4.53,0) circle (.09) node[below right] {1};
				\draw[fill] (1.58,-.9) circle (.09) node[below] {$z_{\min}$};
				\draw[fill] (2.9,1.74) circle (.09) node[above,yshift=4] {$z_{\max}$};
				\draw[thick] plot [smooth, tension=1] coordinates { (-.3,2.5) (-.9,.4) (-2,-1)};
				\draw[thick] plot [smooth, tension=1] coordinates { (.3,2.5) (1.5,-.9) (3,1.7) (4.2,-2.5)};
				\draw[thick] plot [smooth, tension=1] coordinates {(4.85,-2.5) (5.4,0) (6.8,2.5)};
				\draw [thick,domain=0:180,dashed] plot ({4.51+.455*cos(\x)}, {-1.5+.18*sin(\x)});
				\draw [thick,domain=0:-180] plot ({4.51+.455*cos(\x)}, {-1.5+.18*sin(\x)});
				\draw [thick,domain=0:-180] plot ({4.655+1.195*cos(\x)}, {1+.23*sin(\x)});
				\draw [thick,domain=0:180,dashed] plot ({4.655+1.195*cos(\x)}, {1+.23*sin(\x)});
				\draw [thick,domain=0:-180] plot ({2.9+.41*cos(\x)}, {1.3+.13*sin(\x)});	
				\draw [thick,domain=0:180,dashed] plot ({2.9+.41*cos(\x)}, {1.3+.13*sin(\x)});	
				\draw [thick,domain=0:-180] plot ({1.488+.765*cos(\x)}, {.5+.17*sin(\x)});	
				\draw [thick,domain=0:180,dashed] plot ({1.488+.765*cos(\x)}, {.5+.17*sin(\x)});	
				\draw [thick,domain=0:-180] plot ({1.52+.51*cos(\x)}, {-.3+.13*sin(\x)});	
				\draw [thick,domain=0:180,dashed] plot ({1.52+.51*cos(\x)}, {-.3+.13*sin(\x)});	
				\draw [thick,domain=0:-180] plot ({-.017+.535*cos(\x)}, {1.3+.18*sin(\x)});	
				\draw [thick,domain=0:180,dashed] plot ({-.017+.535*cos(\x)}, {1.3+.18*sin(\x)});	
			\end{tikzpicture}
			\vspace{10pt}\\
			\begin{tikzpicture}
				[scale=.9, ultra thick]
				\draw[->] (0,0) -- (0,2.5);
				\draw[->] (0,0) -- (-2,-2);
				\draw[line width=.5, dashed] (0,0) -- (0,-2.5);
				\draw[->] (0,0) -- (3.5,0);
				\draw[line width=.5, dashed] (0,0) -- (-4.5,0);
				\draw[line width=.5] (2,2.5) -- (2,-2.5);
				\draw[line width=.5] (-1.45,.3) -- (-1.45,0);
				\draw[line width=.5] (-2.52,1.8) -- (-2.52,0);
				\draw[fill] (0,0) circle (.09) node[below right] {0};
				\draw[fill] (2,0) circle (.09) node[below right] {1};
				\draw[fill] (-1.45,.3) circle (.09) node[right, xshift=4] {$z_{\min}$};
				\draw[fill] (-2.52,1.8) circle (.09) node[above,yshift=4] {$z_{\max}$};
				\draw[thick] plot [smooth, tension=1] coordinates {(2.2,-2.5) (2.8,0) (3.8,2.5)};
				\draw[thick] plot [smooth, tension=.7] coordinates {(1.8,-2.5) (1.5,-1) (.6,1) (.3,2.5)};
				\draw[thick] plot [smooth, tension=.8] coordinates {(-.3,2.5) (-1.4,.3) (-2.6,1.8) (-3.5, -1) (-4,-1.8)};
				\draw [thick,domain=0:-180] plot ({2.0155+.372*cos(\x)}, {-1.5+.14*sin(\x)});
				\draw [thick,domain=0:180,dashed] plot ({2.0155+.372*cos(\x)}, {-1.5+.14*sin(\x)});
				\draw [thick,domain=0:-180] plot ({1.875+1.272*cos(\x)}, {1+.34*sin(\x)});
				\draw [thick,domain=0:180,dashed] plot ({1.875+1.272*cos(\x)}, {1+.34*sin(\x)});
				\draw [thick,domain=0:-180] plot ({-1.421+.51*cos(\x)}, {.9+.16*sin(\x)});
				\draw [thick,domain=0:180, dashed] plot ({-1.421+.51*cos(\x)}, {.9+.16*sin(\x)});
				\draw [thick,domain=0:-180] plot ({-2.504+.4*cos(\x)}, {1.3+.13*sin(\x)});
				\draw [thick,domain=0:180,dashed] plot ({-2.504+.4*cos(\x)}, {1.3+.13*sin(\x)});
				\draw [thick,domain=0:-180] plot ({-.05+.44*cos(\x)}, {1.9+.15*sin(\x)});
				\draw [thick,domain=0:180,dashed] plot ({-.05+.44*cos(\x)}, {1.9+.15*sin(\x)});
			\end{tikzpicture}
		\end{tabular}
	\end{center}
	\caption{Case 3. Plots of $\Re\big(F(z)\big)$ for 
	two subcases, $\tau>\eta$ (upper) and 
	$\tau<\eta$ (lower).}
	\label{fig:case3}
\end{figure}
Deforming the integration contours 
to level lines (similarly to what was done before in
Case 1)
requires no residue-picking. Thus, the limiting density 
is zero for the subcase $\sqrt\tau>\sqrt\eta$.

In the second subcase, the picture is slightly different, 
see Fig.~\ref{fig:case3} (lower). 
The familiar deformation of the contours
to the level lines now requires 
that the $w$ contour swallows the $v$ contour
(see Fig.~\ref{fig:case3_contours}).\begin{figure}[htbp]
	\begin{center}
		\begin{tabular}{cc}
			\begin{tikzpicture}
				[scale=9]
				\def\x{0.008}
				\draw[->,thick] (-.56,0)--(.56,0);
				\draw[->,thick] (0,-.25)--(0,.25);
				\draw[fill] (0,0) circle(\x);\node [below left] at (0,0) {$0$};
				\draw[fill] (.3,0) circle(\x);\node [below] at (.3,0) {$1$};
			    \draw[ultra thick, decoration={markings,
				    mark=at position .1 with {\arrow{>}}}, 
				    postaction={decorate}]  
				    (0,0) ellipse (.44 and .18)
				    node[yshift=38,xshift=90] {$w$};
				\draw[ultra thick, decoration={markings,
				    mark=at position .1 with {\arrow{>}}}, 
				    postaction={decorate}]  
				    (0,0) ellipse (.24 and .1)
				    node[yshift=18,xshift=20] {$v$};
				\draw [fill] (-.44,0) circle (\x) node[below left, xshift=3] {$z_{\max}$};
				\draw [fill] (-.24,0) circle (\x) node[below left, xshift=3] {$z_{\min}$};
			\end{tikzpicture}
		\end{tabular}
	\end{center}
	\caption{Case 3. Deformed contours when
	$\tau<\eta$.}
	\label{fig:case3_contours}
\end{figure}
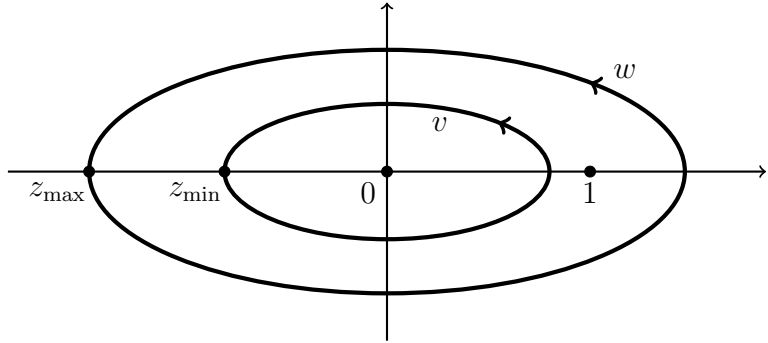
This results in the extra residue
\begin{align*}
	-\frac{1}{2\pi\i}
	\oint_{\Gamma_0}\frac{dv}{v}
	\Res_{w=v}
	\frac{e^{L(F(v)-F(w))}}{v-w}
	=\frac{1}{2\pi\i}
	\oint_{\Gamma_0}\frac{dv}{v}=1.
\end{align*}
Thus, the limiting density is $1$ in this case.

\medskip

Summarizing, we see that 
the asymptotic density of the vertical
lozenges is nontrivial 
for each given $\tau$, inside the 
parabola $\mathrm{discr}=0$
\eqref{discr} in the $(\nu,\eta)$-plane.
Outside of this parabola, the density of the vertical
lozenges either vanishes or tends to $1$, 
signaling the \emph{frozen parts} (\emph{facets})
of the limit shape, see Figures~\ref{fig:asymp_picture}
and~\ref{fig:sim_asymp_picture}.
\begin{figure}[htbp]
	\begin{center}
		\begin{tabular}{cc}
		\begin{tikzpicture}
			[scale=.75, ultra thick,
			block/.style ={rectangle, thick, draw=black, text width=13em,
			align=center, rounded corners, minimum height=2em}]
			\def\rt{0.866025}
			\draw[->] (-1,2*\rt)--(0,0) -- (11,0) node[below, xshift=3] {$\nu$};
			\draw[->] (0,0) -- (-4.1,8.2*\rt) node[left] {$\eta$};
			\def\sml{.3}
			\def\vlbx{-3.5-\sml/2}
			\def\vlby{7}
			\foreach \vl in {
			(\vlbx,\vlby*\rt),
			(\vlbx+\sml/2,\vlby*\rt-\sml*\rt),
			(\vlbx+\sml,\vlby*\rt-2*\sml*\rt),
			(\vlbx-\sml/2,\vlby*\rt+\sml*\rt),
			(\vlbx+\sml/2,\vlby*\rt+\sml*\rt),
			(\vlbx-\sml,\vlby*\rt+2*\sml*\rt),
			(\vlbx+0*\sml,\vlby*\rt+2*\sml*\rt),
			(3.6,\rt*\sml),(3.6-7/2*\sml,2*\rt*\sml),(3.6+3/2*\sml,2*\rt*\sml),
			(3.6-5*\sml,3*\rt*\sml),
			(3.6-1*\sml,3*\rt*\sml),(3.6+4*\sml,3*\rt*\sml)}
			{
			\begin{scope}[shift=\vl,scale=\sml]
				\draw [thick] (0,0) -- (.5,\rt) -- (1,0) -- (.5,-\rt) -- cycle;
			\end{scope}}
			\def\rlbx{7}
			\def\rlby{0}
			\foreach \rl in {
			(\rlbx-\sml,\rlby),(\rlbx,\rlby),(\rlbx+\sml,\rlby),
			(\rlbx+2*\sml,\rlby),(\rlbx+3*\sml,\rlby),
			(\rlbx+4*\sml,\rlby),(\rlbx+5*\sml,\rlby),
			(\rlbx+6*\sml,\rlby),(\rlbx+\sml/2,\rlby+\rt*\sml),
			(\rlbx+3*\sml/2,\rlby+\rt*\sml),(\rlbx+5*\sml/2,\rlby+\rt*\sml),
			(\rlbx+7*\sml/2,\rlby+\rt*\sml),(\rlbx+9*\sml/2,\rlby+\rt*\sml),
			(\rlbx+11*\sml/2,\rlby+\rt*\sml),(\rlbx+13*\sml/2,\rlby+\rt*\sml)}
			{
			\begin{scope}[shift=\rl,scale=\sml]
				\draw [thick] 
				(0,0) -- (.5,\rt) -- (1.5,\rt) -- (1,0) -- cycle;
			\end{scope}}
			\def\llbx{-\sml/2}
			\def\llby{0}
			\foreach \ll in {
			(\llbx,\llby),
			(\llbx+\sml,\llby),
			(\llbx+2*\sml,\llby),
			(\llbx+\sml/2,\llby+\rt*\sml),
			(\llbx-\sml/2,\llby+\rt*\sml),
			(\llbx-\sml,\llby+2*\rt*\sml),
			(\llbx-0*\sml,\llby+2*\rt*\sml)}
			{
			\begin{scope}[shift=\ll,scale=\sml]
				\draw [thick] (0,0) -- (-.5,\rt) -- (.5,\rt) -- (1,0) -- cycle;
			\end{scope}}
			
			\draw[ultra thick]
			(-3.20175, 
			7.7841) --(-3.14443, 7.3984) --(-3.08146, 7.0225) --(-3.01282, 
			6.6564) --(-2.93853, 6.3001) --(-2.85859, 5.9536) --(-2.77298, 
			5.6169) --(-2.68172, 5.29) --(-2.58479, 4.9729) --(-2.48221, 
			4.6656) --(-2.37397, 4.3681) --(-2.26008, 4.0804) --(-2.14052, 
			3.8025) --(-2.01531, 3.5344) --(-1.88444, 3.2761) --(-1.74791, 
			3.0276) --(-1.60573, 2.7889) --(-1.45788, 2.56) --(-1.30438, 
			2.3409) --(-1.14522, 2.1316) --(-0.9804, 1.9321) --(-0.809923, 
			1.7424) --(-0.633789, 1.5625) --(-0.451996, 1.3924) --(-0.264545, 
			1.2321) --(-0.0714363, 1.0816) --(0.12733, 0.9409) --(0.331755, 
			0.81) --(0.541838, 0.6889) --(0.757579, 0.5776) --(0.978978, 
			0.4761) --(1.20604, 0.3844) --(1.43875, 0.3025) --(1.67712, 
			0.2304) --(1.92115, 0.1681) --(2.17084, 0.1156) --(2.42619, 
			0.0729) --(2.6872, 0.04) --(2.95386, 0.0169) --(3.22618, 
			0.0036) --(3.50416, 0.0001) --(3.7878, 0.0064) --(4.07709, 
			0.0225) --(4.37205, 0.0484) --(4.67266, 0.0841) --(4.97893, 
			0.1296) --(5.29085, 0.1849) --(5.60844, 0.25) --(5.93168, 
			0.3249) --(6.26058, 0.4096) --(6.59514, 0.5041) --(6.93536, 
			0.6084) --(7.28124, 0.7225) --(7.63277, 0.8464) --(7.98996, 
			0.9801) --(8.35281, 1.1236) --(8.72132, 1.2769) --(9.09549, 
			1.44)
			--
			(9.47531, 1.6129)-- (9.86079, 1.7956)--(10.2519, 1.9881)--(10.6487, 
			2.1904);

			\draw [fill] (-2*\rt,3) circle (0.08) 
			node [below left] {$\eta=\tau$};
			\draw [fill] (4*\rt,0) circle (0.08)
			node [below, yshift=-4,xshift=5] {$\nu=\tau$};
			\draw (6,5) node[block] (discr) {$\mathrm{discr}=
			4\eta\tau-\big(\eta+\tau-\nu\big)^{2}=0$};
			\draw[->, thick] (discr.south) -- (8,1.09);
			\draw[->, thick] (discr.west) -- (-2.66,5.5);
			\node at (2,2.3) {$\rho=\dfrac{\arg(z_c)}{\pi}$};
			\node at (10,1) {$\rho=0$};
			\node at (-1,-.7) {$\rho=0$};
			\draw[->,thick] (-1.2,-.4) -- (-.05,.7);
			\node at (-4,4.8) {$\rho=1$};
			\draw[->,thick] (-4.3,5.15) -- (-3.15,6.15);
		\end{tikzpicture}
		\end{tabular}
	\end{center}
	\caption{Limiting density of the vertical lozenges
	in the $(\nu,\eta)$-plane.}
	\label{fig:asymp_picture}
\end{figure}
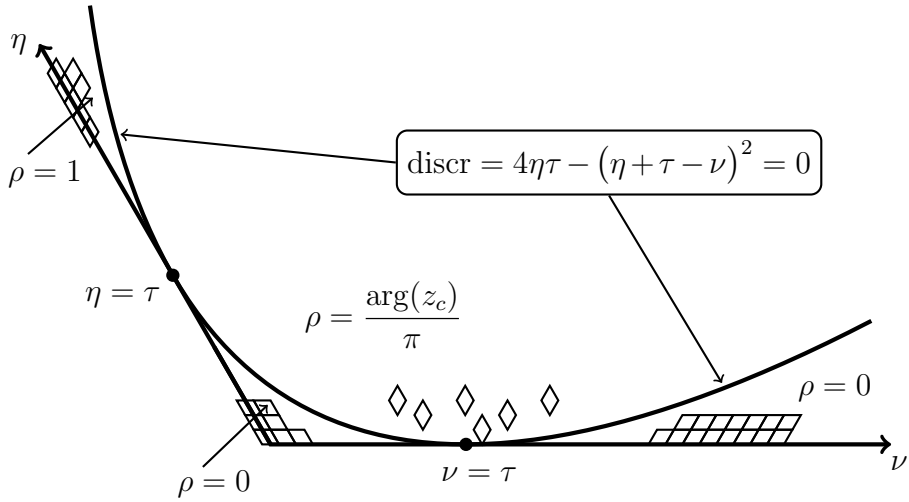

\begin{figure}[htbp]
	\begin{center}
		\begin{tabular}{cc}
			\includegraphics[width=0.75\textwidth]{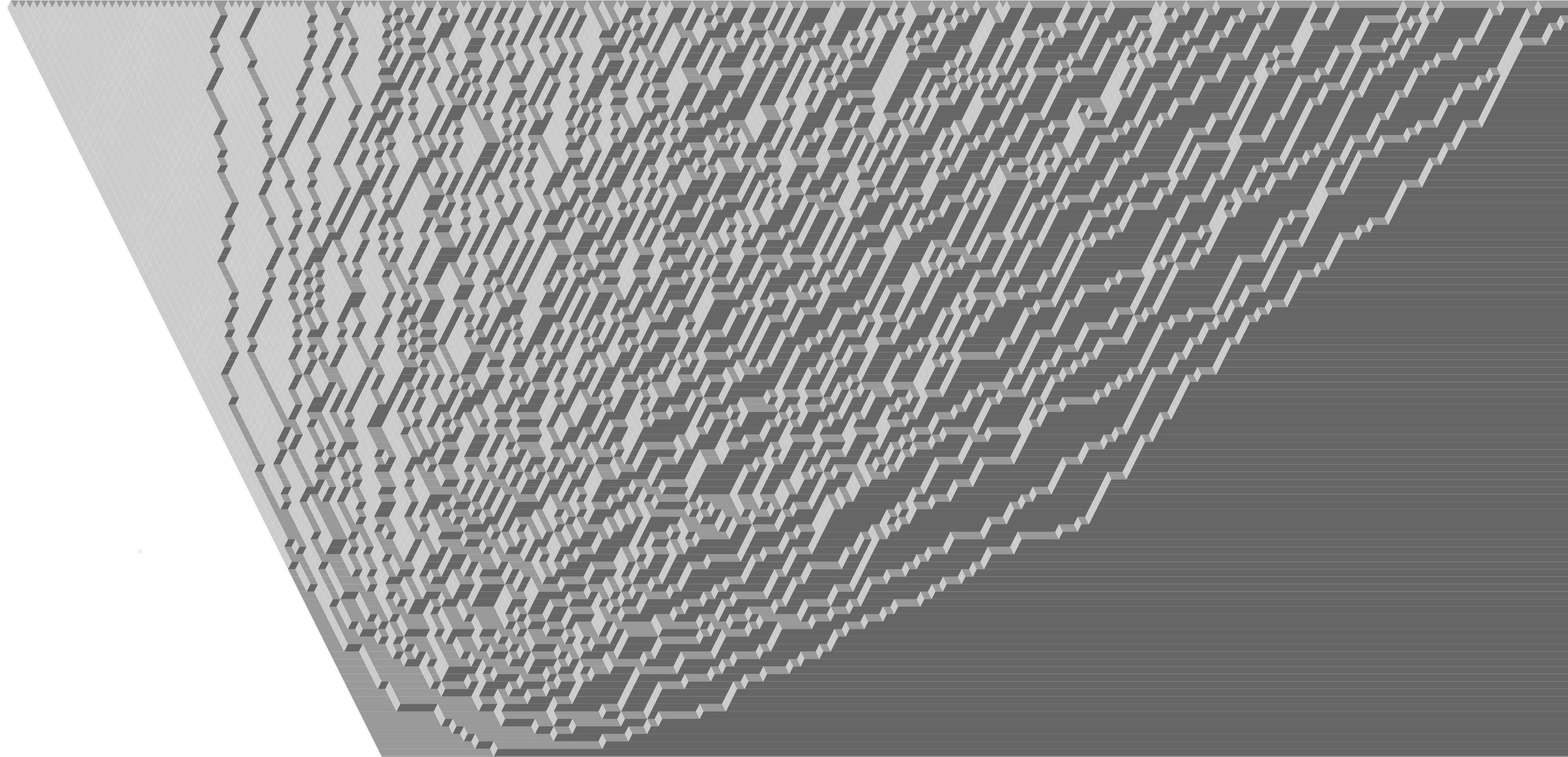}
		\end{tabular}
	\end{center}
	\caption{Simulation of the limiting distribution of lozenges.
	See also \cite{Ferrari_anim}.}
	\label{fig:sim_asymp_picture}
\end{figure}

In a similar way, using products of operators
$\D^{(1)}$ with different values of $q$, 
one can extract integral representations
for higher \emph{correlation functions} of 
vertical lozenges 
(i.e., probabilities that
a given set of locations is occupied by vertical lozenges).
Those integral representations
can be analyzed exactly in the same fashion 
as above, this was done in
\cite{BorFerr2008DF}. Indeed, 
if one knows
(here $h$ is fixed, but one can also handle 
different $h$'s)
\begin{align*}
	\sum_{\mu_1\ge \ldots\ge \mu_h}
	\Big(\sum_{i=1}^{h}q_1^{\mu_i+h-i}\Big)
	\ldots
	\Big(\sum_{i=1}^{h}q_s^{\mu_i+h-i}\Big)
	\Prob_{t,h}\{\mu\}
\end{align*}
for any $q_1,\ldots,q_s\in\C$, 
one can extract the order $s$
correlation function by 
looking at coefficients of 
monomials $q_1^{n_1}q_2^{n_2}\ldots q_{s}^{n_s}$.
The result reproduces known formulas for the correlation
functions of the so-called \emph{Schur processes}, e.g., 
see \cite{BorodinGorinSPB12} and references therein.

It should also be possible to carry out a similar 
program
for the case of the growing 
hexagon with sides $a,b,c,a,b,c$
when the 
triple ratio $a:b:c$
remains constant. This would require
analyzing the asymptotics
of ratios 
of the form
\begin{align*}
	\frac{s_\la(q_1,\ldots,q_s,1,\ldots,1)}
	{s_\la(1,1,\ldots,1)}
\end{align*}
with growing $\la$ as in \eqref{la_for_hexagon},
which can probably be done via 
recently developed techniques of~\cite{GorinPanova2012}.

In a different way, integral representations
for the correlation functions
in the hexagon were recently obtained
and asymptotically analyzed in 
\cite{Petrov2012}, \cite{Petrov2012GFF}, \cite{Petrov2012GT}.


\section{Markov dynamics} 
\label{sec:markov_dynamics}

Our next goal is to add an extra dimension
to our probabilistic models by introducing
suitable Markov evolutions
on them. This is not obvious and requires preliminary work.

\subsection{Dyson Brownian motion and its discrete counterparts} 
\label{sub:dyson_brownian_motion}

A hint at the existence of a nontrivial Markov dynamics
comes from the relation to random
matrices mentioned before (in particular, see 
the third limit regime in 
\S \ref{sub:distribution_of_lozenges_on_a_horizontal_slice}).
Indeed, a GUE matrix of size $N\times N$
has density with respect to the Lebesgue 
measure on the linear space 
$\mathrm{Herm}(N)$
of Hermitian $N\times N$ matrices
given by 
\begin{align*}
	\mathrm{M}(dX)=
	e^{-\Tr(X^2)/2}dX=
	\prod_{i=1}^{N}e^{-x_{ii}^{2}/2}
	\prod_{1\le i<j\le N}
	e^{-(\Re x_{ij})^{2}}
	e^{-(\Im x_{ij})^{2}}dX,
\end{align*}
where $X=[x_{ij}]_{i,j=1}^{N}$.
Thus, the $N^2$ quantities
\begin{align}\label{xij_elements}
	(x_{ii},\sqrt 2\cdot \Re x_{ij},\sqrt2\cdot \Im x_{ij})
\end{align}
are independent identically distributed
standard normal random variables.
Following Dyson \cite{dyson1962brownian}, one can replace these 
variables by standard Brownian motions. 
A nontrivial computation
shows that the corresponding Markov process
on Hermitian matrices
projects to a \emph{Markov}
process
on the spectra of matrices.
The generator of the process on 
the spectra is given by 
(here $\mathrm{Spec}(X)=(x_1,\ldots,x_N)$):
\begin{align}\label{Dyson_BM}&
	\frac{1}{2}\sum_{i=1}^{N}
	\frac{\partial^{2}}{\partial x_i^{2}}
	+\sum_{i=1}^{N}\Big(\sum_{j\ne i}
	\frac{1}{x_i-x_j}
	\Big)
	\frac{\partial}{\partial x_i}
	\\&\hspace{70pt}\nonumber={\prod\limits_{1\le i<j\le N}(x_i-x_j)^{-1}}
	\circ
	\frac{1}{2}\sum_{i=1}^{N}\frac{\partial^{2}}{\partial x_i^{2}}
	\circ
	\prod\limits_{1\le i<j\le N}(x_i-x_j).
\end{align}
Here on the right, $\circ$
means composition of operators:
First, we multiply by the Vandermonde determinant
$\prod_{i<j}(x_i-x_j)$, then 
apply the Laplacian, and after that 
divide by the Vandermonde determinant, similarly to 
\eqref{Df_general} above.
The projection of the (random) matrix
$X(t)\in\mathrm{Herm}(N)$ 
(evolving according to standard Brownian motions
of its elements \eqref{xij_elements})
to the spectrum 
then has the distribution density 
(see for example \cite{AndersonGuionnetZeitouniBook})
\begin{align*}
	\mathrm{const}\cdot \prod_{1\le <j\le N}
	(x_i-x_j)^{2}\prod_{i=1}^{N}e^{-x_i^{2}/2t}.
\end{align*}

The dynamics with generator \eqref{Dyson_BM} 
(called the \emph{Dyson Brownian motion})
can be easily mimicked for all the 
ensembles of the form 
$\mathrm{const}\cdot \prod_{i<j}(x_i-x_j)^{2}\prod_{i}w(x_i)$
considered in 
\S \ref{sub:distribution_of_lozenges_on_a_horizontal_slice}.
Let us focus on the Poisson
($ab/c\to t$) case, when $w(x)=t^{x}/x!$, $x\in\Z_{\ge0}$.
Consider a Markov jump process
with generator
\begin{align}\label{Poisson_generator}
	L^{(N)}_{\mathrm{Poisson}}
	={\prod\limits_{1\le i<j\le N}(x_i-x_j)^{-1}}
	\circ
	\sum_{i=1}^{N}
	\nabla_{i}
	\circ
	\prod\limits_{1\le i<j\le N}(x_i-x_j),
\end{align}
where $(\nabla f)(x)=f(x+1)-f(x)$
is the generator of the standard
Poisson process, and $\nabla_i$
acts as $\nabla$
on the $i$th coordinate.
One 
easily checks that the measures with 
$w(x)=t^{x}/x!$
are generated 
by the above Markov process
started from the initial condition
$(x_1,\ldots,x_N)=(N-1,N-2,\ldots,1,0)$.

The process with generator \eqref{Poisson_generator}
can be obtained by conditioning 
independent Poisson processes not to 
intersect until 
time $+\infty$, and also to grow at the same rate:
\begin{align*}
	\lim_{t\to\infty}\frac{x_1}{t}=\ldots=
	\lim_{t\to\infty}\frac{x_N}{t}.
\end{align*}
(Different growth rates
of different $x_i$'s
will result in conjugating 
$\sum_{i=1}^{N}\nabla_{i}$ by a different function, 
cf. 
\cite{konig2002non},
\cite{BG2011non}.)
This is similar to the
stationary version of the
Dyson Brownian
motion being obtained
from independent one-dimensional 
standard Brownian motions
by conditioning on the event that 
they never intersect,
and, moreover, stay 
within the distance $o(\sqrt{\mathrm{time}})$
from the origin as time goes to plus or minus infinity.


\subsection{Gibbs property and stochastic links} 
\label{sub:gibbs_property}

There is also another 
``perpendicular'' Markovian structure 
on the measures from 
\S \ref{sub:distribution_of_lozenges_on_a_horizontal_slice}.
Observe that the uniform measure
on lozenge tilings has the following
property:
If we pick a domain inside the hexagon, then 
fixing the boundary lozenge configuration 
induces the uniform measure
on tilings of the interior.
This seemingly trivial observation
becomes useful 
when the hexagon becomes infinitely
large in some way 
(as in \S \ref{sub:distribution_of_lozenges_on_a_horizontal_slice}).
Then the global uniform measure
makes no sense, but this property
survives. We will refer to it as to the 
\emph{Gibbs property}.

In particular, fixing $h$
vertical lozenges on the horizontal slice
of height $h$ (as on Fig.~\ref{fig:slice_h})
induces the uniform
measure on the set of all
configurations of lozenges
between this slice and the lower border (height zero). 
Thus, given locations 
$x^{(h)}=(x^{(h)}_{1},\ldots,x^{(h)}_{h})$
of the vertical lozenges 
on the $h$-th slice, the distribution of 
$h-1$ vertical lozenges at height $h-1$
is given by the ratio
(assuming that $x^{(h-1)}$ interlaces $x^{(h)}$)
\begin{align}
	&\nonumber
	\Prob\big\{x^{(h-1)}\,|\,x^{(h)}
	\big\}
	\\&\hspace{10pt}=\frac{\text{\# of GT schemes of depth $h-1$ with top 
	row $x^{(h-1)}$}}
	{\text{\# of GT schemes of depth $h$ with top 
	row $x^{(h)}$}}\label{links}
	\\&\hspace{10pt}=
	\frac{\prod\limits_{1\le i<j\le h-1}
	\dfrac{x_i^{(h-1)}-x_j^{h-1}}{j-i}}
	{\prod\limits_{1\le i<j\le h}
	\dfrac{x_i^{(h)}-x_j^{h}}{j-i}}
	=
	(h-1)!\cdot
	\frac{\prod\limits_{1\le i<j\le h-1}
	{(x_i^{(h-1)}-x_j^{h-1})}}
	{\prod\limits_{1\le i<j\le h}
	(x_i^{(h)}-x_j^{h})}.\nonumber
\end{align}
(we have used Proposition \ref{prop:s_la_1111}).
We will denote the 
above probabilities by 
$\La^{h}_{h-1}(x^{(h)};x^{(h-1)})$.

Note that the horizontal slices of measures 
that we obtain 
in \S \ref{sub:distribution_of_lozenges_on_a_horizontal_slice}
by taking limits $ab/c\to t$
and $b/(b+c)\to\xi$ 
of the hexagon
are also related by these \emph{stochastic links}
$\La^{h}_{h-1}$.
In the GUE limit, the formula
remains the same, 
except that the
$x^{(h-1)}_i,x^{(h)}_i$
are now reals, not integers.
In this case the above formula 
\eqref{links}
gives the density of a Markov kernel
with respect to the Lebesgue measure.


\subsection{Example of a two-dimensional dynamics} 
\label{sub:an_example_of_a_two_dimensional_dynamics}

The two Markov processes discussed above 
(the Dyson Brownian motion and its discrete analogue
$L^{(N)}_{\mathrm{Poisson}}$)
are quite canonical, 
but they have one deficiency ---
they are one-dimensional (in the sense that the state space consists of particle configurations
in $\Z^{1}$ or $\R^{1}$).
We would like to construct a 
two-dimensional 
process which 
has interlacing two-dimensional
arrays 
\eqref{GT_scheme}
as its state space, and that
``stitches together'' the above one-dimensional
processes in a natural way.
We begin by considering 
one such process which is constructed as follows.

Consider random words built from the 
alphabet 
$\{1,2,\ldots,N\}$
as follows: Each letter $j$
is appended at the end of 
the word according to a standard (=~rate 1) Poisson
process' jumps, and different letters appear independently.
We can encode this as on Fig.~\ref{fig:RSK_LPP}:
We draw a star ($*$)
in row $j$
at the time moment when a new letter $j$
is added. The stars in each row form a Poisson
process, and different rows are independent.
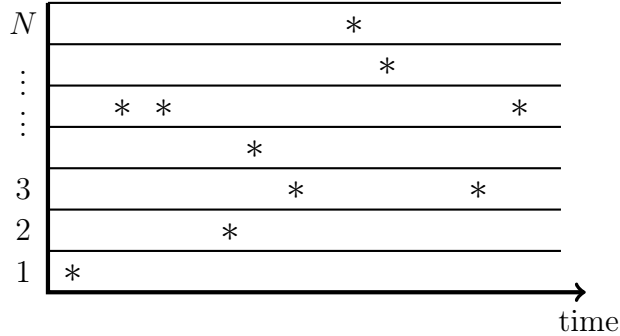
\begin{figure}[htbp]
	\begin{center}
		\begin{tabular}{cc}
			\begin{tikzpicture}
				[scale = 1.1, thick]
				\draw[->, ultra thick] (0,3.5)--(0,0) -- (6.5,0) node[below right, xshift=-15, yshift=-2] 
				{time};
				\foreach \line in {1,2,3,4,5,6,7}
				{
					\draw (6.2,\line/2)--(0,\line/2);
				}
				\foreach \line in {1,2,3}
				{
					\node at (-.3, \line/2-.25) {\line};
				}
				\node at (-.3, 7/2-.25) {$N$};
				\node at (-.3, 2.13) {$\vdots$};
				\node at (-.3, 2.63) {$\vdots$};
				\node at (.3,.23) {\large$*$};
				\node at (.9,2.23) {\large$*$};
				\node at (1.4,2.23) {\large$*$};
				\node at (2.2,.73) {\large$*$};
				\node at (2.5,1.73) {\large$*$};
				\node at (3,1.23) {\large$*$};
				\node at (3.7,3.23) {\large$*$};
				\node at (4.1,2.73) {\large$*$};
				\node at (5.2,1.23) {\large$*$};
				\node at (5.7,2.23) {\large$*$};
			\end{tikzpicture}
		\end{tabular}
	\end{center}
	\caption{Encoding random words.}
	\label{fig:RSK_LPP}
\end{figure}

From these data we construct
a Gelfand--Tsetlin scheme
\eqref{GT_scheme} of depth $N$
written as 
\begin{align*}
	\boldsymbol
	\la=(\la^{(1)}\prec\la^{(2)}\prec \ldots\prec\la^{(N)}),
	\qquad
	\la^{(h)}=(\la^{(h)}_{1}\ge \ldots\ge \la^{(h)}_{h}),
\end{align*}
as follows (see Fig.~\ref{fig:RSK_paths}):
\begin{align}
	\la^{(h)}_{1}+\la^{(h)}_{2}
	+\ldots+\la^{(h)}_{j}=
	\left(\parbox[3em]{0.58\textwidth}{
	the maximal number of ($*$)
	one can collect 
	on Fig.~\ref{fig:RSK_LPP}
	along $j$
	nonintersecting up-right
	paths that connect 
	points
	$(1,2,\ldots,j)$
	on the left border ($\mathrm{time}=0$), and $(h-j+1,h-j+2,\ldots,h)$
	on the right border ($\mathrm{time}=t>0$)
	}
	\right).
	\label{RSK_GT_definition}
\end{align}
\begin{figure}[htbp]
	\begin{center}
		\begin{tabular}{cc}
			\begin{tikzpicture}
				[scale = 1.1, thick]
				\draw[->, ultra thick] (0,3.5)--(0,0) -- (6.5,0) node[below right, xshift=-15, yshift=-2] 
				{time};
				\foreach \line in {1,2,3,4,5,6,7}
				{
					\draw (6.2,\line/2)--(0,\line/2);
				}
				\foreach \line in {1,2,3}
				{
					\node at (-.3, \line/2-.25) {\line};
				}
				\node at (-.3, 5/2-.25) {$h$};
				\node at (-.3, 7/2-.25) {$N$};
				\def\opac{.65}
				\draw[line width=3.5, color=red, opacity=\opac] (0,.73) 
				--++(.6,0)--++(0,1.5)
				--(6.2,2.23);
				\draw[line width=3.5, color=red, opacity=\opac] (0,.23) 
				--++(1.2,0)--++(0,.5)--++(1.4,0)--++(0,.5)
				--++(3,0)--++(0,.5)--(6.2,1.73);
				\draw[fill] (0,.73) circle(.1);
				\draw[fill] (0,.23) circle(.1);
				\draw[fill] (6.2,2.23) circle(.1);
				\draw[fill] (6.2,1.73) circle(.1);
				\node at (.3,.23) {\large$*$};
				\node at (.9,2.23) {\large$*$};
				\node at (1.4,2.23) {\large$*$};
				\node at (2.2,.73) {\large$*$};
				\node at (2.5,1.73) {\large$*$};
				\node at (3,1.23) {\large$*$};
				\node at (3.7,3.23) {\large$*$};
				\node at (4.1,2.73) {\large$*$};
				\node at (5.2,1.23) {\large$*$};
				\node at (5.7,2.23) {\large$*$};
			\end{tikzpicture}
		\end{tabular}
	\end{center}
	\caption{Nonintersecting paths used to determine 
	$\la^{(h)}_1+\la^{(h)}_2+\ldots+\la^{(h)}_j$, see \eqref{RSK_GT_definition}.
	On the picture, $h=5$, $j=2$, and $\la^{(5)}_{1}+\la^{(5)}_{2}=7$.}
	\label{fig:RSK_paths}
\end{figure}
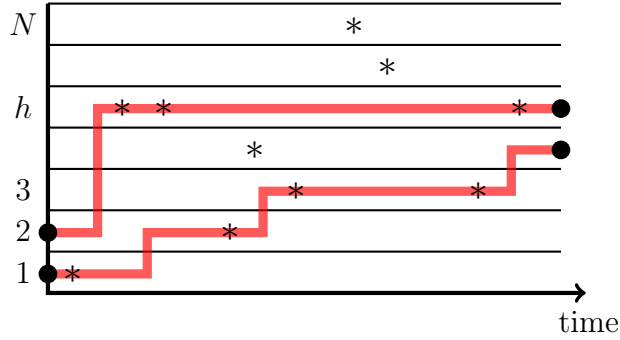

In particular, we see that $\la^{(h)}_{1}$, $h=1,\ldots,N$,
is the length of the longest increasing
subsequence of letters 
in the subword made of letters $\{1,2,\ldots,h\}$.
Moreover, 
$\la^{(h)}_{1}+\ldots+\la^{(h)}_{h}$
is the total number of letters $1,2,\ldots,h$
in our word.

\begin{proposition}
	After time $t$, 
	the distribution of the Gelfand--Tsetlin
	scheme $\la$ defined by \eqref{RSK_GT_definition} is the same as the 
	$ab/c\to t$
	limit of the uniform measure
	on tilings of hexagon. 
	That is, to obtain
	the measure on Gelfand--Tsetlin schemes, one takes
	the following 
	distribution of the top row $\la^{(N)}$:
	\begin{align*}
		\mathrm{const}\cdot
		\prod_{1\le i<j\le N}
		\left(\la^{(N)}_{i}-\la^{(N)}_{j}+j-i\right)^{2}\prod_{i=1}^{N}
		\frac{t^{\la^{(N)}_{j}+N-j}}{\big(\la^{(N)}_j+N-j \big)!},
	\end{align*}
	and projects it down by 
	the stochastic links, i.e., multiplies it by 
	\begin{align*}
		\La^{N}_{N-1}(\la^{(N)},\la^{(N-1)})
		\La^{N-1}_{N-2}(\la^{(N-1)},\la^{(N-2)})
		\ldots
		\La^{2}_{1}(\la^{(2)},\la^{(1)}).
	\end{align*}
\end{proposition}
\begin{proof}
	This is essentially Greene's theorem
	for the Robinson--Schensted correspondence
	coupled with explicit 
	formulas for the number of 
	standard and semistandard Young tableaux. 
	See, e.g.,
	\cite{Greene1974},
	\cite{sagan2001symmetric},
	\cite{Stanley1999}.
\end{proof}

As we are interested in time evolution, 
the following statement is relevant:

\begin{proposition}\label{prop:RSK_dyn}
	The Markov process on random
	words (i.e., the process of adding new letters according
	to standard Poisson processes)
	projects to a Markov
	process on Gelfand--Tsetlin schemes defined above.
	It can be described by the following rules:
	\begin{enumerate}[$\bullet$]
		\item Each ``particle''
		$\la^{(h)}_{1}$
		has an independent Poissonian clock of rate 1.
		When the clock rings, the particle jumps by 1,
		i.e., $\la^{(h)}_{1}\mapsto \la^{(h)}_{1}+1$.
		\item 
		When any particle $\la^{(h)}_{j}$
		moves by 1, it triggers either the move
		$\la^{(h+1)}_{j}\mapsto \la^{(h+1)}_{j}+1$,
		or
		$\la^{(h+1)}_{j+1}\mapsto \la^{(h+1)}_{j+1}+1$
		(exactly one of them), see Fig.~\ref{fig:moves_LR}.
		The second one is chosen generically, 
		while the first one is chosen
		only if $\la^{(h+1)}_{j}=\la^{(h)}_{j}$, 
		i.e., if the move
		$\la^{(h)}_{j}\mapsto \la^{(h)}_{j}+1$
		violated the interlacing constraint.
	\end{enumerate}
\end{proposition}
\begin{proof}
	See \cite{BorodinPetrov2013NN} (in particular,
	\S 7) and references therein.
\end{proof}
\begin{figure}[htbp]
	\begin{center}
		\begin{tikzpicture}[scale=1.3,thick]
			\def\h{1.25};
			\draw[fill] (\h,1) circle (.12) node [above right, xshift=-3] {$\la^{(h+1)}_{j}$};
			\draw[fill] (-\h,1) circle (.12) node [above left, xshift=3] {$\la^{(h+1)}_{j+1}$};
			\def\x{.1};
			\draw[->,line width=3, dotted] 
			(-\h*\x,\x) -- ++ (-\h+2*\h*\x,1-2*\x);
			\draw[->,line width=3, dotted] (\h*\x,\x) -- ++ (\h-2*\h*\x,1-2*\x);
			\draw[fill] (0,0) circle (.12) node [below right] {$\la^{(h)}_{j}$};
		\end{tikzpicture}	
	\end{center}
	\caption{Possible directions of move propagation, 
	see Proposition \ref{prop:RSK_dyn}.}
	\label{fig:moves_LR}
\end{figure}
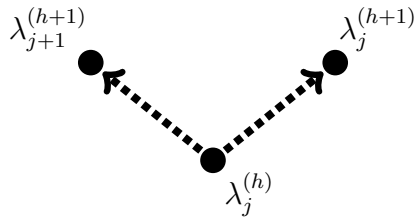

The Markov dynamics on Gelfand--Tsetlin schemes
from Proposition \ref{prop:RSK_dyn}
turns out to have the following properties:
\begin{enumerate}[(I)]
	\item For each $h\ge1$ and any initial condition, the evolution
	of $\{\la^{(1)},\ldots,\la^{(h)}\}$
	is Markovian (i.e., lower rows do not care about the upper ones).
	\item For each $h\ge1$, the evolution
	preserves the Gibbs property (\S \ref{sub:gibbs_property}) on 
	$\{\la^{(1)},\ldots,\la^{(h)}\}$.
	That is, if one starts with an initial condition
	of the form
	\begin{align*}
		\Prob\{\la^{(1)},\ldots,\la^{(h)}\}
		=m_h(\la^{(h)})
		\La^{h}_{h-1}(\la^{(h)},\la^{(h-1)})
		\ldots
		\La^{2}_{1}(\la^{(2)},\la^{(1)}),
	\end{align*}
	then after running the dynamics for any time $t>0$,
	the distribution of 
	$\{\la^{(1)},\ldots,\la^{(h)}\}$
	will be of the same form with a different
	probability measure $\tilde m_h$.
	\item 
	For each $h\ge1$, the map $m_h\mapsto\tilde m_h$
	is the time $t$ evolution of the Markov process
	with the generator $L^{(h)}_{\mathrm{Poisson}}$
	\eqref{Poisson_generator}.
\end{enumerate}
While (I) is obvious, (II) and (III)
are not; they follow e.g. from 
Theorem \ref{thm:Schur_system_of_eqns} below.

There is one more property 
which can be easily
observed from 
the random words description of the dynamics.
Namely, the projection of the process of 
Proposition \ref{prop:RSK_dyn}
to the rightmost particles 
$\la^{(1)}_{1},\ldots,\la^{(N)}_{1}$
is \emph{Markov}. 
It is more convenient to describe it in shifted
strictly 
ordered coordinates 
$y_1=\la^{(1)}_{1}< 
y_2=\la^{(2)}_{1}+1<\ldots< y_N=\la^{(N)}+N-1$
(cf. \eqref{shifted_GT_scheme}). Then each $y_j$
jumps to the right by 1
independently with rate 1,
and pushes $y_{j+1}$
over by 1 
if $y_{j+1}$ occupies the 
target location of $y_j$
(i.e., if we had $y_{j+1}=y_{j}+1$ before the jump).
We call this process the 
\emph{PushTASEP}, i.e., the
\emph{Pushing Totally Asymmetric Simple Exclusion Process}
(it was introduced in
\cite{Spitzer1970} under the name \emph{long-range TASEP},
see also \cite{BorFerr08push}).

\begin{remark}
	Definition \eqref{RSK_GT_definition}
	is 
	powered by what is known as 
	Robinson--Schen\-sted algorithm in 
	Combinatorics.
	Questions related to application of
	various 
	insertion algorithms (including
	the general Robinson--Schensted
	algorithm) 
	to random input 
	were considered in, e.g., 
	\cite{baik1999distribution},
	\cite{johansson2000shape},
	\cite[\S5]{Johansson2005lectures}, and
	\cite{ForresterRains2007},
	and can be traced back to 
	the 
	work of Vershik and Kerov
	\cite{Vershik1986}
	in mid-1980's.
	The 
	dynamical perspective 
	has been 
	substantially developed by 
	O'Connell \cite{OConnell2003Trans}, \cite{OConnell2003}, \cite{Oconnell2009_Toda} and
	Biane--Bougerol--O'Connell \cite{BBO2004}
	(see also Chhaibi \cite{Chhaibi2013}).
\end{remark}


\subsection{General construction of two-dimensional dynamics} 
\label{sub:general_construction_of_two_dimensional_dynamics}

The existence of the Markov dynamics
(of Proposition \ref{prop:RSK_dyn})
satisfying (I)--(III)
is remarkable, yet its above 
construction is fairly complicated. 
We would like to access it in a different way.

Let us search for all continuous-time Markov jump
processes on Gelfand--Tsetlin
schemes which satisfy conditions
(I)--(III) of \S \ref{sub:an_example_of_a_two_dimensional_dynamics}.
They must have the following structure: 
Each particle $\la^{(h)}_{j}$
jumps to the right by 1
with a certain rate (potentially dependent on 
$\la^{(1)},\ldots,\la^{(h)}$),
and its jump triggers further moves 
on the higher levels
$\la^{(h+1)},\ldots,\la^{(N)}$.
Indeed, because of (III) and the fact that 
$L^{(h)}_{\mathrm{Poisson}}$
moves one particle at a time, no two particles
on the same level can jump simultaneously.
Moreover, because of (I), moves can propagate
only upwards.

In order to reach a reasonable classification, we need 
to restrict the class further
by requiring \emph{nearest neighbor interactions}:
A move of $\la^{(h)}_{j}$
can only trigger
(potentially with certain probabilities)
moves of the immediate top right neighbor 
$\la^{(h+1)}_{j}$
and the 
the immediate top left 
neighbor $\la^{(h+1)}_{j+1}$ (see Fig.~\ref{fig:moves_LR}), which 
can trigger moves 
on level $h+2$, and so on.
Actually, it is better to extend the notion of
the top right nearest neighbor
from $\la^{(h+1)}_{j}$ to 
the first particle in the 
sequence
$\la^{(h+1)}_{j},\la^{(h+1)}_{j-1},\ldots,\la^{(h+1)}_{1}$
whose jump 
does not violate interlacing.
We will additionally assume 
(extending the nearest neighbor hypothesis)
that the individual 
jump rates of particles 
at level $h$
may only depend on $\la^{(h-1)}$ and $\la^{(h)}$,
and that the 
same is true for left and right probabilities of move
propagation from level $h-1$
to level $h$.

Let us now parametrize our possibilities.
Fix $h\ge2$, and denote 
by $w_j=w_j(\la^{(h-1)},\la^{(h)})$
the jump rate of $\la^{(h)}_{j}$, $1\le j\le h$.
Also, denote by $l_j$ the \emph{conditional} probabilities
that,
given that the 
$j$th  
part of $\la^{(h-1)}$
has just increased by 1,
this move propagates
to the top left neighbor
$\la^{(h)}_{j+1}$ of $\la^{(h-1)}_{j}$:
\begin{align*}
	l_j(\la^{(h-1)},\la^{(h)})=
	\Prob\big(\la^{(h)}\mapsto
	\la^{(h)}+\de_{j+1}\,|\,
	\la^{(h-1)}\mapsto
	\la^{(h-1)}+\de_{j}\big),
\end{align*}
where $\de_j$ is the vector 
having zeros at each position 
except the $j$th where it has 1.
Similarly, let 
\begin{align*}
	r_j(\la^{(h-1)},\la^{(h)})=
	\Prob\big(\la^{(h)}\mapsto
	\la^{(h)}+\de_{\xi(j)}\,|\,
	\la^{(h-1)}\mapsto
	\la^{(h-1)}+\de_{j}\big),
\end{align*}
where $\xi(j)$
is the lower index of the nearest 
top right neighbor of $\la^{(h-1)}_{j}$
that is free to jump (typically, $\xi(j)=j$).

\begin{proposition}\label{prop:apart}
	Assume that $\la^{(h-1)}$ and $\la^{(h)}$
	interlace with strict inequalities instead of weak ones.
	Then properties {\rm{}(II)\/}, {\rm{}(III)\/} imply
	\begin{align*}
		r_m+l_{m-1}+w_m=1,\qquad 1\le m\le h,
	\end{align*}
	where we set $r_h=l_0=0$.
\end{proposition}
\begin{proof}
	Consider the Gibbs measure on the first
	$h$ levels that projects to
	the delta measure
	at $\la^{(h)}$ on level $h$.
	The rate of any 
	jump $\la^{(h)}\mapsto\la^{(h)}+\de_m$, 
	$1\le m\le h$, can be computed in two different
	ways, using $L^{(h)}_{\mathrm{Poisson}}$
	or using the two-dimensional dynamics
	on the array.
	The projection of the array dynamics
	to levels $(h-1,h)$
	looks as follows:
	On level $h-1$ we have
	the process driven by 
	$L^{(h-1)}_{\mathrm{Poisson}}$
	whose jumps may propagate
	to level $h$ 
	with probabilities $(l_j,r_j)$.
	Moreover, 
	particles on level $h$
	can jump independently
	according to the jump rates $w_j$.
	Comparing the two ways to describe the rate of 
	$\la^{(h)}\mapsto\la^{(h)}+\de_m$ yields the desired relations.
\end{proof}

The system of equations of Proposition 
\ref{prop:apart} needs to be modified 
if the inequalities between parts of 
$\la^{(h-1)}$ and $\la^{(h)}$
are not strict. 
Indeed, if $\la^{(h)}_{j}$ is ``blocked'' 
by $\la^{(h-1)}_{j-1}$, i.e., 
$\la^{(h)}_{j}=\la^{(h-1)}_{j-1}$, then
$w_j$ must be zero, and also $l_{j-1}$
and $r_{j-1}$ make no sense as $\la^{(h-1)}$ could not
have just come from the jump
$\la^{(h-1)}-\de_{j-1}\mapsto\la^{(h-1)}$.
The modification looks as follows.
\begin{theorem}\label{thm:Schur_system_of_eqns}
	For any $\la^{(h-1)}\prec\la^{(h)}$, let 
	\begin{align*}
		\big\{j_1+1<j_2+1<\ldots<j_\kappa+1\big\}
	\end{align*}
	be the set of indices such that 
	each particle $\la^{(h)}_{j_m+1}$ 
	is free to move, i.e., $\la^{(h)}_{j_m+1}<\la^{(h-1)}_{j_m}$.
	Then
	\begin{align}\label{linear_eqn_j_m}
		r_{j_{m+1}}+l_{j_m}+w_{j_m+1}=1,\qquad 
		1\le m\le \kappa,
	\end{align}
	with the agreement that $r_{j_{\kappa}+1}=l_{j_0}=0$.

	Solving these equations for all pairs 
	$\la^{(h-1)}\prec\la^{(h)}$, $h=2,3,\ldots$,
	under the conditions 
	$r_j(\la^{(h-1)},\la^{(h)})\ge0$,
	$l_j(\la^{(h-1)},\la^{(h)})\ge0$,
	$w_j(\la^{(h-1)},\la^{(h)})\ge0$, 
	and 
	$r_j(\la^{(h-1)},\la^{(h)})+
	l_j(\la^{(h-1)},\la^{(h)})\le 1$, 
	is equivalent to
	constructing a nearest neighbor
	Markov dynamics
	as defined above
	satisfying conditions {\rm{}(I)--(III)\/}, with an additional
	``forced move'' rule: 
	If
	$\la^{(h-1)}_{j}=\la^{(h)}_{j}$
	and $\la^{(h-1)}_{j}$
	moves (by 1),
	then $\la^{(h)}_{j}$
	also moves.\footnote{The forced move corresponds
	to the only possibility
	of having $\la^{(h-1)}$ and $\la^{(h)}$
	not interlacing 
	after a move on level $h-1$.}
\end{theorem}
\begin{proof}
	Similar to Proposition \ref{prop:apart},
	cf. \cite[\S6]{BorodinPetrov2013NN}.
\end{proof}

It is easy to describe linear spaces of solutions
to the above linear
systems. Any combination of them, 
for every pair $\la^{(h-1)}\prec\la^{(h)}$,
gives us a Markov process with desired properties.
One can choose such combinations to design
different processes.


\subsection{Further examples of two-dimensional dynamics} 
\label{sub:examples_of_two}

We give three examples
below,
see \cite{BorodinPetrov2013NN} for more.

\medskip

\noindent
\textbf{Example 1.}
All $l_j\equiv 1$, all $r_j\equiv 0$, and 
\begin{align*}
	w_j=\begin{cases}
		1,&j=1;\\
		0,&\mbox{otherwise}.
	\end{cases}
\end{align*}
This is the dynamics 
described above in 
\S \ref{sub:an_example_of_a_two_dimensional_dynamics}
via nonintersecting paths.

\medskip

\noindent
\textbf{Example 2.}
All $r_j\equiv 1$, all $l_j\equiv 0$, and
\begin{align*}
	w_j=\begin{cases}
		1,&j=h;\\
		0,&\mbox{otherwise}.
	\end{cases}
\end{align*}
This dynamics can be viewed as coming from the 
column 
insertion algorithm (as opposed to the row insertion algorithm
corresponding to the dynamics of 
\S \ref{sub:an_example_of_a_two_dimensional_dynamics}).
Observe that the restriction of this dynamics
to the left-most particles
$\{\la^{(h)}_{h}\}_{h=1}^{N}$
is Markovian. Via the shift $y_h=\la^{(h)}_{h}-h$
this restriction matches the well-known \emph{Totally Asymmetric Simple Exclusion Process}
(\emph{TASEP}).
This dynamics was first introduced in \cite{OConnell2003Trans}.

\medskip

\noindent
\textbf{Example 3.}
All $r_j\equiv 0$, 
all $l_j\equiv 0$,
and all $w_j\equiv 1$.
This dynamics has minimal pushing and maximal
``noise''
(coming from individual jumps).
It can be viewed as a two-dimensional
growth model;
in terms of the stepped surfaces
interpretation, 
independently with rate one
this dynamics adds all possible 
``sticks'' (directed columns) of the form
\begin{center}
	\begin{tikzpicture}
		[scale=.5]
		\def\rr{.866};
		\def\h{.5};
		\foreach \vl in {(0,0)}
		{
		\begin{scope}[shift=\vl]
			\draw [thick] (0,0) -- (.5,0.866) -- (1,0) -- (.5,-0.866) -- cycle;
        \end{scope}
        }
        \foreach \ll in {(\h,-\rr)}
		{
		\begin{scope}[shift=\ll]
			\draw [thick] (0,0) -- (.5,0.866) -- (1.5,0.866) -- (1,0) -- cycle;
        \end{scope}
        }
        \foreach \rl in {(2*\h,0)}
		{
		\begin{scope}[shift=\rl]
			\draw [thick] (0,0) -- (-.5,0.866) -- (.5,0.866) -- (1,0) -- cycle;
        \end{scope}
        }
	\end{tikzpicture}
	\hspace{30pt}
	\begin{tikzpicture}
		[scale=.5]
		\def\rr{.866};
		\def\h{.5};
		\foreach \vl in {(0,0),(-\h,-\rr)}
		{
		\begin{scope}[shift=\vl]
			\draw [thick] (0,0) -- (.5,0.866) -- (1,0) -- (.5,-0.866) -- cycle;
        \end{scope}
        }
        \foreach \ll in {(\h,-\rr),(0,-2*\rr)}
		{
		\begin{scope}[shift=\ll]
			\draw [thick] (0,0) -- (.5,0.866) -- (1.5,0.866) -- (1,0) -- cycle;
        \end{scope}
        }
        \foreach \rl in {(2*\h,0)}
		{
		\begin{scope}[shift=\rl]
			\draw [thick] (0,0) -- (-.5,0.866) -- (.5,0.866) -- (1,0) -- cycle;
        \end{scope}
        }
	\end{tikzpicture}
	\hspace{30pt}
	\begin{tikzpicture}
		[scale=.5]
		\def\rr{.866};
		\def\h{.5};
		\foreach \vl in {(0,0),(-\h,-\rr),(-2*\h,-2*\rr)}
		{
		\begin{scope}[shift=\vl]
			\draw [thick] (0,0) -- (.5,0.866) -- (1,0) -- (.5,-0.866) -- cycle;
        \end{scope}
        }
        \foreach \ll in {(\h,-\rr),(0,-2*\rr),(-\h,-3*\rr)}
		{
		\begin{scope}[shift=\ll]
			\draw [thick] (0,0) -- (.5,0.866) -- (1.5,0.866) -- (1,0) -- cycle;
        \end{scope}
        }
        \foreach \rl in {(2*\h,0)}
		{
		\begin{scope}[shift=\rl]
			\draw [thick] (0,0) -- (-.5,0.866) -- (.5,0.866) -- (1,0) -- cycle;
        \end{scope}
        }
	\end{tikzpicture}
	\hspace{30pt}
	\begin{tikzpicture}
		[scale=.5]
		\def\rr{.866};
		\def\h{.5};
		\foreach \vl in {(0,0),(-\h,-\rr),(-2*\h,-2*\rr),(-3*\h,-3*\rr)}
		{
		\begin{scope}[shift=\vl]
			\draw [thick] (0,0) -- (.5,0.866) -- (1,0) -- (.5,-0.866) -- cycle;
        \end{scope}
        }
        \foreach \ll in {(\h,-\rr),(0,-2*\rr),(-\h,-3*\rr),(-1,-4*\rr)}
		{
		\begin{scope}[shift=\ll]
			\draw [thick] (0,0) -- (.5,0.866) -- (1.5,0.866) -- (1,0) -- cycle;
        \end{scope}
        }
        \foreach \rl in {(2*\h,0)}
		{
		\begin{scope}[shift=\rl]
			\draw [thick] (0,0) -- (-.5,0.866) -- (.5,0.866) -- (1,0) -- cycle;
        \end{scope}
        }
	\end{tikzpicture}
	\hspace{20pt}
	\raisebox{24pt}{$\ldots$}
\end{center}
directed as shown
(no overhangs allowed). 

Projection of this dynamics
to the leftmost particles 
$\{\la^{(h)}_{h}\}_{h=1}^{N}$
gives TASEP, and 
projection 
to the rightmost particles
$\{\la^{(h)}_{1}\}_{h=1}^{N}$
gives PushTASEP.
See \cite{BorFerr2008DF}
for more details on this dynamics.

\medskip

Pictorially, the three examples 
can be represented as on Fig.~\ref{fig:3ex}.\begin{figure}[htbp]
	\begin{center}
		\begin{tabular}{ccc}
			\begin{tikzpicture}[scale=.65]
				\def\r{.14};
				\def\h{.9};
				\def\x{.7};
				\foreach \pt in {(0,0),
				(-\x,\h),(\x,\h),
				(2*\x,2*\h),(-2*\x,2*\h),(0,2*\h),
				(3*\x,3*\h),(1*\x,3*\h),(-\x,3*\h),(-3*\x,3*\h)}
				{
					\draw[fill] \pt circle (\r);
				}
				\foreach \pt in {(0,0),
				(\x,\h),
				(2*\x,2*\h),
				(3*\x,3*\h)}
				{
					\draw[color=red,ultra thick] \pt circle (1.8*\r);
				}
				\foreach \la in {(0,0),
				(\x,\h),(-\x,\h),
				(2*\x,2*\h),(0,2*\h),(-2*\x,2*\h)}
				{
					\draw[line width=2.4,->] \la -- ++(-.9*\x,.9*\h);
				}
			\end{tikzpicture}
			&\hspace{20pt}
			\begin{tikzpicture}[scale=.65]
				\def\r{.14};
				\def\h{.9};
				\def\x{.7};
				\foreach \pt in {(0,0),
				(-\x,\h),(\x,\h),
				(2*\x,2*\h),(-2*\x,2*\h),(0,2*\h),
				(3*\x,3*\h),(1*\x,3*\h),(-\x,3*\h),(-3*\x,3*\h)}
				{
					\draw[fill] \pt circle (\r);
				}
				\foreach \pt in {(0,0),
				(-\x,\h),
				(-2*\x,2*\h),
				(-3*\x,3*\h)}
				{
					\draw[color=red,ultra thick] \pt circle (1.8*\r);
				}
				\foreach \ra in {(0,0),
				(\x,\h),(-\x,\h),
				(2*\x,2*\h),(0,2*\h),(-2*\x,2*\h)}
				{
					\draw[line width=2.4,->] \ra -- ++(.9*\x,.9*\h);
				}
			\end{tikzpicture}
			&\hspace{20pt}
			\begin{tikzpicture}[scale=.65]
				\def\r{.14};
				\def\h{.9};
				\def\x{.7};
				\foreach \pt in {(0,0),
				(-\x,\h),(\x,\h),
				(2*\x,2*\h),(-2*\x,2*\h),(0,2*\h),
				(3*\x,3*\h),(1*\x,3*\h),(-\x,3*\h),(-3*\x,3*\h)}
				{
					\draw[fill] \pt circle (\r);
				}
				\foreach \pt in {(0,0),
				(-\x,\h),(\x,\h),
				(2*\x,2*\h),(-2*\x,2*\h),(0,2*\h),
				(3*\x,3*\h),(1*\x,3*\h),(-\x,3*\h),(-3*\x,3*\h)}
				{
					\draw[color=red,ultra thick] \pt circle (1.8*\r);
				}
			\end{tikzpicture}
		\end{tabular}
	\end{center}
	\caption{Examples 1 (left), 2 (center), and 3 (right)
	of two-dimensional dynamics on interlacing arrays.
	Circles correspond to $w_j=1$, 
	and right and left arrows to 
	$r_j$ or $l_j=1$.}
	\label{fig:3ex}
\end{figure}
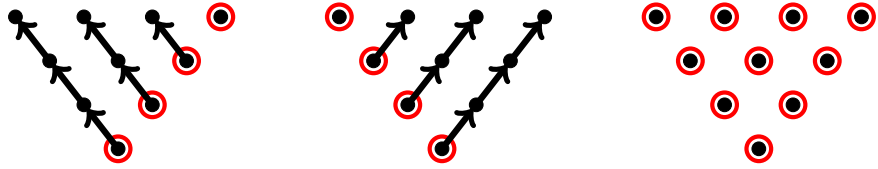
Pictures like Fig.~\ref{fig:any_ex} 
provide other interesting examples 
of two-dimensional dynamics, cf. \cite[\S7]{BorodinPetrov2013NN}.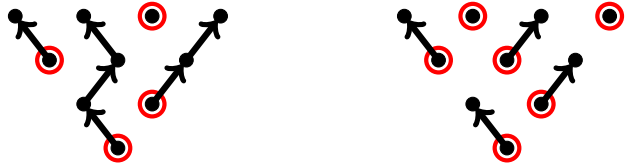
\begin{figure}[htbp]
	\begin{center}
		\begin{tabular}{cc}
			\begin{tikzpicture}[scale=.65]
				\def\r{.14};
				\def\h{.9};
				\def\x{.7};
				\foreach \pt in {(0,0),
				(-\x,\h),(\x,\h),
				(2*\x,2*\h),(-2*\x,2*\h),(0,2*\h),
				(3*\x,3*\h),(1*\x,3*\h),(-\x,3*\h),(-3*\x,3*\h)}
				{
					\draw[fill] \pt circle (\r);
				}
				\foreach \pt in {(0,0),
				(\x,\h),
				(-2*\x,2*\h),
				(1*\x,3*\h)}
				{
					\draw[color=red,ultra thick] \pt circle (1.8*\r);
				}
				\foreach \la in {(0,0),
				(0,2*\h),(-2*\x,2*\h)}
				{
					\draw[line width=2.4,->] \la -- ++(-.9*\x,.9*\h);
				}
				\foreach \ra in {
				(\x,\h),(-\x,\h),
				(2*\x,2*\h)}
				{
					\draw[line width=2.4,->] \ra -- ++(.9*\x,.9*\h);
				}
			\end{tikzpicture}	
			&\hspace{50pt}
			\begin{tikzpicture}[scale=.65]
				\def\r{.14};
				\def\h{.9};
				\def\x{.7};
				\foreach \pt in {(0,0),
				(-\x,\h),(\x,\h),
				(2*\x,2*\h),(-2*\x,2*\h),(0,2*\h),
				(3*\x,3*\h),(1*\x,3*\h),(-\x,3*\h),(-3*\x,3*\h)}
				{
					\draw[fill] \pt circle (\r);
				}
				\foreach \pt in {(0,0),
				(\x,\h),
				(-2*\x,2*\h),(0,2*\h),
				(3*\x,3*\h),(-1*\x,3*\h)}
				{
					\draw[color=red,ultra thick] \pt circle (1.8*\r);
				}
				\foreach \la in {(0,0),
				(-2*\x,2*\h)}
				{
					\draw[line width=2.4,->] \la -- ++(-.9*\x,.9*\h);
				}
				\foreach \ra in {
				(\x,\h),
				(0*\x,2*\h)}
				{
					\draw[line width=2.4,->] \ra -- ++(.9*\x,.9*\h);
				}
			\end{tikzpicture}
		\end{tabular}
	\end{center}
	\caption{Other examples of a two-dimensional dynamics.}
	\label{fig:any_ex}
\end{figure}
\begin{figure}[htbp]
	\begin{center}
		\framebox{\begin{tikzpicture}[scale=1,thick]
			\def\h{.65};
			\draw[fill] (\h,1) circle (.12);
			\draw[color=red, line width=1.8] (\h,1) circle (.22);
			\draw[color=white, line width=1.8] (-\h,1) circle (.22);
			\draw[fill] (-\h,1) circle (.12);
			\def\x{.12};
			\draw[->,line width=3] (0-\h*\x,\x) -- (-\h+\h*\x,1-\x);
			\draw[line width=3] (0,0) circle (.12);
		\end{tikzpicture}}
		\qquad\raisebox{23pt}{$\Longleftrightarrow$}\qquad
		\framebox{\begin{tikzpicture}[scale=1,thick]
			\def\h{.65};
			\draw[fill] (\h,1) circle (.12);
			\draw[color=red, line width=1.8] (\h,1) circle (.22);
			\draw[color=red, line width=1.8] (-\h,1) circle (.22);
			\draw[fill] (-\h,1) circle (.12);
			\def\x{.12};
			\draw[line width=3] (0,0) circle (.12);
		\end{tikzpicture}}
		\qquad
		\raisebox{23pt}{$\Longleftrightarrow$}\qquad
		\framebox{\begin{tikzpicture}[scale=1,thick]
			\def\h{.65	};
			\draw[fill] (\h,1) circle (.12);
			\draw[color=red, line width=1.8] (-\h,1) circle (.22);
			\draw[color=white, line width=1.8] (\h,1) circle (.22);
			\draw[fill] (-\h,1) circle (.12);
			\def\x{.12};
			\draw[->,line width=3] (\h*\x,\x) -- (\h-\h*\x,1-\x);
			\draw[line width=3] (0,0) circle (.12);
		\end{tikzpicture}}
	\end{center}
	\caption{Local ``flips'' are defined for any three particles
	$\la^{(h+1)}_{j}$, $\la^{(h+1)}_{j+1}$, and $\la^{(h)}_{j}$
	in the array such that the are no outside arrows pointing to any of the upper
	particles. 
	The type of the lower particle $\la^{(h)}_{j}$ can be arbitrary
	(i.e., it can jump independently or be pushed/pulled).
	The ``flip'' operation allows to replace one of the three 
	local pictures by any of the two remaining ones.}
	\label{fig:flips}
\end{figure}
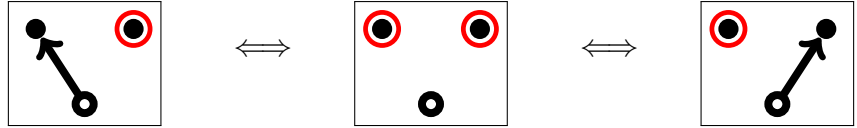
Note that due to \eqref{linear_eqn_j_m}, pictures 
corresponding to various
``fundamental'' dynamics for which 
two of the three quantities 
$r_{j_{m+1}}$, $l_{j_m}$, and $w_{j_m+1}$
in each of the equations
\eqref{linear_eqn_j_m}
are zero (and the remaining quantity is $1$)
can be transformed one into the other
by a sequence of local 
``flips'' as on Fig.~\ref{fig:flips}.


\subsection{Conclusion} 
\label{sub:conclusion1}

We have seen how to construct
random growth models
in $(1+1)$ dimension (TASEP, PushTASEP)
and $(2+1)$ dimension, and 
in \S \ref{sec:asymptotics} we have seen 
how these models can be analyzed 
at large times.
We will now move on to a $q$-deformation of this picture, 
which will eventually lead us to directed polymers in random media.



\section{The $(q,t)$-generalization} 
\label{sec:_q_t_generalization}

\subsection{Historical remarks: deformations 
of Schur polynomials} 
\label{sub:historical_remarks}

The developments of the previous sections
were heavily based on properties
of the Schur polynomials
\begin{align*}
	s_\la(x_1,\ldots,x_N)
	&=\frac{\det\big[x_i^{N+\la_j-j}\big]_{i,j=1}^{N}}
	{\det\big[x_i^{N-j}\big]_{i,j=1}^{N}}
	\\&=\sum_{\la^{(1)}\prec\la^{(2)}\prec 
	\ldots\prec\la^{(N)}}
	x_1^{|\la^{(1)}|}
	x_2^{|\la^{(2)}|-|\la^{(1)}|}\ldots
	x_N^{|\la^{(N)}|-|\la^{(N-1)}|}.
\end{align*}
We would like to 
add parameters to the theory.

It is 
easy to deform 
(=~add parameters to)
our model viewed as a probabilistic
object. However, most such deformations would
lack solvability properties of the 
original model based on Schur polynomials.
The reason is that 
the Schur polynomials
are algebraic objects, and 
algebraic structures 
(in contrast with probabilistic ones)
are usually very rigid.
Thus, to find meaningful
(solvable) deformations
of the model 
requires nontrivial 
algebraic work.

Historically, first two different one-parameter 
deformations of the Schur polynomials
were suggested: around 1960
by algebraists Philip Hall and 
D.E.~Littlewood,\footnote{This is not 
the most famous mathematician with this last name,
that would be J.E.~Littlewood.}
and around 1970 by 
a statistician Henry Jack.

The \emph{Hall--Littlewood polynomials} naturally arose in 
finite group theory and were later
shown to be indispensable in representation
theory of $GL(n)$
over finite and $p$-adic fields.

The \emph{Jack polynomials}
extrapolated the so-called zonal spherical functions
arising in harmonic analysis on
Riemannian symmetric spaces 
from three
distinguished parameter values
that correspond to spaces over $\R,\C$, and $\mathbb{H}$.
They are also known as eigenfunctions
of the trigonometric Calogero--Sutherland integrable system.

In mid-1980's, in a remarkable development
Ian Macdonald united the two deformations
into a two-parameter deformation
known as \emph{Macdonald polynomials}.
The two parameters are traditionally denoted as $q$
and $t$. We will soon set $t$ to $0$, 
so it will not interfere with the time variable
in our Markov processes. 
The Hall--Littlewood polynomials arise 
when $q=0$, and the Jack polynomials correspond
to the limit regime $t=q^{\theta}\to1$, where $\theta>0$.
Schur polynomials correspond to $q=t$. 
Other significant values are: 
Schur's Q-functions (for $q=0$, $t=-1$);
monomial symmetric functions ($q=0$, $t=1$);
and (the most important for us)
$q$-Whittaker functions arising for $t=0$.


\subsection{Definition of Macdonald polynomials} 
\label{sub:definition_of_macdonald_polynomials}

The shortest way\footnote{The 
exposition below is very brief; a much more detailed one can 
be found in \cite[Ch. VI]{Macdonald1995}.} 
to define Macdonald polynomials
is to say that these are elements
of $\mathbb{Q}(q,t)[x_1,\ldots,x_N]^{S(N)}$
(this is the algebra of symmetric polynomials
in variables $x_1,\ldots,x_N$ whose 
coefficients are rational functions in $q$ and $t$),
that
diagonalize the following
first order $q$-difference operator:
\begin{align}\label{Macdonald_difference_operator}
	\begin{array}{rcl}
		\D^{(1)}&=&\displaystyle\sum_{i=1}^{N}
		\left[\prod_{1\le i<j\le N}\frac{1}{x_i-x_j}
		\T_{t,x_i}
		\prod_{1\le i<j\le N}({x_i-x_j})
		\right]\T_{q,x_i}\\&=&\displaystyle
		\sum_{i=1}^{N}\prod_{j\ne i}\frac{tx_i-x_j}{x_i-x_j}
		\T_{q,x_i},
	\end{array}
\end{align}
where, as before, $(\T_{q,x}f)(x)=f(qx)$.
It is immediately recognized 
as a deformation of the $q=t$ operator
\eqref{D1_q}
from \S \ref{ssub:general_recipe}.

The operator $\D^{(1)}$ from \eqref{Macdonald_difference_operator}
is called the \emph{first Macdonald difference operator}.
There are also higher order ones,
\begin{align}\label{k_Macdonald}
	\D^{(k)}=\sum_{\scriptstyle I\subset\{1,2,\ldots,N\}\atop
	\scriptstyle |I|=k}
	\prod_{\scriptstyle i\in I\atop
	\scriptstyle j\notin I}\frac{tx_i-x_j}{x_i-x_j}
	\prod_{i\in I}\T_{q,x_i}.
\end{align}
The operators $\D^{(k)}$ are diagonalized 
by the same polynomial basis \cite[Ch. VI]{Macdonald1995}. 

As Schur polynomials, the Macdonald polynomials
in $N$ variables are paramet\-rized by 
$\la=(\la_1\ge \ldots\ge\la_N)$.
We denote the (\emph{monic}, i.e., with 
coefficient 1 of the lexicographically largest monomial, 
which is $x_1^{\la_1}x_2^{\la_2}\ldots x_N^{\la_N}$) 
Macdonald
polynomials by $P_\la$. They satisfy
\begin{align}\label{Macdonald_eigenrelation}
	\D^{(k)}P_\la=e_k(q^{\la_1}t^{N-1},q^{\la_2}t^{N-2},
	\ldots,q^{\la_N})P_\la,
\end{align}
where 
\begin{align*}
	e_k(y_1,\ldots,y_N)=
	\sum_{1\le i_1<i_2< \ldots<i_k\le N}
	y_{i_1}y_{i_2} \ldots y_{i_k}
\end{align*}
are the elementary symmetric polynomials.
The eigenrelation \eqref{Macdonald_eigenrelation}
for $k=1$ and $q=t$
was used in \S \ref{sec:lozenge_tilings_and_representation_theory}
to compute the density function of the vertical lozenges.


\subsection{$q$-Whittaker facts} 
\label{sub:_q_whittaker_facts}

Developing the (beautiful)
theory of Macdonald polynomials
requires significant efforts, and we will not pursue
this here. 
An excellent resource is the Macdonald's 
book \cite{Macdonald1995}. Instead, we 
will focus on the $q$-Whittaker ($t=0$)
case, where, for a story parallel to the 
Schur case 
(\S\S \ref{sec:lozenge_tilings_and_representation_theory}--\ref{sec:markov_dynamics}), we need the following facts.

\begin{proposition}[$q$-analogue of Lemma 
\ref{lemma:Schur_branching}]\label{prop:fact1_qWhit}
	For any $\la=(\la_1\ge \ldots\ge\la_N)\in\Z^{N}$,
	we have
	\begin{align*}
		P_\la(x_1,\ldots,x_N)=
		\sum_{\mu\colon\mu\prec\la}
		P_\mu(x_1,\ldots,x_{N-1})
		\,\frac{\prod_{i=1}^{N-1}(\la_i-\la_{i+1})!_q}
		{\prod_{i=1}^{N-1}
		(\la_i-\mu_i)!_{q}(\mu_i-\la_{i+1})!_q}\,
		x_N^{|\la|-|\mu|},
	\end{align*}
	where $k!_q=\frac{(1-q)(1-q^{2})\ldots(1-q^{k})}{(1-q)^{k}}$
	is the $q$-analogue of the factorial.
\end{proposition}

We thus see that the interlacing structure
(Gelfand--Tsetlin schemes \eqref{GT_scheme})
remains intact, but the 
Gibbs property is $q$-deformed.
We will now say that a probability measure
on Gelfand--Tsetlin schemes 
$\boldsymbol\la=(\la^{(1)}\prec \ldots\prec\la^{(N)})$
is \emph{Gibbs} if for any 
$2\le h\le N$, 
\begin{align*}&
	\Prob\big\{
	\la^{(1)},\ldots,\la^{(h-1)}\,|\,\la^{(h)}
	\big\}\\&\hspace{60pt}=
	\La^{h}_{h-1}(\la^{(h)},\la^{(h-1)})
	\La^{h-1}_{h-2}(\la^{(h-1)},\la^{(h-2)})
	\ldots
	\La^{2}_{1}(\la^{(2)},\la^{(1)}),
\end{align*}
with the $q$-deformed stochastic links 
\begin{align*}
	\La^{h}_{h-1}(\la,\mu)=
	\frac{P_\mu(1,\ldots,1)}
	{P_\la(1,\ldots,1)}\cdot 
	\frac{\prod_{i=1}^{N-1}(\la_i-\la_{i+1})!_q}
	{\prod_{i=1}^{N-1}
	(\la_i-\mu_i)!_{q}(\mu_i-\la_{i+1})!_q}.
\end{align*}
Recall that in the Schur case 
we had simply $\dfrac{s_\mu(1,\ldots,1)}
{s_\la(1,\ldots,1)}$ 
which was then explicitly evaluated in 
\eqref{links}.

\begin{proposition}[$q$-analogue of the limit $ab/c\to t$ and 
$L^{(h)}_{\mathrm{Poisson}}$ of \S \ref{sub:dyson_brownian_motion}]
\label{prop:fact2_qWhit}
	If we define the coefficients 
	$\Prob_{t,h}\{\mu\}$ by expanding\footnote{Such a decomposition exists 
	as $\{P_\mu\}_{\mu_1\ge \ldots\ge\mu_h\ge0}$
	form a basis in the linear space of symmetric polynomials
	in $x_1,\ldots,x_h$. 
	Thus, \eqref{q-Whit_expansion} is a variant of the 
	Taylor expansion.}
	\begin{align}\label{q-Whit_expansion}
		\prod_{i=1}^{h}e^{t(x_i-1)}
		=\sum_{\mu_1\ge \ldots\ge\mu_h}	
		\Prob_{t,h}\{\mu\}\,
		\frac{P_\mu(x_1,\ldots,x_h)}{P_\mu(1,\ldots,1)},
	\end{align}
	then $\Prob_{t,h}\{\mu\}\ge0$ 
	for any $\mu$. Moreover, these probability
	measures on $\{\mu_1\ge \ldots\ge\mu_h\ge0\}$
	are time $t$ distributions 
	of a jump Markov process
	with jump rates 
	\begin{align}\label{q-Whit_generator}
		L^{(h;q)}_{\mathrm{Poisson}}
		(\la\mapsto\la+\de_j)=
		\frac{P_{\la+\de_j}(1,\ldots,1)}
		{P_{\la}(1,\ldots,1)}
		\cdot (1-q^{\la_{j-1}-\la_j}),\qquad
		1\le j\le h,
	\end{align}
	where for $j=1$ the last factor is omitted.
\end{proposition}
A proof 
(in the more general $(q,t)$-setting) can be
found in \cite[\S2.3]{BorodinCorwin2011Macdonald}.
In the Schur case, \eqref{q-Whit_generator}
reduces to the ratio of Vandermonde determinants,
cf. \eqref{Poisson_generator} above.


\subsection{$q$-deformed Markov dynamics} 
\label{sub:_q_deformed_markov_dynamics}

Propositions \ref{prop:fact1_qWhit} and \ref{prop:fact2_qWhit}
give us sufficient data to 
run the same search (as in \S \ref{sec:markov_dynamics})
for nearest neighbor Markov
processes which preserve 
the Gibbs measures,
and which on each level are described by 
$L^{(h;q)}_{\mathrm{Poisson}}$ \eqref{q-Whit_generator}.
This immediately leads to the following $q$-analogue
of Theorem \ref{thm:Schur_system_of_eqns}:
\begin{theorem}[\cite{BorodinPetrov2013NN}]\label{thm:q_dyn_main_thm}
	For any $\la^{(h-1)}\prec\la^{(h)}$, define
	\begin{align*}
		T_i(\la^{(h-1)},\la^{(h)})&:=
		\frac{\big(1-q^{\la^{(h-1)}_i-\la^{(h)}_{i+1}}\big)
		\big(1-q^{\la^{(h-1)}_{i-1}-\la^{(h-1)}_{i}+1}\big)}
		{\big(1-q^{\la^{(h)}_{i}-\la^{(h-1)}_{i}+1}\big)},
		\\
		S_j(\la^{(h-1)},\la^{(h)})&:=
		\frac{\big(1-q^{\la^{(h-1)}_{j-1}-\la^{(h)}_{j}}\big)
		\big(1-q^{\la^{(h)}_{j}-\la^{(h)}_{j+1}+1}\big)}
		{\big(1-q^{\la^{(h)}_{j}-\la^{(h-1)}_{j}+1}\big)},
	\end{align*}
	with $1\le i\le k-1$, $1\le j\le k$. If there
	is a $\la^{(\,\cdot\,)}_{\,\cdot\,}$ whose indices make no sense,
	then the corresponding factor is omitted:
	\begin{align*}
		T_1(\la^{(h-1)},\la^{(h)})&
		=\frac{1-q^{\la^{(h-1)}_1-\la^{(h)}_2}}
		{1-q^{\la^{(h)}_1-\la^{(h-1)}_1+1}},
		\quad
		S_1(\la^{(h-1)},\la^{(h)})
		=\frac{1-q^{\la^{(h)}_1-\la^{(h)}_2+1}}
		{1-q^{\la^{(h)}_1-\la^{(h-1)}_1+1}},\\&
		\hspace{40pt}
		S_h(\la^{(h-1)},\la^{(h)})
		=1-q^{\la^{(h-1)}_{h-1}-\la^{(h)}_h}.
	\end{align*}
	Let 
	$\big\{j_1+1<j_2+1<\ldots<j_\kappa+1\big\}$
	be all the indices such that 
	particle $\la^{(h)}_{j_m+1}$ 
	is free to move, i.e., $\la^{(h)}_{j_m+1}<\la^{(h-1)}_{m}$.
	Then
	\begin{align*}
		T_{j_{m+1}}r_{j_{m+1}}+T_{j_m}l_{j_m}+w_{j_m+1}=S_{j_m+1},
		\qquad 1\le m\le\kappa,
	\end{align*}
	with agreement $r_{j_{\kappa+1}}=l_{j_0}=0$,
	and also $T_{h}=0$.
	Solving these equations for all 
	pairs $\la^{(h-1)}\prec\la^{(h)}$
	is equivalent
	to constructing nearest neighbor Markov dynamics
	satisfying the $q$-versions of conditions
	{\rm{}(I)--(III)\/} of \S \ref{sub:an_example_of_a_two_dimensional_dynamics}.
\end{theorem}


\subsection{Examples of $q$-deformed two-dimensional dynamics} 
\label{sub:examples_of_q_deformed_two_dimensional_dynamics}

Using Theorem \ref{thm:q_dyn_main_thm}, 
we can now explore the same three examples as in 
\S \ref{sub:examples_of_two}:

\medskip

\noindent
\textbf{Example 1.}
We enforce the almost sure move propagation
(i.e., $r_j+l_j\equiv1$),
and also 
\begin{align*}
	w_j=\begin{cases}
		1,&j=1;\\
		0,&\mbox{otherwise}.
	\end{cases}
\end{align*}
This gives a unique solution
\begin{align*}
	r_j=\frac{S_1+\ldots+S_{j}-T_{1}-\ldots-T_{j-1}}{T_j},
	\qquad l_j=1-r_j,
\end{align*}
for all $j$ such that $\la^{(h)}_{j+1}$ is free.
In fact, this expression telescopes to give
\begin{align*}
	r_j=\begin{cases}
		q^{\la^{(h)}_{1}-\la^{(h-1)}_1},& j=1;\\
		q^{\la^{(h)}_{j}-\la^{(h-1)}_{j}}
		\dfrac{1-q^{\la^{(h-1)}_{j-1}-\la^{(h)}_j}}
		{1-q^{\la^{(h-1)}_{j-1}-\la^{(h-1)}_j}},&
		2\le j\le h.
	\end{cases}
\end{align*}
We observe that for $0\le q<1$, all the
probabilities $r_j$ and $l_j$ are nonnegative,
and the projection to the rightmost particles
$\{\la^{(h)}_{1}\}_{h=1}^{N}$
is \emph{Markovian}.
In the shifted variables
$y_h=\la^{(h)}_{1}+h$, 
it can be described as follows:
Each particle jumps to the 
right by 1
independently with Poisson clock of rate 1.
If the $j$th particle moved, it triggers the move of $(j+1)$st
one with probability $q^{\gap}$, where
$\gap$
is the number of empty 
spots in front of the $j$th particle before the move 
(which in its turn may trigger the move of the $(j+2)$nd particle, etc.).
Note that the probability $q^{\gap}$
is~1 if $\gap=0$. We call this particle system
the \emph{$q$-PushTASEP}, it was first introduced in \cite{BorodinPetrov2013NN}.
Its generalization (called \emph{$q$-PushASEP})
with particles moving in both directions
can be found in \cite{CorwinPetrov2013}.

\medskip

\noindent
\textbf{Example 2.}
Now we again enforce $l_j+r_j\equiv 1$, and 
\begin{align*}
	w_j=\begin{cases}
		1,&j=h;\\
		0,&\mbox{otherwise}.
	\end{cases}
\end{align*}
This gives 
\begin{align*}
	r_j=1+\frac{q^{\la^{(h-1)}_{j}-\la^{(h)}_{j+1}}}{T_j},
	\qquad
	l_j=1-r_j,
\end{align*}
for all $j$ such that $\la^{(h)}_{j+1}$ is free.
Obviously, this gives negative probabilities, 
and we do not pursue this example further.

\medskip

\noindent
\textbf{Example 3.}
Here we enforce $l_j=r_j\equiv0$.
This clearly gives $w_j=S_j$, and 
for $0\le q<1$
this is a well-defined
Markov process without
long-range interactions.
It was first constructed in 
\cite{BorodinCorwin2011Macdonald}, and it is closely
related to the
$q$-Boson stochastic particle system of 
\cite{SasamotoWadati1998}, 
see also
\cite{BorodinCorwinSasamoto2012},
\cite{BorodinCorwinPetrovSasamoto2013}.
While the projection of this process to the rightmost
particles $\{\la^{(h)}_{1}\}_{h=1}^{N}$
does not appear to be Markovian (because 
$S_1$ depends on $\la_2^{(h)}$), the projection 
to the leftmost particles 
$\{\la^{(h)}_{h}\}_{h=1}^{N}$
\emph{is} Markovian.
In the shifted coordinates $y_h=\la^{(h)}_{h}-h$
it can be described as follows:
Each particle jumps to the right by 1 independently
of the others with rate $1-q^{\gap}$, where 
$\gap$ is (as before)
the number of empty spaces 
in front of $y_h$ before the jump.
Note that this rate vanishes when $\gap=0$,
which correspond to a TASEP-like blocking of the move.
We call this interacting particle system
the \emph{$q$-TASEP}. 

\medskip

Obviously, as $q\to0$, the 
$q$-PushTASEP turns into the usual PushTASEP, and $q$-TASEP
becomes the usual TASEP. 


\subsection{Conclusion} 
\label{sub:conclusion}

We have thus obtained $q$-deformations of the random
growth models from the Schur case.
Our next task will be 
to investigate their asymptotic
behavior at large times.



\section{Asymptotics of $q$-deformed growth models} 
\label{sec:asymptotics_of_q_deformed_growth_models}

Our main tool in studying asymptotics will be the Macdonald
difference operators 
(\S \ref{sub:definition_of_macdonald_polynomials}).

\subsection{A contour integral formula for expectations of observables} 
\label{sub:a_contour_integral_formula_for_expectations_of_observables}

By $\Prob_{t,h}$ we mean the measure defined 
in Proposition \ref{prop:fact2_qWhit}.
\begin{proposition}
	For any $1\le h\le N$, 
	\begin{align}
	\begin{array}{ll}
		&\displaystyle
		\sum_{\la_1\ge \la_2\ge \ldots\ge \la_N\ge0}
		q^{\la_N+\ldots+\la_{N-k+1}}
		\Prob_{t,N}\{\la\}
		\\&\displaystyle\hspace{25pt}=
		\frac{(-1)^{\frac{k(k-1)}{2}}}
		{(2\pi\i)^{k}k!}
		\oint \ldots\oint 
		\prod_{1\le A<B\le k}
		(z_A-z_B)^{2}
		\prod_{j=1}^{k}\frac{e^{(q-1)tz_j}}{(1-z_j)^{N}}
		\frac{dz_j}{z_j^{k}},
	\end{array}
	\label{qWhit_integral_formula}
	\end{align}
	where all the integrals 
	are taken over small positively oriented 
	closed contours around 1.
\end{proposition}
\begin{proof}
	In the proof we need to use the 
	second Macdonald parameter $t\ne 0$.
	For this, let us in this proof 
	denote the time 
	variable by $\tau$ to avoid the confusion.

	We apply the $k$th order Macdonald operator 
	$\D^{(k)}$ \eqref{k_Macdonald} to the series expansion
	\eqref{q-Whit_expansion}
	defining our measures, which now looks as
	\begin{align}\label{expansion_with_tau}
		e^{\tau\sum_{i=1}^{N}(x_i-1)}
		=\sum_{\la_1\ge \ldots\ge\la_N\ge0}	
		\Prob_{\tau,h}\{\la\}\,
		\frac{P_\la(x_1,\ldots,x_h)}{P_\la(1,\ldots,1)}.
	\end{align}
	We then replace the sum in the left-hand side 
	by the residue expansion
	of the integral
	(see \cite[\S2.2.3]{BorodinCorwin2011Macdonald} for more detail)
	\begin{align*}&
		\frac{\D^{(k)}e^{\tau\sum_{i=1}^{N}(x_i-1)}}
		{e^{\tau\sum_{i=1}^{N}(x_i-1)}}
		=\frac{1}{(2\pi\i)^{k}k!}
		\oint \ldots\oint
		\frac{\prod\limits_{1\le A<B\le k}(tz_A-tz_B)(z_B-z_A)}
		{\prod\limits_{A,B=1}^{k}(tz_A-z_B)}
		\times\\&\hspace{180pt}\times\prod_{j=1}^{k}
		\left(
		\prod_{m=1}^{N}\frac{tz_j-x_m}{z_j-x_m}
		e^{(q-1)\tau z_j}dz_j
		\right),
	\end{align*}
	where the contours encircle $\{x_1,\ldots,x_N\}$
	and no other poles 
	(i.e., the residues are taken at
	$z_j=x_{m_j}$, $j=1,\ldots,k$, $1\le m_j\le N$).
	Note that in fact
	\begin{align*}
		\frac{\prod_{1\le A<B\le k}(tz_A-tz_B)(z_B-z_A)}
		{\prod_{A,B=1}^{k}(tz_A-z_B)}=
		\det\left[\frac{1}{tz_A-z_B}\right]_{A,B=1}^{k}
	\end{align*}
	via the Cauchy determinant formula.

	In the right-hand side of 
	\eqref{expansion_with_tau} we use the 
	eigenrelation 
	(see \S \ref{sub:definition_of_macdonald_polynomials}):
	\begin{align*}
		\D^{(k)}P_\la&
		=e_k(q^{\la_1}t^{N-1},q^{\la_2}t^{N-2},
		\ldots,q^{\la_N})P_\la
		\\&=
		\left(
		q^{\la_N+\la_{N-1}+\ldots+\la_{N-k+1}}t^{\frac{k(k-1)}2}
		{}+{}\text{higher powers of $t$}
		\right)P_\la.
	\end{align*}
	We then divide both sides of \eqref{expansion_with_tau}
	by $t^{\frac{k(k-1)}2}$,
	take the limit as $t\to0$, and also set $x_j=1$.
\end{proof}


There are two simple limit transitions 
that one can observe in the right-hand side of 
the contour integral formula
\eqref{qWhit_integral_formula}. We consider them in 
\S \ref{sub:gaussian_limit} and 
\S \ref{sub:polymer_limit} below.

\subsection{Gaussian limit} 
\label{sub:gaussian_limit}

The first limit regime is $q=e^{-\eps}\to1$, 
$t=\tau\eps^{-1}\to\infty$, $z_j$'s do not change.
Looking at the left-hand side of 
\eqref{qWhit_integral_formula}, which is 
$\E\left(q^{\la_N+\ldots+\la_{N-k+1}}\right)$, 
it is natural to expect that each $\la_j$
grows as $\eps^{-1}$ so that the 
quantities $q^{\la_j}$
have finite limits
(which may still be random variables).
Looking at higher powers of Macdonald
operators indeed reveals that this is a Gaussian limit:
$\eps\la_j$ has a law of large numbers with Gaussian
fluctuations 
of size $\eps^{-1/2}$, and the Markov dynamics
we constructed converge to Gaussian processes.
We do not pursue this limit regime here,
its detailed exposition will appear in
\cite{BorodinCorwinFerrari2014}.
Another, structurally similar appearance of Gaussian processes can be found in
\cite{BorodinGorin2013beta}.


\subsection{Polymer limit} 
\label{sub:polymer_limit}

The second limit is a bit more complicated.
We again take 
$q=e^{-\eps}\to1$, but now
$t=\tau\eps^{-2}\to\infty$
(i.e., we wait for a longer time than in the Gaussian limit of 
\S \ref{sub:gaussian_limit}).
Then to see a nontrivial limit, we have to 
take the $z_j$'s of distance $O(\eps)$
from 1: $z_j=1+\eps w_j$.
Then we have
\begin{gather*}
	e^{(q-1)t z_j}=e^{(-\eps+\eps^{2}/2-\ldots)\tau \eps^{-2}
	(1+\eps w_j)}
	=
	e^{-\tau \eps^{-1}}e^{\tau/2-\tau w_j},\\
	(1-z_j)^{N}=
	(-\eps w_j)^{N}=(-1)^{N}\eps^{N}w_j^{N},\\
	\prod_{1\le A<B\le k}
	(z_A-z_B)^{2}=
	\eps^{k(k-1)}\prod_{1\le A<B\le k}(w_A-w_B)^{2},\\
	\prod_{j=1}^{N}dz_j=\eps^{k}\prod_{j=1}^{N}dw_j,
\end{gather*}
and we see that the right-hand side of \eqref{qWhit_integral_formula}
becomes 
$e^{-\tau k\eps^{-1}}\eps^{k(k-1)-kN+k}$
times an asymptotically finite expression.
To figure out the limiting behavior of
$\la_N+\ldots+\la_{N-k+1}$, we now have to
take $\log_{q}$ of this expression, or take the 
natural logarithm and multiply by $-\eps^{-1}$.
This gives 
\begin{align*}
	\la_N+\ldots+\la_{N-k+1}
	\sim
	\tau k\eps^{-2}+(-kN+k^{2})\frac{\ln\eps^{-1}}{\eps}
	+\frac{R_{N,k}}{\eps},
\end{align*}
where the remainder $R_{N,k}$
is supposed to be a finite random variable.
Equivalently, 
\begin{align}\label{laj_polymer_scaling}
	\la_j=\tau\eps^{-2}-(N+1-2j)\frac{\ln\eps}{\eps}
	+\frac{T_{N,j}}{\eps},
	\qquad1\le j\le N,
\end{align}
with some limiting random
variables $T_{N,j}$.
(Note that at this moment this is simply a \emph{guess}!)

We can now test what is happening with our dynamics
under this conjectural scaling. For example, 
consider the $q$-PushTASEP. 
The asymptotics of the pushing probability is
\begin{align*}
	q^{\la^{(h)}_1-\la^{(h-1)}_1}\sim
	e^{-\eps\left(-\frac{\ln\eps}{\eps}+
	\frac{T_{h,1}-T_{h-1,1}}{\eps}\right)}	
	=\eps e^{T_{h-1,1}-T_{h,1}}.
\end{align*}
The increment of 
$\la^{(h)}_1$ over time
$d\tau=\eps^{2}dt$
must then be (1) the increment coming from 
its own jumps, which is 
$\eps^{-2}dt+\eps^{-1}dB_h$, where
$B_h$ is a Brownian motion, and 
(2) the increment coming from 
pushing, which is $\eps e^{T_{h-1,1}-T_{h,1}}$
times the increment of $\la^{(h-1)}_{1}$.
Collecting terms of order $\eps^{-1}$, we conclude that
\begin{align*}
	d T_{h,1}=dB_h+e^{T_{h-1,1}-T_{h,1}} d\tau,
	\qquad 1\le h\le N,
\end{align*}
where $B_1,B_2,\ldots,B_N$
are independent standard Brownian motions
(for $h=1$ the last term is omitted).
This system of stochastic differential equations (SDEs, for short) is solved by 
\begin{align}
	T_{h,1}=\log
	\int\limits_{0<s_1<\ldots<s_{h-1}<\tau}
	e^{B_1(s_1)+
	\ldots
	+\big(B_h(\tau)-B_h(s_{h-1})\big)}
	ds_1 \ldots ds_{h-1}.
	\label{polymer_integral}
\end{align}
The integral in the right-hand
side can be viewed as the logarithm
of the partition function
(i.e., the \emph{free energy})
of a semi-discrete 
Brownian polymer, see Fig.~\ref{fig:polymer},
and also \cite[Chapter 5]{BorodinCorwin2011Macdonald} for a general discussion
of directed polymers in random media.\begin{figure}[htbp]
	\begin{center}
		\begin{tabular}{cc}
			\begin{tikzpicture}
				[scale=1]
				\foreach \y in {1,2,3,4,5,6}
				{
					\draw[thick] (7,\y)--(0,\y);
				}
				\draw[thick,
					decoration={markings,mark=at position 1 with {\arrow[scale=2]{>}}},
    				postaction={decorate}] 
    				(4,1)--(7,1);
    			\foreach \y in {1,2,3}
				{
					\node at (-.35,\y) {$\y$};
				}
				\node at (-.35,6) {$N$};
				\node at (7.1,.65) {\large$\tau$};
				\node at (7.2,4.4) {$(\tau,h)$};
				\draw[line width=4, color=red, opacity=.65] (0,1)--(1.2,1)--(1.2,2)--(1.6,2)--(1.6,3)
				--(3.8,3)--(3.8,4)--(4.4,4)--(7,4);
				\draw[fill] (0,1) circle (.12);
				\draw[fill] (7,4) circle (.12);
				\newcommand{\BM}[5]{
				\draw[#4] {#5}
				\foreach \x in {1,...,#1}
				{   -- ++(#2,rand*#3)
				}
				}
				\BM{80}{0.02}{0.06}{thick,blue}{(0,6)};
				\BM{80}{0.02}{0.06}{thick,blue}{(0,5)};
				\BM{80}{0.02}{0.06}{thick,blue}{(0,4)};
			\end{tikzpicture}
		\end{tabular}
	\end{center}
	\caption{Semi-discrete 
	Brownian polymer.}
	\label{fig:polymer}
\end{figure}
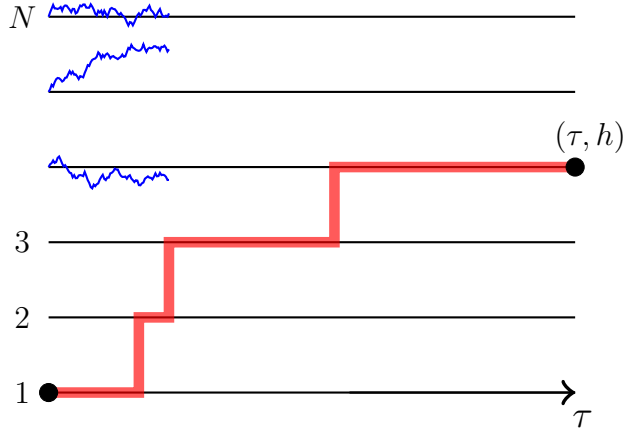
More precisely, to any Poisson-type up-right path $\phi_{(0,1)\to(\tau,h)}$
that travels from $1$ to $h$
during time $\tau$
with jumps at moments 
$0<s_1<\ldots<s_{h-1}<\tau$, assign the energy
\begin{align*}
	E(\phi_{(0,1)\to(\tau,h)})=B_1(s_1)+
	\big(B_2(s_2)-B_2(s_1)\big)
	+\ldots
	+\big(B_h(\tau)-B_h(s_{h-1})\big).
\end{align*}
Then $T_{h,1}$ is the logarithm of the 
integral of the Boltzmann factor
$$\exp\big(E(\phi_{(0,1)\to(t,h)})\big)$$
over the Lebesgue measure on all such paths
(the inverse temperature can be absorbed
into the rescaling of $\tau$ with the help of the 
Brownian scaling).

Similar empirical scaling arguments
show that the 
Markov process on Gelfand--Tsetlin
schemes of depth $N$
that lead to the 
$q$-PushTASEP (Example 1 in
\S \ref{sub:examples_of_q_deformed_two_dimensional_dynamics})
converges to 
a solution of the following system of SDEs:
\begin{align}\label{SDE_hierarchy}
	dT_{h,k}=\mathbf{1}_{k=1}dB_h
	+\mathbf{1}_{k\ne1}dT_{h-1,k-1}
	+\left(e^{T_{h-1,k}-T_{h,k}}-
	e^{T_{h-1,k-1}-T_{h,k-1}}\right)d\tau.
\end{align}

\begin{theorem}\label{thm:polymer_limit}
	Under the above scaling \eqref{laj_polymer_scaling}, 
	the measure $\Prob_{t,h}\{\la\}$
	weakly converges to a probability measure
	on arrays\footnote{Note
	that there is no interlacing
	in this limit!} $\{T_{h,k}\}\in\R^{\frac{N(N+1)}2}$
	that can be written in the form
	\begin{align}\label{polymer_hierarchy}
		T_{h,1}+\ldots+T_{h,k}=
		\log\int \ldots\int e^{E(\phi_1)+\ldots
		+E(\phi_k)}d\phi_1 \ldots d\phi_k,
	\end{align}
	the integral taken over the Lebesgue
	measure on 
	the polytope of 
	$k$-tuples
	of nonintersecting Poisson-type
	paths $\phi_1,\ldots,\phi_k$
	joining 
	$\{(0,1),(0,2),\ldots,(0,k)\}$
	with 
	$\{(\tau,h-k+1),(\tau,h-k+2),\ldots,(\tau,h)\}$,
	see Fig.~\ref{fig:nonint_polymer}.
	The limit measure
	is invariant under the flip
	$\{T_{h,k}\leftrightarrow -T_{h,h-k+1}\}$.
\end{theorem}
\begin{proof}
	The Markov dynamics related to the semi-discrete
	directed
	Brownian polymer was introduced in 
	\cite{Oconnell2009_Toda}.
	The weak convergence was proven in
	\cite[\S4]{BorodinCorwin2011Macdonald}.
\end{proof}
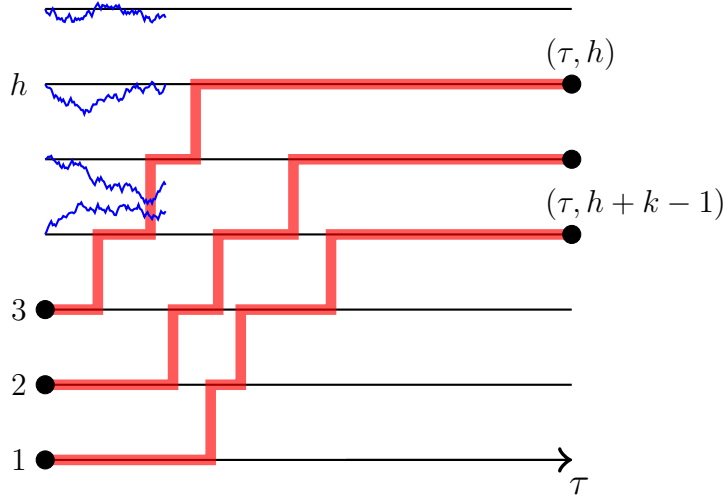
\begin{figure}[htbp]
	\begin{center}
		\begin{tikzpicture}
			[scale=1]
			\foreach \y in {1,2,3,4,5,6,7}
			{
				\draw[thick] (7,\y)--(0,\y);
			}
			\draw[thick,
				decoration={markings,mark=at position 1 with {\arrow[scale=2]{>}}},
				postaction={decorate}] 
				(4,1)--(7,1);
			\foreach \y in {1,2,3}
			{
				\node at (-.35,\y) {$\y$};
			}
			\node at (-.35,6) {$h$};
			\node at (7.1,.65) {\large$\tau$};
			\node[right] at (6.5,4.4) {$(\tau,h+k-1)$};
			\node[right] at (6.5,6.4) {$(\tau,h)$};
			\draw[line width=4, color=red, opacity=.65] (0,1)--(2.2,1)--(2.2,2)--(2.6,2)--(2.6,3)
			--(3.8,3)--(3.8,4)--(4.4,4)--(7,4);
			\draw[line width=4, color=red, opacity=.65] (0,2)--(1.7,2)--(1.7,3)--(2.3,3)--(2.3,4)
			--(3.3,4)--(3.3,5)--(7,5);
			\draw[line width=4, color=red, opacity=.65] (0,3)--(.7,3)--(.7,4)--(1.4,4)--(1.4,5)
			--(2,5)--(2,6)--(7,6);
			\draw[fill] (0,1) circle (.12);
			\draw[fill] (0,2) circle (.12);
			\draw[fill] (0,3) circle (.12);
			\draw[fill] (7,4) circle (.12);
			\draw[fill] (7,5) circle (.12);
			\draw[fill] (7,6) circle (.12);
			\newcommand{\BM}[5]{
			\draw[#4] {#5}
			\foreach \x in {1,...,#1}
			{   -- ++(#2,rand*#3)
			}
			}
			\BM{80}{0.02}{0.06}{thick,blue}{(0,7)};
			\BM{80}{0.02}{0.06}{thick,blue}{(0,6)};
			\BM{80}{0.02}{0.06}{thick,blue}{(0,5)};
			\BM{80}{0.02}{0.06}{thick,blue}{(0,4)};
		\end{tikzpicture}
	\end{center}
	\caption{Nonintersecting Poisson paths $\phi_1,\ldots,\phi_k$.}
	\label{fig:nonint_polymer}
\end{figure}

It is also known that the right-hand side 
of \eqref{polymer_hierarchy}
satisfies the system of SDEs 
\eqref{SDE_hierarchy}. Thus,
the convergence in the above theorem
should extend to a trajectory-wise statement, 
but, to our best knowledge, this has not been worked out in 
full
detail yet.

\medskip

Observe the similarity of \eqref{polymer_hierarchy}
and the row Robinson--Schensted construction of
\S \ref{sub:an_example_of_a_two_dimensional_dynamics}
(in particular, see Figures 
\ref{fig:RSK_paths} and \ref{fig:nonint_polymer}). 
In fact, \eqref{polymer_hierarchy}
arises via a \emph{geometric} lifting
of the Robinson--Schensted(--Knuth) correspondence,
see 
\cite{Oconnell2009_Toda}, \cite{COSZ2011},
\cite{OSZ2012}.

It is worth noting that other (Brownian-type) scaling limits of 
growth models discussed here are considered in \cite{GorinShkolnikov2012}, 
\cite{GorinShkolnikov2014}.



\section{Moments of $q$-Whittaker processes} 
\label{sec:moments_for_q_whittaker_processes}

From now on we focus on the asymptotic behavior
of the free energy
$T_{N,1}\stackrel{d}{=}-T_{N,N}$ 
(equality in distribution)
of the semi-discrete Brownian polymer,
see \S \ref{sub:polymer_limit} (recall that this is the 
logarithm of the polymer's partition function $e^{-T_{N,N}}$).
The free energy can be viewed as a limit of either $\la^{(N)}_{1}$
(thus, $q$-PushTASEP), or of 
$\la^{(N)}_{N}$ (corresponding to the $q$-TASEP),
and either 
one can be used for the analysis.
(Note that to obtain the polymer's partition function, 
$N$ must remain fixed.)
We will employ the $q$-TASEP,
as this is a bit more straightforward,
and there are more details on the $q$-TASEP in the literature.

\subsection{Moments of the $q$-TASEP} 
\label{sub:moments_of_the_qtasep}

We start by employing products
of first order Macdonald operators
rather than a single one to obtain moments (of all orders) of the
$q$-TASEP particle locations.
\begin{proposition}\label{prop:moments_of_the_qtasep}
	Consider the random Gelfand--Tsetlin schemes
	of depth $N\ge1$
	distributed according to 
	\begin{align*}
		\Prob_{t,N}\{\la^{(N)}\}
		\La^{N}_{N-1}(\la^{(N)},\la^{(N-1)})
		\La^{N-1}_{N-2}(\la^{(N-1)},\la^{(N-2)})
		\ldots
		\La^{2}_{1}(\la^{(2)},\la^{(1)})
	\end{align*}
	(this is the Gibbs measure 
	with the top row distributed according to 
	$\Prob_{t,N}$, see \S \ref{sub:_q_whittaker_facts}).
	Then for any $N\ge N_1\ge N_2\ge \ldots\ge N_k\ge1$, 
	\begin{align}\label{qTASEP_moments}&
		\E \Big(q^{\la^{(N_1)}_{N_1}+\ldots+\la^{(N_k)}_{N_k}}\Big)
		\\&\hspace{40pt}=\nonumber
		\frac{(-1)^k q^{\frac{k(k-1)}2}}{(2\pi\i)^{k}}
		\oint \ldots\oint \prod_{1\le A<B\le k}
		\frac{z_A-z_B}{z_A-qz_B}
		\prod_{j=1}^{k}\frac{e^{(q-1)tz_j}}{(1-z_j)^{N_j}}
		\frac{dz_j}{z_j},
	\end{align}
	where the integral is taken over 
	positively oriented,
	\emph{nested}
	contours around 1: 
	The $z_k$ contour encircles 1
	and no other poles, the 
	$z_{k-1}$ contour encircles 1
	and the contour $\{qz_{k}\}$, 
	and so on; 
	the $z_1$ contour
	encircles
	$\{1,qz_{k},qz_{k-1},\ldots,qz_2\}$,
	see Fig.~\ref{fig:qTASEP_contours}.
\end{proposition}
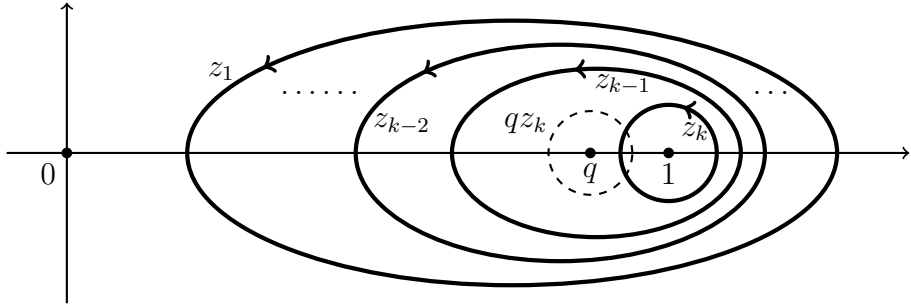
\begin{figure}[htbp]
	\begin{center}
		\begin{tikzpicture}
			[scale=8]
			\def\x{0.008}
			\def\q{0.87}
			\draw[->,thick] (-.1,0)--(1.4,0);
			\draw[->,thick] (0,-.25)--(0,.25);
			\draw[fill] (0,0) circle(\x);\node [below left] at (0,0) {$0$};
			\draw[fill] (1,0) circle(\x);\node [below] at (1,0) {$1$};
			\draw[fill] (\q,0) circle(\x);\node [below] at (\q,0) {$q$};
			\draw[ultra thick, decoration={markings,
			      mark=at position .2 with {\arrow{>}}}, 
			      postaction={decorate}] 
		    (1,0) circle(.08) node[above right] {$z_k$};
		    \draw[thick, dashed] 
		    (\q,0) circle(.08*\q)
		    node[above left, xshift=-12, yshift=3] {$qz_k$};
		    \draw[ultra thick, decoration={markings,
			    mark=at position .28 with {\arrow{>}}}, 
			    postaction={decorate}]  
			    (\q+.01,0) ellipse (.24 and .14)
			    node[yshift=26,xshift=10] {$z_{k-1}$};
			\draw[ultra thick, decoration={markings,
			    mark=at position .39 with {\arrow{>}}}, 
			    postaction={decorate}]  
			    (\q-.05,0) ellipse (.34 and .18)
			    node[yshift=11,xshift=-60] {$z_{k-2}$};
			\node at (.42,.1) {$\ldots \ldots$};
			\node at (1.17,.1) {$\ldots$};
			\draw[ultra thick, decoration={markings,
			    mark=at position .42 with {\arrow{>}}}, 
			    postaction={decorate}]  
			    (\q-.13,0) ellipse (.54 and .22)
			    node[yshift=31,xshift=-110] {$z_{1}$};
		\end{tikzpicture}
	\end{center}
	\caption{\emph{Nested} contours of integration in \eqref{qTASEP_moments}.}
	\label{fig:qTASEP_contours}
\end{figure}
\begin{proof}
	We consider $k=2$, 
	for larger $k$
	the argument is similar. 
	We start with the defining identity
	\eqref{q-Whit_expansion}:
	\begin{align}\label{proof_qTASEP_moments_1}
		e^{t\sum_{i=1}^{N_1}(x_i-1)}
		=\sum_{\la^{(N_1)}}
		\Prob_{t,N_1}\{\la^{(N_1)}\}
		\widetilde{P}_{\la^{(N_1)}}
		(x_1,\ldots,x_{N_1}),
	\end{align}
	where $\tilde P_\la$ means the 
	normalization of $P_\la$
	by itself evaluated at all $x_j\equiv1$.
	Apply the first Macdonald operator
	in $N_1$ variables to \eqref{proof_qTASEP_moments_1}
	(note that the second Macdonald parameter $t$ is zero, 
	and it has no relation to the time $t$ in \eqref{proof_qTASEP_moments_1}).
	This operator has the form
	\begin{align*}
		\D^{(1)}_{N_1}=\sum_{i=1}^{N_1}
		\prod_{j\ne i}
		\frac{-x_j}{x_i-x_j}
		\T_{q,x_i}.
	\end{align*}
	Because of the eigenrelation
	$D^{(1)}_{N_1}\tilde P_{\la^{(N_1)}}=q^{\la^{(N_1)}_{N_1}}P_{\la^{(N_1)}}$,
	in the right-hand side of \eqref{proof_qTASEP_moments_1}
	we observe
	\begin{align*}
		\sum_{\la^{(N_1)}}
		q^{\la^{(N_1)}_{N_1}}
		\Prob_{t,N_1}\{\la^{(N_1)}\}
		\widetilde{P}_{\la^{(N_1)}}
		(x_1,\ldots,x_{N_1}).
	\end{align*}
	On the other hand, the left-hand side of
	\eqref{proof_qTASEP_moments_1} is, 
	by residue expansion, equal to the integral
	\begin{align}
		-\frac{x_1,\ldots,x_{N_1}}{2\pi\i}
		\oint \frac{e^{t\sum_{i=1}^{N_1}(x_i-1)}e^{(q-1)tz}}
		{(x_1-z)\ldots (x_{N_1}-z)}
		\frac{dz}{z},
		\label{proof_qTASEP_moments_2}
	\end{align}
	over a positively oriented contour
	around the simple poles $x_1,\ldots,x_{N_1}$.

	The above argument works for applying
	$\D^{(1)}_{N_1}$
	to any multiplicative function
	$F(x_1)\ldots F(x_{N_1})$ with the exponentials 
	$e^{t\sum_{i=1}^{N_1}(x_i-1)}e^{(q-1)tz}$ in 
	\eqref{proof_qTASEP_moments_2}
	replaced by 
	$F(x_1)\ldots F(x_{N_1})\frac{F(qz)}{F(z)}$.
	
	In the next step,
	we apply $\D^{(1)}_{N_2}$
	to the result 
	of application of $\D^{(1)}_{N_1}$
	to \eqref{proof_qTASEP_moments_1}
	after setting
	$x_{N_2+1}=\ldots=x_{N_1}=1$
	(the order $N_2\le N_1$ is important).
	For the right-hand side of the resulting expression,
	we use Proposition \ref{prop:fact1_qWhit}
	that gives
	\begin{align*}&
		P_{\la^{(N_1)}}(x_1,\ldots,x_{N_2},
		\underbrace{1,\ldots,1}_{N_1-N_2})
		\\&\hspace{30pt}=\sum_{\la^{(N_2)}}
		\Prob_{t,N_1}\{\la^{(N_1)}\}
		\La^{N_1}_{N_2}(\la^{(N_1)},\la^{(N_2)})
		\tilde P_{\la^{(N_2)}}(x_1,\ldots,x_{N_2}).
	\end{align*}
	Here $\La^{N_1}_{N_2}$ means the matrix product
	$\La^{N_1}_{N_1-1}\ldots \La^{N_2+1}_{N_2}$. 
	
	This implies, together with 
	the eigenrelation
	$\D^{(1)}_{N_2}P_{\la^{(N_2)}}=q^{\la^{(N_2)}_{N_2}}P_{\la^{(N_2)}}$,
	that the 
	right-hand side (after setting $x_1=\ldots=x_{N_2}=1$) gives 
	$\E q^{\la^{(N_1)}_{N_1}+\la^{(N_2)}_{N_2}}$.
	On the other hand, in the left-hand side, the $x$-dependence 
	in \eqref{proof_qTASEP_moments_2},
	after setting $x_{N_2+1}=\ldots=x_{N_1}=1$,
	is in the form of $G(x_1)\ldots G(x_{N_2})$
	with $G(x)=\frac{x}{x-z}e^{t(x-1)}$.
	Hence, we can apply the same residue
	expansion (for computing the application of 
	$\D^{(1)}_{N_2}$ to \eqref{proof_qTASEP_moments_2}) 
	using the fact that 
	\begin{align*}
		\frac{G(qw)}{G(w)}=q
		\frac{z-w}{z-qw}e^{t(q-1)w}.
	\end{align*}
	Here $w$ is the new integration variable 
	whose contour has to encircle 
	$x_1,\ldots,x_{N_2}$, but not any other poles
	(in particular, not $q^{-1}z$).
	Renaming $(z,w)\mapsto (z_1,z_2)$,
	we obtain the desired formula for $k=2$.
	For larger $k$ the proof is similar.
\end{proof}

The above iterative argument is due to V. Gorin, see 
\cite{BCGS2013} for a more general version.
The original proof of \eqref{qTASEP_moments} in
\cite{BorodinCorwinSasamoto2012}
involved a discretization of the 
quantum delta Bose gas, see also 
\cite{BorodinCorwinPetrovSasamoto2013} and
\S \ref{sub:moments_of_semi_discrete_brownian_polymer} below.


\subsection{Moments of the semi-discrete Brownian polymer} 
\label{sub:moments_of_semi_discrete_brownian_polymer}

Since we already know the 
scaling which takes us from $\la^{(N)}_N$
to the polymer partition function (\S \ref{sub:polymer_limit}), we can immediately 
do the limit in the integral of 
Proposition \ref{prop:moments_of_the_qtasep}.
This is very similar to the limit that we took in 
\S \ref{sub:polymer_limit}. 
That is, let us use
\begin{align*}
	t=\tau \eps^{-2},\qquad
	q=e^{-\eps}, 
	\qquad z_j=1-\eps w_j.
\end{align*}
This leads to the following formula for the 
moments of the polymer partition function:
For $N_1\ge N_2\ge \ldots\ge N_k\ge0$
(as in Proposition \ref{prop:moments_of_the_qtasep}), we have
\begin{align}\label{polymer_moments}
	\E \Big(e^{-T_{N_1,N_1}- \ldots -T_{N_k,N_k}}\Big)
	=
	\frac{e^{\tau k/2}}{(2\pi\i)^{k}},
	\oint \ldots\oint
	\prod_{1\le A<B\le k}
	\frac{w_A-w_B}{w_A-w_B-1}
	\prod_{j=1}^{k}
	\frac{e^{\tau w_j}}{w_j^{N_j}}
	dw_j,
\end{align}
where the integrals 
are now over nested contours around $0$:
the $w_k$ contour contains only $0$,
the $w_{k-1}$ contour contains $\{w_k+1\}$ and $0$,
and so on;
the $w_1$ contour contains $\{0,w_k+1,w_{k-1}+1,\ldots,w_{2}+1\}$.
This limit transition 
from \eqref{qTASEP_moments}, however, is \emph{not} a proof of the formula 
\eqref{polymer_moments}.
Indeed, Theorem \ref{thm:polymer_limit} only claims weak convergence,
and we have exponential moments under the expectation
(that is, expectations of \emph{unbounded} functions).

\medskip
\noindent
\emph{Sketch of the proof of \eqref{polymer_moments}.}
Observe that if we define
(for \emph{any}, not necessarily ordered $N_1,\ldots,N_k\ge0$)
\begin{align*}
	F(\tau;N_1,\ldots,N_k):=
	\frac{1}{(2\pi\i)^{k}}
	\oint \ldots\oint
	\prod_{1\le A<B\le k}\frac{w_A-w_B}{w_A-w_B-1}
	\prod_{j=1}^{k}\frac{e^{\tau w_j}}{w_j^{N_j}}dw_j
\end{align*}
(with contours as in \eqref{polymer_moments}),
then
\begin{align*}
	&
	\frac{d}{d\tau}
	F(\tau;N_1,N_2\ldots,N_k)=
	F(\tau;N_1-1,N_2,\ldots,N_k)
	\\&\hspace{50pt}+
	F(\tau;N_1,N_2-1,\ldots,N_k)
	+\ldots+
	F(\tau;N_1,N_2,\ldots,N_k-1).
\end{align*}
Further, for $N_k=0$, $F$ vanishes (because there are no poles inside the smallest contour),
and for $N_i=N_{i+1}$, 
\begin{align*}&
	F(\tau;\ldots,N_i-1,N_{i+1},\ldots)-
	F(\tau;\ldots,N_i,N_{i+1}-1,\ldots)
	\\&\hspace{170pt}-F(\tau;\ldots,N_i,N_{i+1},\ldots)=0,
\end{align*}
because when we write out the integral for this linear 
combination, the integrand will be skew-symmetric
in $w_i$ and $w_{i+1}$, 
and the two corresponding contours can be taken to be the same
(the obstacle to both these properties
in \eqref{polymer_moments} is the factor
$(w_i-w_{i+1}-1)$, 
and the linear combination above exactly cancels this factor out).
These properties together
with initial condition
$F(0;N_1,\ldots,N_k)=\mathbf{1}_{N_1=\ldots=N_k=1}$
uniquely 
determine $F(\tau;N_1,N_2\ldots,N_k)$
for $N_1\ge \ldots\ge N_k\ge0$.
Thus, it suffices to check that 
the moments of the polymer partition function
(see \eqref{polymer_hierarchy}) satisfy the 
same properties. This
check is fairly straightforward.
\qed
\medskip

A discussion of how the evolution equation and the boundary conditions at
$N_i=N_{i+1}$
used above relate to a discrete quantum
Bose gas can be found in 
\cite{BorodinCorwin2011Macdonald},
\cite{BorodinCorwinSasamoto2012},
\cite{BorodinCorwinPetrovSasamoto2013}.


\subsection{Continuous Brownian polymer} 
\label{sub:continuous_brownian_polymer}

There is a further limit that takes the semi-discrete
polymer
to a fully continuous one.
In that case, one defines
\begin{align*}
	Z(t,x)=
	\frac{1}{\sqrt{2\pi t}}e^{-\frac{x^{2}}{2t}}
	\,\E_{_{\text{Brownian bridge}\hfill\atop
	b\colon(0,t)\to\R,\;
	b(0)=0,\; b(t)=x}}
	:\exp:\left\{
	\int_{0}^{t}\xi(s,b(s))ds
	\right\},
\end{align*}
where $\xi$ is the space-time white noise, and $:\exp:$
means normally ordered exponential,
e.g., see \cite{AlbertsKhaninQuastel2012} for an explanation.
Equivalently (via the Feynman-Kac formula),
$Z$ solves 
the \emph{stochastic heat equation with multiplicative noise}
\begin{align*}
	\frac{\partial}{\partial t}Z(t,x)=
	\frac{1}{2}\frac{\partial^{2}}{\partial x^{2}}
	Z(t,x)+
	\xi(t,x)Z(t,x),\qquad
	Z(0,x)=\delta(x).
\end{align*}
Then, defining
\begin{align*}
	u(t;x_1,\ldots,x_k)=\E_{\text{white noise}}
	\left(
	Z(t,x_1)\ldots Z(t,x_k)
	\right),
\end{align*}
we have
\begin{align}\label{SHE_u_eqn}
	\left(
	\frac{\partial}{\partial t}-\frac{1}{2}\sum_{i=1}^{k}\frac{\partial^2}{\partial x_i^{2}}
	\right)u=0
\end{align}
away from diagonal subset, and
\begin{align}\label{SHE_u_bc}
	\left(
	\frac{\partial}{\partial x_i}-\frac{\partial}{\partial x_{i+1}}
	-1
	\right)u=0
	\qquad\mbox{if $x_i=x_{i+1}$}.
\end{align}

The observations
\eqref{SHE_u_eqn}--\eqref{SHE_u_bc} are not 
hard and were recorded
at least as far back as 
the end of 1980's
by Kardar
\cite{Kardar1987}
and Molchanov
\cite{Molchanov1991}.
We believe that a rigorous proof can be extracted from the results 
of \cite{BertiniCancrini1995}.
In particular, case $N=2$ was treated in \cite[I.3.2]{Albeverio2005}.
To the best of our knowledge,
the general case has not been
worked out in full detail yet.

A concise way to write
\eqref{SHE_u_eqn}--\eqref{SHE_u_bc}
is 
\begin{align*}
	\frac{\partial}{\partial t}u=Hu,\qquad
	H=\frac{1}{2}\left(
	\sum_{i=1}^{k}\frac{\partial^{2}}{\partial x_i^{2}}
	+\sum_{i\ne j}
	\delta(x_i-x_j)
	\right).
\end{align*}
$H$ is called the \emph{delta-Bose gas} 
(or \emph{Lieb--Liniger}) \emph{Hamiltonian}.
The nested contour
integral then takes the form
\begin{align}\label{SHE_moments}
	u(t;x_1,\ldots,x_k)=
	\int\limits_{\al_1-\i\infty}^{\al_1+\i\infty}
	dz_1
	\ldots
	\int\limits_{\al_k-\i\infty}^{\al_k+\i\infty}
	dz_k
	\prod_{1\le A<B\le k}
	\frac{z_A-z_B}{z_A-z_B-1}
	\prod_{j=1}^{k}
	e^{\frac{t}{2}z_j^{2}+x_jz_j},	
\end{align}
where $\al_1>\al_2+1>\ldots>\al_k+(k-1)$, 
and $x_1\le \ldots\le x_k$
(the above formula is not true without the order assumption 
on the $x_j$'s).
For more detail, e.g., see
\cite[\S6.2]{BorodinCorwin2011Macdonald} and references therein.


\subsection{Intermittency} 
\label{sub:intermittency}

By setting $N_1=\ldots=N_k=N$ in \eqref{polymer_moments}, 
we see that the nested contour integrals 
provide us with all moments of the polymer
partition function $e^{-T_{N,N}}$.
One might expect that this
is sufficient to 
find 
its distribution or, equivalently,
the distribution of the free energy $T_{N,N}$.
It turns out that in this 
particular situation this is not true.
The distribution of the polymer partition function
displays \emph{intermittency}, which we now discuss.

This term appeared in studying 
the velocity and temperature fields in a turbulent medium
\cite{BatchelorTownsend1949},
and describes structures that appear in random media having the 
form of peaks that arise at random
places and at random time moments.
The phenomenon is widely
discussed in physics literature, with magnetic hydrodynamics
(like on the surface of the Sun) and cosmology
(theory of creation of galaxies) being two well-known examples,
e.g., see
\cite{ZeldovichMolchanov},
\cite{Molchanov1991}.

The main property that allows one 
to detect an intermittent distribution
is anomalous behavior
(as compared to the Gaussian case, for example)
of ratios of successive moments. 

\medskip
\noindent\textbf{Toy example.}
Consider a sequence of independent identically distributed
random variables $\xi_1,\xi_2,\ldots$,
each taking value $0$ or $2$ with probability $1/2$.
Set $\xi=\xi_1 \ldots\xi_N$.
Then, clearly,
\begin{align*}
	\E \xi=1,\qquad
	\E(\xi^{2})=2^{N},
	\qquad \ldots,\qquad
	\E(\xi^{p})=2^{(p-1)N}.
\end{align*}
The growth speed can be measured by the quantities
\begin{align*}
	\gamma_p=\frac{\log\E(\xi^{p})}{N}=(p-1)\ln2.
\end{align*}
On the other hand, computing a
similar quantity for the sum of 
the $\xi_j$'s (which is asymptotically Gaussian)
gives\footnote{We divide by $\log N$ because the random variable
in question grows roughly linearly in $N$, and $\xi$ above 
has exponential growth.}
\begin{align*}
	\gamma_p=\lim_{N\to\infty}
	\frac{\log\E\big((\xi_1+\ldots+\xi_N)^{p}\big)}{\log N}=p.
\end{align*}
The key difference of these two cases
is that $\frac{\gamma_p}{p}<\frac{\gamma_{p+1}}{p+1}$
in the first case, while
$\frac{\gamma_p}{p}=\frac{\gamma_{p+1}}{p+1}$
in the second one.

\medskip

In general, imagine that 
one has a time-dependent
nonnegative
random variable $Z(t)$
which grows in $t$
roughly exponentially 
(or $\log Z$ grows roughly linearly).
There are many ways to 
measure such growth; 
we mostly follow \cite{CarmonaMolchanov} in the 
exposition below. Define:
\begin{enumerate}[$\bullet$]
	\item \emph{Almost sure Lyapunov exponent}
	\begin{align*}
		\tilde\gamma_1:=\lim_{t\to\infty}\frac{\ln Z(t)}{t},
	\end{align*}
	if the a.s. limit exists (the ``a.s.''
	requirement can be weakened to 
	convergence in probability).
	\item \emph{Moment} (or \emph{annealed})
	\emph{Lyapunov exponents}
	\begin{align*}
		\gamma_p:=\lim_{t\to\infty}
		\frac{\ln\E \big(Z(t)\big)^{p}}{t}
	\end{align*}
	(assuming that limits exist).
\end{enumerate}
H\"older's inequality implies that 
\begin{align*}
	\frac{\gamma_p}{p}\le \frac{\gamma_{p+1}}{p+1}
\end{align*}
(because
$(\E Z^{p})^{\frac1p}\le (\E Z^{p+1})^{\frac1{p+1}}$).
The strict inequalities will be referred to as \emph{intermittency}.
Note that we also obviously have
$\tilde \gamma_1\le \gamma_1$.
In a typical situation, when the distribution
of $Z(t)$ does not deviate much from its mean,
\begin{align*}
	\tilde\gamma_1=\gamma_1=\frac{\gamma_2}{2}=\frac{\gamma_3}{3}
	=\ldots=\frac{\gamma_p}{p}=\ldots.
\end{align*}

\begin{lemma}
	If there exists $k\ge1$ such that
	\begin{align*}
		\frac{\gamma_k}{k}<\frac{\gamma_{k+1}}{k+1},
	\end{align*}
	then for all $p\ge k$,
	\begin{align*}
		\frac{\gamma_p}{p}<\frac{\gamma_{p+1}}{p+1}.
	\end{align*}
\end{lemma}
\begin{proof}
	H\"older's inequality with $\frac{1}{2}+\frac{1}{2}=1$
	gives
	\begin{align*}
		(\E Z^{k})^{2}\le \E Z^{k+h}\cdot \E Z^{k-h},\qquad h=1,2,\ldots,k,
	\end{align*}
	which implies that 
	$\gamma_k\le \dfrac{\gamma_{k+h}+\gamma_{k-h}}{2}$.
	Replacing $k$ by $k+1$ and taking $h=1$, we have
	\begin{align*}
		\gamma_{k+1}\le \frac{\gamma_{k+2}+\gamma_k}{2}<
		\frac12\left(\gamma_{k+2}+\frac{k}{k+1}\gamma_{k+1}\right)
	\end{align*}
	(we used the hypothesis of the lemma). Rearranging terms gives the 
	needed inequality for $p=k+1$.
	Repeating inductively, we obtain the desired claim.
\end{proof}

Let us now show how the definition of intermittency
relates to peaks. 
If we pick $\al$ such that 
$\dfrac{\gamma_{p}}{p}<\al<\dfrac{\gamma_{p+1}}{p+1}$,
then for large enough $t$ (below we omit $t$ in the 
notation for $Z(t)$):
\begin{enumerate}[$\bullet$]
	\item $\Prob\{Z>e^{\al t}\}>0$, because otherwise we would have $\left(\E Z^{p+1}\right)^{\frac1{p+1}}\le \al$.
	\item An overwhelming contribution to $\E Z^{p+1}$
	comes from the region where $Z>e^{\al t}$. Indeed, 
	\begin{align*}
		\E Z^{p+1}=\E\left(Z^{p+1}\mathbf{1}_{Z\le e^{\al t}}\right)
		+\E\left(Z^{p+1}\mathbf{1}_{Z>e^{\al t}}\right).
	\end{align*}
	The first term is $\le e^{\al(p+1)t}\ll e^{\gamma_{p+1}t}$, 
 	and we know that the left-hand side behaves exactly as $e^{\gamma_{p+1}t}$.
 	\item $\Prob\{Z>e^{\al t}\}\le e^{-(\al-\gamma_p/p)pt}$
 	because $\E Z^{p}\ge e^{\al pt}\Prob\{Z>e^{\al t}\}$.
\end{enumerate}
Hence, we observe a hierarchy of higher and higher peaks concentrated on smaller and smaller sets
(that are actually exponentially small in probability), and higher
peaks contribute overwhelmingly 
to high enough moments.
In the situation of random fields when ergodicity allows to replace
computing expectations by space averaging, at each fixed large time 
one can then observe a hierarchy of islands with exponentially
(in time)
high values that dominate moment computations.

Intermittency is a characteristic feature 
of products of a large number of independent
random variables (cf. the toy example above).
Indeed, by the central limit theorem, 
let us check that random variables 
of the form
$e^{\xi_1+\ldots+\xi_t}\sim e^{\mathcal{N}(\mu t,\sigma^{2} t)}$
(for example, with independent identically distributed $\xi_j$'s)
are intermittent. We have
\begin{align*}
	\E\left(\big[e^{\mathcal{N}(\mu t,\sigma^{2}t)}\big]^{p}\right)=
	e^{t\left(p\mu+\frac{p^{2}}2\sigma^{2}\right)},
\end{align*}
which implies that 
\begin{align*}
	\tilde\gamma_1=\mu<\gamma_1=\mu+\frac{1}{2}\sigma^{2}<
	\frac{\gamma_2}{2}=\mu+\sigma^{2}
	< \frac{\gamma_3}{3}=\mu+\frac32\sigma^{2}
	< \ldots.
\end{align*}


\subsection{Moment problem and intermittency} 
\label{sub:under_intermittency_moments_do_not_define_the_distribution}

Since under intermittency the moments are dominated 
by increasingly atypical behavior (i.e., observed with 
small probability), it is hard to 
expect that the moments would determine the distribution.
For example, for the exponential of the standard Gaussian
$\mathcal{N}(0,1)$ they do not:
Any distribution
with density
\begin{align*}
	f(x)=
	\begin{cases}
		\frac{1}{\sqrt{2\pi}}
		\frac{1}{x}\exp\left(
		-\frac{(\ln x)^{2}}{2}
		\right)
		\cdot (1+\eps h(x))
		,&x >0;\\
		0,&x \le0,
	\end{cases}
\end{align*}
with $h(x):=\sin(2\pi \ln(x))$, $-1\le \eps\le 1$, gives the same moments.
When $\eps=0$, this is the density of the log-normal random	variable $e^{\mathcal{N}(0,1)}$.
See
\cite{Stoyanov2004}
for more detail.

We will now check if the polymer partition function is intermittent.
\begin{theorem}\label{thm:polymer_intermittency}
	The moment Lyapunov exponents for the Brownian polymer
	are given by 
	\begin{enumerate}[$\bullet$]
		\item In the semi-discrete case (\S \ref{sub:polymer_limit} 
		and \S \ref{sub:moments_of_semi_discrete_brownian_polymer}), $\gamma_p=H_p(z_{0,p})$, 
		where (for $N=t$)
		\begin{align*}
			H_p(z)=\frac{p^{2}}{2}+pz-\log(z(z+1)\ldots (z+p-1)),
		\end{align*}
		and $z_{0,p}$ is the unique solution to $H'_p(z)=0$ on $(0,+\infty)$.
		\item In the fully continuous case (\S \ref{sub:continuous_brownian_polymer}),
		for paths between $(0,0)$ and $(t,0)$, 
		\begin{align*}
			\gamma_p=\frac{p^{3}-p}{24}.
		\end{align*}
	\end{enumerate}
\end{theorem}
\begin{proof}
	Steepest descent for contour integral representations
	of moments given before (\S \ref{sub:moments_of_semi_discrete_brownian_polymer}
	and \S \ref{sub:continuous_brownian_polymer}). 
	See \cite{BorodinCorwin2012Anderson} for details.
\end{proof}

In the fully continuous setting, the above result is due to 
Kardar \cite{Kardar1987} (non-rigorously),
and Bertini--Cancrini \cite{BertiniCancrini1995}.


\subsection{Replica trick} 
\label{sub:replica_trick}

Given our previous discussion on peak domination in moments,
it seems hopeless that this limit behavior of moments
would carry any information
about the behavior of the main bulk of the distribution.
This is, however, a suitable moment to demonstrate the 
(in)famous \emph{replica trick}
widely used in physics.

Note that, at least formally,
\begin{align*}
	\log Z(t)=\lim_{p\to0}\frac{(Z(t))^{p}-1}{p}.
\end{align*}
Averaging both sides, dividing by $t$ and taking $t\to\infty$
suggests
\begin{align*}
	\tilde \gamma_1=\lim_{p\to0}\lim_{t\to\infty}
	\frac{1}{t}\frac{e^{t\gamma_p}-1}{p}=
	\lim_{p\to0}\frac{\gamma_p}{p},
\end{align*}
where we tacitly use a (non-unique!)
analytic continuation of $\gamma_p$
off nonnegative integers.
For the semi-discrete polymer, $H_p(z)=\frac{p^{2}}{2}+pz-\log \frac{\Gamma(z+p)}{\Gamma(z)}$, 
and 
\begin{align*}
	\lim_{p\to0}\frac{H_p(z)}{p}=z-\psi(z),\qquad
	\psi(z)=(\log \Gamma(z))'.
\end{align*}
Taking the value of $z-\psi(z)$
at the only critical point of this function
on $(0,+\infty)$, gives
\begin{align*}
	\tilde \gamma_1=\inf_{z\in(0,+\infty)}(z-\psi(z)).
\end{align*}
This actually \emph{is} the correct answer! 
It was conjectured in 
\cite{OConnellYor2001} and proven in 
\cite{MoriartyOConnell}.

Similarly, in the fully continuous case, 
\begin{align*}
	\tilde \gamma_1=\lim_{p\to0}\frac{1}{p}\frac{p^{3}-p}{24}=-\frac{1}{24},
\end{align*}
which is also correct
\cite{AmirCorwinQuastel2011},
\cite{SasamotoSpohn2010}.

We thus see that this very nonrigorous procedure,
quite remarkably, lead us to the correct almost sure behavior!
In the next section we show how to access
these results rigorously, and our 
approach will also explain 
in a way why the replica trick
worked in this particular situation.



\section{Laplace transforms} 
\label{sec:laplace_transforms}

\subsection{Setup} 
\label{sub:setup'}

As we have seen in \S \ref{sec:moments_for_q_whittaker_processes},
the intermittency phenomenon prevents us from recovering 
the distribution of the polymer partition function
from its moments. However, this is not so in the
$q$-setting. 
Namely, the $q$-moments $\E q^{k\la^{(N)}_N}$, 
$k=1,2,\ldots$, uniquely determine
the distribution of $\la^{(N)}_{N}$
(because $\la^{(N)}_{N}\ge0$ and $q\in(0,1)$, so these are moments
of a bounded random variable). 

Our plan is thus to convert
the $q$-moment formulas that we have 
(Proposition \ref{prop:moments_of_the_qtasep}) into 
a formula for the expectation of a one-parameter family of observables
that remain bounded (unlike the moments $\E q^{k\la^{(N)}_N}$)
in the $q\to1$ which leads to polymers.
Since this will involve $q$-moments with $k\to\infty$, 
it is inconvenient 
to use nested contours in integral representations 
as their positions depend on $k$.
There are two ways to ``un-nest''
the contours: (1) to deform all of them 
to identical large concentric circles $|z|=R>1$;
or (2) to deform all of 
them to identical small concentric circles $|z-1|=r<\eps$.
The first way is easier to realize, but it is harder to 
turn the result into a meaningful asymptotic information.
Thus, we proceed with the second one. 
The following lemma is nontrivial and very useful:
\begin{lemma}\label{lemma:unnesting}
	Let $f$ be a meromorphic function and $\mathbb{A}$ be its singular set
	which must not include~$0$. Assume that $q^{m}\mathbb{A}$ 
	is disjoint from $\mathbb{A}$ for all integers $m\ge1$.
	Then
	\begin{align}
		\nonumber
		\mu_k&:=\frac{(-1)^{k}q^{\frac{k(k-1)}2}}{(2\pi\i)^{k}}
		\oint \ldots\oint 
		\prod_{1\le A<B\le k}\frac{z_A-z_B}{z_A-qz_B}
		\frac{f(z_1)\ldots f(z_k)}{z_1 \ldots z_k}
		dz_1 \ldots dz_k
		\\&\phantom{:}=
		\nonumber
		k!_q
		\sum_{\la=(\la_1\ge\la_2\ge \ldots\ge\la_\ell>0)\atop{
		\la_1+\ldots+\la_\ell=k\atop
		\la=1^{m_1}2^{m_2}\ldots}} \frac{1}{m_1! m_2! \ldots}
		\frac{(1-q)^{k}}{(2\pi\i)^{\ell}}\oint \ldots\oint
		\det\left[\frac{1}{w_iq^{\la_i}-w_j}\right]_{i,j=1}^{\ell}\times\\&
		\hspace{120pt}
		\times
		\prod_{j=1}^{\ell}f(w_j)f(qw_j) \ldots f(q^{\la_j-1}w_j)dw_j,
		\label{unnesting}
	\end{align}
	where each $z_p$ contour 
	contains $\{qz_j\}_{j>p}$ and $\mathbb{A}$, but not $0$
	(thus, the $z$ contours are nested), and 
	the $w_j$ contours contain $\mathbb{A}$ and no other poles
	(so, the $w$ contours can be taken to be all the same).
\end{lemma}
Note that for $f(z)=\dfrac{e^{(q-1)tz}}{(1-z)^{N}}$, 
\eqref{unnesting} becomes the formula for $\E q^{k\la^{(N)}_N}$,
see Proposition \ref{prop:moments_of_the_qtasep}.
In this case, $\mathbb{A}=\{1\}$.
\begin{proof}
	This lemma is the result of mere bookkeeping 
	of the residues when we shrink the contours
	and take into account the poles at $z_A=qz_B$.
	The fact that the result is rather nice\footnote{The number
	of residues involved is much larger than the number of terms in \eqref{unnesting}.} 
	is nontrivial,
	and takes origin in harmonic analysis on 
	Riemannian symmetric spaces and Hecke algebras,
	cf. \cite{HeckmannOpdam1997}.
	A proof of the lemma can be found in 
	\cite[Prop. 3.2.1]{BorodinCorwin2011Macdonald},
	see also \cite[Lemma 3.3 and Prop. 7.4]{BorodinCorwinPetrovSasamoto2013}.
\end{proof}
A limiting case of this lemma, as $q=e^{-\eps}\to1$, 
$z_j=1+\eps \tilde z_j$, $w_j=1+\eps \tilde w_j$,
is at the heart of the moments asymptotics
which were stated in \S \ref{sub:under_intermittency_moments_do_not_define_the_distribution}.


\subsection{Generating functions} 
\label{sub:generating_functions}

The form of the right-hand side of \eqref{unnesting}
suggests that one could take a generating function 
of such expressions over different $k$.
More exactly, it easily implies that
\begin{align}
	\begin{array}{ll}\displaystyle
		\sum_{k\ge0}\mu_k
		\frac{\zeta^{k}}{k!_q}&\displaystyle=
		\sum_{\ell\ge0}\frac{1}{\ell!}
		\sum_{n_1,\ldots,n_\ell\ge1}
		\frac{1}{(2\pi\i)^{\ell}}
		\oint \ldots\oint
		\det\left[\frac{1}{q^{n_i}w_i-w_j}\right]_{i,j=1}^{\ell}
		\times\\&\displaystyle\hspace{50pt}\times\prod_{j=1}^{\ell}
		(1-q)^{n_j}\zeta^{n_j}f(w_j)\ldots
		f(q^{n_j-1}w_j)dw_j.
	\end{array}
	\label{getting_qLaplace}
\end{align}
We will now take $f(z)=\dfrac{e^{(q-1)tz}}{(1-z)^{N}}$,
with all the integration contours above being small enough
positively oriented contours around 1. 
Here $(n_1,\ldots,n_{\ell})$
are simply permuted values of $(\la_1\ge \ldots\ge \la_\ell)$ in 
\eqref{unnesting}, 
and the change of the combinatorial
factor from $\dfrac{1}{m_1!m_2! \ldots}$ to $\dfrac{1}{\ell!}$
is due to that un-ordering.\footnote{Indeed, 
$\dfrac{\ell!}{m_1!m_2! \ldots}$
is the number of different ways to obtain a given $\la=1^{m_1}2^{m_2}\ldots$
with $|\la|=\ell$ from $(n_1,\ldots,n_\ell)\in\Z^{\ell}_{\ge1}$.}

Now, using the $q$-exponential identity
(e.g., see 
\cite{AndrewsAskeyRoy2000}, \cite{GasperRahman})
\begin{align*}
	\sum_{k\ge0}\frac{a^{k}}{k!_q}
	=\prod_{m\ge0}\frac{1}{1-(1-q)aq^{m}}
	=:\frac{1}{\big((1-q)a;q\big)_{\infty}},
\end{align*}
we can rewrite the left-hand side of 
\eqref{getting_qLaplace} as 
\begin{align*}
	\E \frac{1}{\big((1-q)q^{\la^{(N)}_N}\zeta;q\big)_{\infty}}.
\end{align*}
One should expect that in a suitable scaling limit
as $q\to1$
(which we can predict by looking at the moment
asymptotics), 
the $q$-moment generating function
would converge to 
the Laplace transform 
of the polymer partition function.
The latter \emph{does}
define the distribution uniquely, 
with or without intermittency.
The real question now is how to 
take a similar limit in the 
right-hand side of \eqref{getting_qLaplace}.
Observe that termwise limit would produce a 
moment generating series, 
and we already know that it diverges!


\subsection{Case $N=1$ and the Mellin--Barnes integral representation} 
\label{sub:case_n_1_and_the_mellin_barnes_integral_representation}

Let us consider the case $N=1$ in which the problem of convergence
is already there. 
Then $\mu_k$'s are the $q$-moments
of the simple continuous-time
one-sided random walk started from $0$
at $t=0$. 
We expect their $q$-generating function
to converge (as $q\to 1$)
to the Laplace transform of the lognormal distribution
(i.e., $e^{\mathcal{N}(0,\tau)}$).
Indeed, we should expect that because
\begin{align*}
	\frac{1}{\big((1-q)q^{\la^{(1)}_1(t)} \zeta;q\big)_{\infty}}
	\stackrel{\varepsilon\to+0}{\xrightarrow{\hspace*{1cm}}} 
	e^{-ue^{-T_{1,1}(\tau)}},
	\quad
	\la^{(1)}_{1}\sim \frac{\tau}{\eps^{2}}
	+\frac{T_{1,1}(\tau)}{\eps},\quad
	\zeta=e^{\tau\eps^{-1}}u,
\end{align*}
where $T_{1,1}(\tau)\sim \mathcal{N}(0,\tau)$.

Observe that we only have first order poles at $w_j=1$
in the right-hand side of \eqref{getting_qLaplace}
for $N=1$. Because of vanishing of the 
$\det\left[\dfrac{1}{q^{n_i}w_i-w_j}\right]_{i,j=1}^{\ell}$
for equal values of the $w_j$'s, we conclude that 
only $\ell\le1$ give a nontrivial contribution. 
This contribution is
\begin{align}\nonumber
	1+&\sum_{n\ge1}\frac{1}{2\pi\i}
	\oint_{\Gamma_1}\frac{dw}{(q^{n}-1)w}(1-q)^{n}
	\zeta^{n}\frac{e^{(q^{n}-1)tw}}
	{(1-w)(1-qw)\ldots(1-q^{n-1}w)}\\&=
	1+\sum_{n\ge1}
	\frac{(1-q)^{n}\zeta^{n}}{1-q^{n}}
	\frac{e^{(q^{n}-1)t}}
	{(1-q)\ldots(1-q^{n-1})}
	\label{N1_MB_example_1}
	=
	\sum_{n\ge0}\frac{\big((1-q)\zeta\big)^{n}e^{(q^{n}-1)t}}{(1-q)\ldots(1-q^{n-1})}.
\end{align}

We now need to take the $q\to1$ limit in the above sum,
and we cannot do that termwise as this would result in a divergent
series.
A standard tool of the theory
of special functions used for
dealing with such a limit
is the \emph{Mellin--Barnes integral representation}
which dates back to the end of the 19th century.
In its simplest incarnation, it says that
\begin{align*}
	\sum_{n\ge0}g(q^{n})\zeta^{n}=\frac{1}{2\pi\i}
	\oint_{\Gamma_{0,1,2,\ldots}} 
	\Gamma(-s)\Gamma(1+s)(-\zeta)^{s}g(q^{s})
	ds,\qquad 
	|\arg(\zeta)|<\pi,
\end{align*}
where the integral in the right-hand side goes in the negative direction 
around the poles $s=0,1,2,\ldots$.\footnote{Note that
$\Gamma(-s)\Gamma(1+s)=-\dfrac{\pi}{\sin(\pi s)}$.} Indeed,
\begin{align*}
	-\Res_{s=n}\Gamma(-s)\Gamma(1+s)(-\zeta)^{s}g(q^{s})=
	g(q^{n})\zeta^{n},
\end{align*}
where we assume $z^{s}$ to be defined 
with the branch cut $(-\infty,0)$.
Omitting convergence and contour deformation
justifications (which can be performed),
we rewrite the series in \eqref{N1_MB_example_1}
as
\begin{align}&\label{N1_MB_example_2}
	\sum_{n\ge0}\frac{\big((1-q)\zeta\big)^{n}e^{(q^{n}-1)t}}{(1-q)\ldots(1-q^{n-1})}
	\\&\hspace{40pt}=\nonumber
	\frac{1}{2\pi\i}
	\int_{\delta-\i\infty}^{\delta+\i\infty}
	\Gamma(-s)\Gamma(1+s)\big((q-1)\zeta\big)^{s}
	e^{(q^{s}-1)t}
	\frac{\prod_{m\ge1}(1-q^{s+m})}{\prod_{m\ge1}(1-q^{m})}ds,
\end{align}
where 
$0<\delta<1$ and
the integration is taken over 
a contour as on Fig.~\ref{fig:mellin_barnes}.\begin{figure}[htbp]
\begin{center}
	\begin{tikzpicture}
		[scale=1, ultra thick]
		\def\x{.08}
		\def\r{.7}
		\draw[->,thick] (-1.4,0)--(1.1,0);
		\draw[thick] (-.2,-\x)--(-.2,\x) node [above, yshift=-2,xshift=-3] {$0$};
		\node[below, yshift=0,xshift=6] at (0,\x) {$\delta$};
		\draw[decoration={markings,
		      mark=at position .7 with {\arrow{latex}}}, 
		      postaction={decorate}] 
	    (0,-1-\r)--(0,-\r);
		\draw[decoration={markings,
		      mark=at position .7 with {\arrow{latex}}}, 
		      postaction={decorate}] 
		(0,\r)--(0,1+\r);
		\draw[ultra thick] (0,-\r) arc (-90:-270:\r);
		\draw[dashed, thick] (0,\r+.2)--(0,-\r-.2);
	\end{tikzpicture}
\end{center}
\caption{Integration contour in \eqref{N1_MB_example_2}.}
\label{fig:mellin_barnes}
\end{figure}
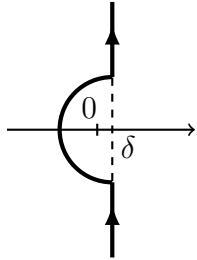
We can now take the needed limit. We note that
\begin{align*}
	\Gamma_{q}(x)=\prod_{m\ge1}
	\frac{1-q^{m}}{1-q^{x+m-1}}(1-q)^{1-x}
\end{align*}
is the $q$-analogue of the Euler $\Gamma$-function, 
and that (e.g., see \cite{AndrewsAskeyRoy2000})
\begin{align*}
	\lim_{q\to1}\Gamma_{q}(x)=\Gamma(x),
	\qquad x\notin\{0,-1,-2,\ldots\}.
\end{align*}
We take the scaling
\begin{align*}
	q=e^{-\eps}\to1,\qquad t=\frac{\tau}{\eps^{2}},\qquad
	\zeta=e^{\tau\eps^{-1}}u,
\end{align*}
so that
\begin{align*}
	\big((q-1)\zeta\big)^{s}=\big(
	(e^{-\eps}-1)e^{\tau\eps^{-1}(-u)}
	\big)^{s}\sim \eps^{s}e^{\tau s\eps^{-1}}u^s,
	\qquad 
	e^{(q^{s}-1)t}\sim e^{-\tau s\eps^{-1}+\frac{\tau s^{2}}2},
\end{align*}
and 
\begin{align*}
	\prod_{m\ge1}\frac{1-q^{s+m}}{1-q^{m}}
	=\frac{(1-q)^{-s}}{\Gamma_{q}(s+1)}
	\sim \eps^{-s}\frac{1}{\Gamma(s+1)}.
\end{align*}
We see that the limit of the integral 
in \eqref{N1_MB_example_2} is
\begin{align*}
	\frac{1}{2\pi\i}
	\int_{-\frac12-\i\infty}^{-\frac12+\i\infty}
	\Gamma(-s)\Gamma(1+s)\frac{u^{s}}{\Gamma(1+s)}e^{\tau s^{2}/2}ds
	=\frac{1}{2\pi\i}
	\int_{-\frac12-\i\infty}^{-\frac12+\i\infty}
	\Gamma(-s)u^{s}e^{\tau s^2/2}ds,
\end{align*}
which is a correct expression for the Laplace
transform of the lognormal random variable
$e^{\mathcal{N}(0,\tau)}$, as we expected.


\subsection{Asymptotics of the generating function for any $N\ge1$} 
\label{sub:asymptotics_of_the_generating_function_for_any_nge1_}

The same Mellin--Barnes integral representation
works for any $N\ge1$.
The summations over $n_1,\ldots,n_\ell\ge1$
are replaced by
integrals over $\frac{1}{2}+\i \R$
with $\Gamma(-s_j)\Gamma(1+s_j)$
inside, and using scaling
of \S \ref{sub:polymer_limit}
together
with Theorem \ref{thm:polymer_limit}
(which guarantees convergence of expectations of bounded functions),
as well as the asymptotic relations above (setting $v_j=q^{w_j}$), we obtain
the following generating function for the 
semi-discrete Brownian polymer's partition function
(\S \ref{sub:polymer_limit}):
\begin{theorem}\label{thm:Laplace_polymer}
	Fix $N\ge1$, $0<\delta_2<1$, and $0<\delta_1<\delta_2/2$.
	Then\footnote{Note that the time parameter $t$ in \eqref{Laplace_polymer}
	was denoted by $\tau$ in \S \ref{sub:polymer_limit} and 
	\S \ref{sub:case_n_1_and_the_mellin_barnes_integral_representation}.}
	\begin{align}
		\nonumber
		\E e^{-ue^{-T_{N,N}}(t)}
		&=1+\sum_{\ell\ge1}
		\frac{1}{\ell!}
		\oint\limits_{|v_1|=\delta_1}dv_1 \ldots\oint\limits_{|v_\ell|=\delta_1}dv_\ell
		\int\limits_{\delta_2-\i\infty}^{\delta_2+\i\infty}
		ds_1\ldots
		\int\limits_{\delta_2-\i\infty}^{\delta_2+\i\infty}
		ds_\ell\\&\hspace{70pt}\times
		\prod_{j=1}^{\ell}
		\Gamma(-s_j)\Gamma(1+s_j)
		\left(
		\frac{\Gamma(v_j)}{\Gamma(s_j+v_j)}
		\right)^{N}\times
		\nonumber
		\\\label{Laplace_polymer}
		&\hspace{70pt}\times
		\frac{u^{s_j+v_j}}{u^{v_j}}
		\frac{e^{\frac{t}{2}(s_j+v_j)^{2}}}
		{e^{\frac t2 v_j^{2}}}
		\det\left[
		\frac{1}{v_i+s_i-v_j}\right]_{i,j=1}^{\ell}.
	\end{align}
\end{theorem}
The expression in the right-hand side above is 
actually well-suited for further asymptotic analysis.
Let us first state the final result:
\begin{theorem}[{\cite{BorodinCorwin2011Macdonald}, \cite{BorodinCorwinFerrari2012}}]
\label{thm:TW_convergence}
	For any $\varkappa>0$, define
	\begin{align*}
		f_\varkappa=\min_{s>0}(\varkappa s-\psi(s)),\qquad
		s_{\varkappa}=\argmin_{s>0}(\varkappa s-\psi(s)),
		\qquad
		g_{\varkappa}=-\psi''(s_{\varkappa})>0
	\end{align*}
	(as before, $\psi(z)=\big(\log \Gamma(z)\big)'$).
	Then for $t=\varkappa N$, we have
	\begin{align*}
		\lim_{N\to\infty}
		\Prob\left\{
		\frac{-T_{N,N}(t)-N f_{\varkappa}}{N^{1/3}}\le r
		\right\}=F_{GUE}\left(\left(\frac{g_{\varkappa}}{2}\right)^{-\frac13}r\right),
	\end{align*}
	where $F_{GUE}$ is the GUE Tracy--Widom distribution.
\end{theorem}
Recall that $-T_{N,N}(t)$ can be
identified with the logarithm
of the polymer partition function 
as in \eqref{polymer_integral} (and that
$-T_{N,N}(t)\stackrel{d}{=}T_{N,1}(t)$).

Note that Theorem \ref{thm:TW_convergence}
proves the value of the almost
sure Lyapunov exponent
$\tilde \gamma_1$
that we guessed (for $\varkappa=1$,
but this could have been for any $\varkappa$)
using replica trick in \S \ref{sub:replica_trick}.

The Tracy--Widom distribution in the right-hand side of
\eqref{Laplace_polymer} arises as the series
\begin{align}
	\nonumber
	F_{GUE}\left(\left(\frac{g_{\varkappa}}{2}\right)^{-\frac13}r\right)
	&=1+\sum_{\ell\ge1}
	\frac{1}{\ell!}
	\frac{1}{(2\pi\i)^{2\ell}}
	\int \ldots\int da_1 \ldots da_\ell
	\int \ldots\int db_1 \ldots db_\ell
	\\&\hspace{20pt}\times
	\prod_{j=1}^{\ell}
	\frac{1}{a_j-b_j}
	\frac{\exp\left(-\frac{g_{\varkappa}}{6}a_j^{3}+ra_j\right)}
	{\exp\left(-\frac{g_{\varkappa}}{6}b_j^{3}+rb_j\right)}
	\det\left[\frac{1}{b_i-a_j}\right]_{i,j=1}^{\ell},
	\label{TW}
\end{align}
where the $a_i$ and the $b_j$ contours are 
as on Fig.~\ref{fig:Airy_cont}.\begin{figure}[htbp]
	\begin{tabular}{cc}
		\begin{tikzpicture}
			[scale=1, ultra thick]
			\def\x{.1}
			\draw[->,thick] (-1.5,0)--(.7,0);
			\draw[thick] (0,-\x)--(0,\x) node [below, yshift=-4,xshift=0] {$0$};
			\draw[decoration={markings,
			      mark=at position .7 with {\arrow{latex}}}, 
			      postaction={decorate}] 
		    (-1,-.866)--(-.5,0) node [above, xshift=5] {$a_i$};
			\draw[decoration={markings,
			      mark=at position .7 with {\arrow{latex}}}, 
			      postaction={decorate}] 
			(-.5,0)--(-1,.866);
			\draw[thick] (-.75,0) arc (-180:-120:.25);
			\node at (-1,-.3) {$\frac\pi3$};
		\end{tikzpicture}
		&\hspace{40pt}
		\begin{tikzpicture}
			[scale=1, ultra thick]
			\def\x{.1}
			\draw[->,thick] (-.6,0)--(1.5,0);
			\draw[thick] (0,-\x)--(0,\x) node [below, yshift=-4,xshift=0] {$0$};
			\draw[decoration={markings,
			      mark=at position .7 with {\arrow{latex}}}, 
			      postaction={decorate}] 
		    (1,-.866)--(.5,0) node [above, xshift=-5] {$b_j$};
			\draw[decoration={markings,
			      mark=at position .7 with {\arrow{latex}}}, 
			      postaction={decorate}] 
			(.5,0)--(1,.866);
			\draw[thick] (.75,0) arc (0:-60:.25);
			\node at (1,-.3) {$\frac\pi3$};
		\end{tikzpicture}
	\end{tabular}
	\caption{The integration contours 
	in \eqref{TW}
	for variables
	$a_i$ (left) and $b_j$ (right).}
	\label{fig:Airy_cont}
\end{figure}
The identification of \eqref{TW} with a traditional formula
for $F_{GUE}$ is explained in 
\cite{BorodinCorwin2011Macdonald} (after formula (4.51)).

The way one reaches \eqref{TW} from 
the right-hand side of
\eqref{Laplace_polymer} is fairly straightforward.
By changing the variables $s_j\to y_j=s_j+v_j$,
one rewrites the part of the integrand 
that depends on the large parameter 
$N$ as
\begin{align*}
	\prod_{j=1}^{\ell}\exp\Big(
	N\big(G(v_j)-G(y_j)\big)
	\Big),\qquad
	G(z)=\ln \Gamma(z)-z \frac{\ln u}{N}-\frac{\varkappa}{2}z^{2}.
\end{align*}
Since
\begin{align*}
	e^{-ue^{-T_{N,N}(t)}}=e^{-e^{-T_{N,N}(t)+\log u}},
\end{align*}
we take $\log u\sim -N f_{\varkappa}-rN^{\frac13}$, 
and then we see that
\begin{align*}
	G(z)\sim \ln \Gamma(z)+f_{\varkappa}z- \frac{\varkappa}{2}z^{2}.
\end{align*}
The analysis then follows the scheme explained
in \S \ref{sec:asymptotics}, with $v$
contours 
being deformed to the domain 
with $\Re G(v)<0$, 
and $y$ contours --- to the domain with $\Re G(y)>0$.
The limiting expression
arises in the situation when $G(z)$
has a double critical point
$G'(z_c)=G''(z_c)=0$, 
and through a local 
change of integration variables near the critical point;
the constant $-g_{\varkappa}$ is actually $G'''(z_c)$.
Details can be found in
\cite{BorodinCorwin2011Macdonald} and \cite{BorodinCorwinFerrari2012}.

Let us conclude by observing that if we expand the right-hand side of \eqref{Laplace_polymer}
into residues at $s_j=1,2,\ldots$, we get
back the divergent generating series
for the moments of the polymer partition function that we found before.
This shows that a more sophisticated replica
trick than the one from \S \ref{sub:replica_trick} 
can actually be used to obtain the
limiting distribution, and
not only the law of large numbers
(i.e., $\tilde \gamma_1$).
Namely, one can obtain the moments
by solving the equations (the delta Bose gas of \S \ref{sub:continuous_brownian_polymer})
that they satisfy, write down the series for the
Laplace transform through moments
(despite the fact that this series diverges),
make sense of this series via the Mellin--Barnes
integral representation, and then proceed
with the asymptotic analysis.
This approach was successfully carried out
in physics papers 
\cite{Dotsenko},
\cite{Calabrese_LeDoussal_Rosso}.
However, the only plausible explanation
we have at the moment 
as to why such an approach leads to the correct
answer, is that it is a limiting case of the 
$q$-deformed situation, where all the steps are legal and indeed 
lead to a proof of the GUE edge fluctuations.



\end{document}